\documentclass[12pt,reqno]{amsart} 
\usepackage{amssymb} 
\usepackage{enumerate}
\usepackage{mathrsfs}
\usepackage[all]{xy}
\usepackage{bm}
\usepackage{dsfont}
\usepackage[colorlinks=true, citecolor=red, linkcolor=blue]{hyperref}
\usepackage{rotating}

\usepackage[usenames]{color}

\makeatletter
\@namedef{subjclassname@2020}{\textup{2020} Mathematics Subject 
Classification} 
\makeatother

\let\Gamma=\varGamma
\let\Omega=\varOmega
\let\Sigma=\varSigma

\definecolor{darkgreen}{rgb}{0,0.5,0}
\definecolor{bluegreen}{rgb}{0,0.2,0.8}
\definecolor{darkred}{rgb}{0.8,0,0}
\definecolor{newercolor}{rgb}{0.2,0,1}
\definecolor{darkyellow}{rgb}{0.7,0.7,0}
\definecolor{orange}{rgb}{0.9,0.4,0}

\newcommand{\prep}[1]{}

\newlength{\short}
\newlength{\shorter}
\newcommand{\EE}[2][]{\mathbf{E}_{#2}}

\newcommand{\trs}[1]{\mathfrak{tr}_{#1}}

\newcommand{\Grf}{\mathfrak{Gr}}
\newcommand{\grf}[1]{\mathfrak{gr}_{#1}}

\newcommand{\ttt}{\mbf{\tau}}

\newcommand{\twoMat}[2]{\text{\boldd{$\widehat{M}_{#1#2}$}}}
\newcommand{\Matnul}[2]{\text{\boldd{$\widehat{M}^0_{#1#2}$}}}
\newcommand{\MM}{M}
\newcommand{\MMnul}{M_0}

\newcommand{\NNnul}{\text{\boldd{$N_0$}}}
\newcommand{\NN}{\textbf{\textit{N}}}

\renewcommand{\SS}{\textbf{\textit{S}}}
\newcommand{\GG}{\text{\boldd{$\varGamma$}}}
\newcommand{\GGnul}{\text{\boldd{$\varGamma_0$}}}
\renewcommand{\AA}{\textbf{\textit{A}}}
\newcommand{\TT}{\textbf{\textit{T}}}
\newcommand{\QQ}{\textbf{\textit{Q}}}
\newcommand{\FF}{\boldsymbol{\mathcal{F}}}

\newcommand{\AAA}[1]{\text{\boldd{$\textit{A}^{(1#1)}$}}}
\newcommand{\aaa}[1]{\text{\boldd{$\textit{A}^{(#1)}$}}}
\newcommand{\nnn}[1]{\text{\boldd{$\textit{N}^{(#1)}$}}}
\newcommand{\NNN}[1]{\nnn{1#1}}

\newcommand{\recul}{\hskip-3pt}
\newcommand{\Recul}{\hskip-4pt}
\def\Qtrp[#1,#2,#3]{\textup{\texttt{<\recul<}}#1,#2,#3%
\textup{\texttt{>\recul>}}}
\def\trp[#1,#2,#3]{\textup{\texttt{[\Recul[}}#1,#2,#3\textup{\texttt{]\Recul]}}}
\newcommand{\dpar}[1]{\textup{\texttt{(\recul(}}#1\textup{\texttt{)\recul)}}}
\newcommand{\sbk}[1]{\textup{\texttt{[}}#1\textup{\texttt{]}}}

\newenvironment{plainlist}[1]{\begin{list}{} 
{\leftmargin=#1truemm \itemsep=6pt \itemindent=-5truemm \labelwidth=0truemm 
\parsep=2pt}} {\end{list}}

\newcommand{\bmid}{\mathrel{\big|}}

\newcommand{\sm}[1]{\textup{\begin{small}#1\end{small}}}

\newcommand{\halfup}[2][2.2]{\raisebox{#1ex}[0pt]{$#2$}}

\newcommand{\widebar}[1]{\overset{\mskip2mu\hrulefill\mskip2mu}{#1}
		\vphantom{#1}}

\newcommand{\3}[1]{\textup{\textbf{3\uppercase{#1}}}}
\newcommand{\4}[1]{\overline{#1}}   
\newcommand{\5}[1]{\widehat{#1}}

\newcommand{\9}[2][0]{{}^{#2}\hskip-#1pt}   

\newcommand{\too}{\longrightarrow}

\let\xto=\xrightarrow

\newcommand{\dbl}[2]{\renewcommand{\arraystretch}{1.0}%
\setlength{\tabcolsep}{0pt}%
\begin{tabular}{c}\rule{0pt}{12pt}$#1$\\$#2$\end{tabular}}
\newcommand{\tpl}[3]{\renewcommand{\arraystretch}{1.0}%
\setlength{\tabcolsep}{0pt}%
\begin{tabular}{c}\rule{0pt}{12pt}$#1$\\$#2$\\$#3$\end{tabular}}

\newcommand{\boldd}[1]{{\mathversion{bold}\textbf{#1}}}

\newcommand{\mbfx}[1]{{\boldmath #1\unboldmath}}
\newcommand{\mbf}[1]{\text{\boldmath $#1$\unboldmath}}

\newcommand{\lie}[3]{\def\test{#2}\def\tst{G}\ifx\test\tst{{}^{#1}#2_{#3}}
\else{{}^{#1}\!#2_{#3}}\fi}

\let\oldcirc=\circ
\renewcommand{\circ}{\mathchoice
    {\mathbin{\scriptstyle\oldcirc}}{\mathbin{\scriptstyle\oldcirc}}
    {\mathbin{\scriptscriptstyle\oldcirc}}
    {\mathbin{\scriptscriptstyle\oldcirc}}}

\newlength{\upto}\newlength{\dnto}
\SelectTips{cm}{10} \UseTips   

\numberwithin{equation}{section}

\def\beq#1\eeq{\begin{equation*}#1\end{equation*}}
\def\beqq#1\eeqq{\begin{equation}#1\end{equation}}

\let\emptyset=\varnothing
\renewcommand{\:}{\colon}   
\newcommand{\longline}{\bigskip\centerline{\hbox to 5cm{\hrulefill}}\bigskip}

\newcommand{\mxtwo}[4]{\left(\begin{smallmatrix}#1&#2\\#3&#4
\end{smallmatrix}\right)}
\newcommand{\mxthree}[9]{\left(\begin{smallmatrix}#1&#2&#3\\#4&#5&#6\\
#7&#8&#9\end{smallmatrix}\right)}

\newcommand{\Mxthree}[9]{\begin{pmatrix}#1&#2&#3\\#4&#5&#6\\
#7&#8&#9\end{pmatrix}}

\newcommand{\coltwo}[2]{\left(\begin{smallmatrix}#1\\#2
\end{smallmatrix}\right)}

\newcommand{\Colthree}[3]{\begin{pmatrix}#1\\#2\\#3\end{pmatrix}}

\newcommand{\mxfoura}[8]{\left(\begin{smallmatrix}#1&#2&#3&#4\\#5&#6&#7&#8\\}
\newcommand{\mxfourb}[8]{#1&#2&#3&#4\\#5&#6&#7&#8\end{smallmatrix}\right)}

\DeclareMathAlphabet\EuR{U}{eur}{m}{n}
\SetMathAlphabet\EuR{bold}{U}{eur}{b}{n}

\newcommand{\higherlim}[2]{\displaystyle\setbox1=\hbox{\rm lim}
	\setbox2=\hbox to \wd1{\leftarrowfill} \ht2=0pt \dp2=-1pt
	\setbox3=\hbox{$\scriptstyle{#1}$}
	\def\test{#1}\ifx\test\empty
	\mathop{\mathop{\vtop{\baselineskip=5pt\box1\box2}}}\nolimits^{#2}
	\else
	\ifdim\wd1<\wd3
	\mathop{\hphantom{^{#2}}\vtop{\baselineskip=5pt\box1\box2}^{#2}}_{#1}
	\else
	\mathop{\mathop{\vtop{\baselineskip=5pt\box1\box2}}_{#1}}%
	\nolimits^{#2}
	\fi\fi}
\newcommand{\higherlimm}[2]{\setbox1=\hbox{\rm lim}
	\setbox2=\hbox to \wd1{\leftarrowfill} \ht2=0pt \dp2=-1pt
	\mathop{\mathop{\vtop{\baselineskip=5pt\box1\box2}}}\limits_{#1}
	\nolimits^{#2}}

\newcounter{let} 
\loop\stepcounter{let}
\expandafter\edef\csname cal\alph{let}\endcsname%
{\noexpand\mathcal{\Alph{let}}}
\ifnum\thelet<26\repeat

\loop\stepcounter{let}
\expandafter\edef\csname scr\alph{let}\endcsname%
{\noexpand\mathscr{\Alph{let}}}
\ifnum\thelet<26\repeat

\loop\stepcounter{let}
\expandafter\edef\csname frak\alph{let}\endcsname%
{\noexpand\mathfrak{\Alph{let}}}
\ifnum\thelet<26\repeat

\newcommand{\tdef}[2][]{\expandafter\newcommand\csname#2\endcsname%
{#1\textup{#2}}}
\tdef{Iso}   \tdef{Aut}    \tdef{Out}    \tdef{Inn}    \tdef{Hom}   
\tdef{Ext}   \tdef{Irr}    \tdef{rad}    \tdef{Sym}    \tdef{ord}
\tdef{End}   \tdef{Inj}    \tdef{map}    \tdef{Ker}    \tdef{Ob}    
\tdef{Mor}   \tdef{Res}    \tdef{Spin}   \tdef{Id}     \tdef{wt}
\tdef{rk}    \tdef{conj}   \tdef{incl}   \tdef{proj}   \tdef{diag} 
\tdef{trf}   \tdef{Sol}     \tdef{cj}    \tdef{Outdiag} \tdef{lcm}
\tdef{free}  \tdef{supp}   \tdef{soc}    \tdef{Ind}    \tdef{Tr}
\tdef[_]{typ}   \tdef[^]{op}    \tdef[^]{ab}  \tdef{Ree}
\tdef{Mon}   \tdef{Perm}   \tdef{Herm}   
\tdef{Aff}

\newcommand{\fdef}[1]{\expandafter\newcommand\csname#1\endcsname%
{\mathfrak{#1}}}
\fdef{X}  \fdef{red}  \fdef{foc}  \fdef{hyp}  
\newcommand{\bb}{\mathfrak{b}}

\newcommand{\bbdef}[1]{\expandafter\newcommand%
\csname#1\endcsname{\mathbb{#1}}}
\bbdef{C} \bbdef{F} \bbdef{R} \bbdef{Z} \bbdef{Q}

\newcommand{\itdef}[1]{\expandafter\newcommand\csname#1\endcsname%
{\textit{#1}}}
\itdef{PSL}  \itdef{PSU}  \itdef{SL}  \itdef{SU}  \itdef{GL} \itdef{GU}
\itdef{UT}   \itdef{Sp}   \itdef{PSp} \itdef{PGL} \itdef{GO} \itdef{Co}
\itdef{SO}   \itdef{SD}   \itdef{Th}  \itdef{He}  \itdef{Sz} \itdef{AGL}
\itdef{Suz}  \itdef{McL}  \itdef{Ly}  \itdef{HS}  \itdef{Ru} \itdef{Fi}

\newcommand{\SP}{\Sp^*}

\newcommand{\GGL}{\textit{$\varGamma$L}}

\newcommand{\PGGL}{\textit{P$\varGamma$L}}

\newcommand{\gee}{\varepsilon}

\newcommand{\gen}[1]{\langle{#1}\rangle}
\newcommand{\Gen}[1]{\bigl\langle{#1}\bigr\rangle}

\let\nsg=\normal
\let\nnsg=\ntrianglelefteq

\newcommand{\syl}[2]{\textup{Syl}_{#1}(#2)}
\newcommand{\sylp}[1]{\syl{p}{#1}}
\newcommand{\autf}[1][]{\Aut_{\calf_{#1}}}
\newcommand{\outf}[1][]{\Out_{\calf_{#1}}}
\newcommand{\homf}[1][]{\Hom_{\calf_{#1}}}
\newcommand{\isof}[1][]{\Iso_{\calf_{#1}}}

\newcommand{\sminus}{\smallsetminus}
\newcommand{\defeq}{\overset{\textup{def}}{=}}

\renewcommand{\Im}{\textup{Im}}

\newcommand{\longleft}[1]{\;{\leftarrow%
\count255=0 \loop \mathrel{\mkern-6mu}%
    \relbar\advance\count255 by1\ifnum\count255<#1\repeat}\;}
\newcommand{\longright}[1]{\;{\count255=0 \loop \relbar\mathrel{\mkern-6mu}%
    \advance\count255 by1\ifnum\count255<#1\repeat\rightarrow}\;}
\newcommand{\Right}[2]{\overset{#2}{\longright#1}}
\newcommand{\RIGHT}[3]{\mathrel{\mathop{\kern0pt\longright#1}
	\limits^{#2}_{#3}}}

\newcommand{\LEFT}[3]{\mathrel{\mathop{\kern0pt\longleft#1}\limits^{#2}_{#3}}
}
\newcommand{\longleftright}[1]{\;{\leftarrow\mathrel{\mkern-6mu}%
    \count255=0\loop\relbar\mathrel{\mkern-6mu}%
    \advance\count255 by1\ifnum\count255<#1\repeat\rightarrow}\;} 
 
\newcommand{\onto}[1]{\;{\count255=0 \loop \relbar\joinrel
    \advance\count255 by1
    \ifnum\count255<#1 \repeat \twoheadrightarrow}\;}

\newcommand{\RLEFT}[3]{\mathrel{%
   \mathop{\vcenter{\baselineskip=0pt\hbox{$\kern0pt\longright#1$}%
   \hbox{$\kern0pt\longleft#1$}}}\limits^{#2}_{#3}}}

\newcommand{\newsubb}[2]{\mbfx{\smallskip\subsection{#1}\label{#2}%
\leavevmode\noindent}}
\numberwithin{table}{section}

\newenvironment{Table}[1][]{\stepcounter{equation}\begin{table}[#1]}
{\end{table}}

\renewenvironment{enumerate}[1][]
{\begin{enumerat}[#1]\setlength{\itemsep}{6pt}}{\end{enumerat}}

\renewenvironment{itemize}[1][-15]
{\begin{itemiz}\setlength{\itemsep}{6pt}\setlength{\itemindent}{#1pt}}
{\end{itemiz}}

\newenvironment{enuma}{\begin{enumerate}[{\rm(a) }]}{\end{enumerate}}
\newenvironment{enumi}{\begin{enumerate}[{\rm(i) }]}{\end{enumerate}}
\newenvironment{enum1}{\begin{enumerate}[{\rm(1) }]}{\end{enumerate}}

\newtheorem{Thm}[equation]{Theorem}
\newtheorem{Prop}[equation]{Proposition}
\newtheorem{Cor}[equation]{Corollary}
\newtheorem{Lem}[equation]{Lemma}
\newtheorem{Claim}[equation]{Claim}
\newtheorem{Ass}[equation]{Assumption}
\newtheorem{Conj}[equation]{Conjecture}
\newtheorem{Not}[equation]{Notation}
\newtheorem{Hyp}[equation]{Hypotheses}

\newtheorem{Thmm}{Theorem}

\theoremstyle{definition}
\newtheorem{Defi}[equation]{Definition} 
\newtheorem{Rmk}[equation]{Remark}
\newtheorem{Ex}[equation]{Example}

\theoremstyle{remark}

\newenvironment{cstack}[1][1.3]{\renewcommand{\arraystretch}{#1}
\renewcommand{\arraycolsep}{0pt}\begin{array}{c}}
{\end{array}}

\title{Fusion systems realizing certain Todd modules}

\author{Bob Oliver}
\address{Universit\'e Sorbonne Paris Nord, LAGA, UMR 7539 du CNRS, 
99, Av. J.-B. Cl\'ement, 93430 Villetaneuse, France.}
\email{bobol@math.univ-paris13.fr}
\thanks{B. Oliver is partially supported by UMR 7539 of the CNRS. Part of 
this work was carried out at the Isaac Newton Institute for Mathematical 
Sciences during the programme GRA2, supported by EPSRC grant nr. 
EP/K032208/1.}


\begin{document}

\begin{abstract} 
We study a certain family of simple fusion systems over finite $3$-groups, 
ones that involve Todd modules of the Mathieu groups $2M_{12}$, $M_{11}$, 
and $A_6=O^2(M_{10})$ over $\F_3$, and show that they are all isomorphic to 
the $3$-fusion systems of almost simple groups. As one consequence, we give 
new $3$-local characterizations of Conway's sporadic simple groups.
\end{abstract}

\maketitle

Fix a prime $p$. A fusion system over a finite $p$-group $S$ is a 
category whose objects are the subgroups of $S$, and whose morphisms 
are injective homomorphisms between the subgroups satisfying certain 
axioms first formulated by Puig \cite{Puig}, and modeled on the Sylow 
theorems for finite groups. The motivating example is the fusion system 
of a finite group $G$ with $S\in\sylp{G}$, whose morphisms are those 
homomorphisms between subgroups of $S$ induced by conjugation in $G$. 

The general theme in this paper is to study fusion systems over finite 
$p$-groups $S$ that contain an abelian subgroup $A\nsg S$ such that 
$A\nnsg\calf$ and $C_S(A)=A$. In such 
situations, we let $\Gamma=\autf(A)$ be its automizer, 
try to understand what restrictions the existence of such a fusion 
system imposes on the pair $(A,O^{p'}(\Gamma))$, and also look for 
tools to describe all fusion systems that ``realize'' a given pair 
$(A,O^{p'}(\Gamma))$ for $A$ an abelian $p$-group and $\Gamma\le\Aut(A)$. 

This paper is centered around one family of examples: those where 
$p=3$, where $O^{3'}(\Gamma)\cong2M_{12}$, $M_{11}$, or $A_6=O^{3'}(M_{10})$, 
and where $A$ is elementary abelian of rank $6$, $5$, or $4$, respectively. 
But we hope that the tools we use to handle these cases will also be 
useful in many other situations. Our main results can be summarized as 
follows:

\begin{Thmm} \label{ThA}
Let $\calf$ be a saturated fusion system over a finite $3$-group $S$ with 
an elementary abelian subgroup $A\le S$ such that 
$C_{S}(A)=A$, and such that either 
\begin{enumi} 
\item $\rk(A)=6$ and $O^{3'}(\Aut_{\calf}(A))\cong2M_{12}$; or 
\item $\rk(A)=5$ and $O^{3'}(\Aut_{\calf}(A))\cong M_{11}$; or 
\item $\rk(A)=4$ and $O^{3'}(\Aut_{\calf}(A))\cong A_6$.
\end{enumi}
Assume also that $A\nnsg\calf$. Then $A\nsg S$, $S$ splits over 
$A$, and $O^{3'}(\calf)$ is simple and isomorphic to the $3$-fusion system 
of $\Co_1$ in case \textup{(i)}, to that of $\Suz$, $\Ly$, or $\Co_3$ in 
case \textup{(ii)}, or to that of $U_4(3)$, $U_6(2)$, $\McL$, or $\Co_2$ in 
case \textup{(iii)}. 
\end{Thmm}

Theorem \ref{ThA} is proven below as Theorem \ref{t:M12case} (case (i)) and 
Theorem \ref{t:A6+M11} (cases (ii) and (iii)). As one consequence of these 
results, we give new 3-local characterizations of the three Conway 
groups as well as of $\McL$ and $U_6(2)$ (Theorems \ref{t:3-loc.Co1}, 
\ref{t:3-loc.Co3}, and \ref{t:3-loc.81.A6}).

All three cases of Theorem \ref{ThA} have already been shown in earlier 
papers using very different methods. In \cite[Theorem A]{vanbeek}, Martin 
van Beek determined (among other results) all fusion systems $\calf$ over a 
Sylow $3$-subgroup of $\Co_1$ with $O_3(\calf)=1$. In \cite{BFM}, 
Baccanelli, Franchi, and Mainardis listed all saturated fusion systems 
$\calf$ with $O_3(\calf)=1$ over a Sylow $3$-subgroup of the \emph{split} 
extension $E_{81}\rtimes A_6$, and this includes the four systems that 
appear in case (iii) of the above theorem. In \cite{PS}, Parker and 
Semeraro develop computer algorithms that they use to list, among other 
things, all saturated fusion systems $\calf$ over $3$-groups of order at 
most $3^7$ with $O_3(\calf)=1$ and $O^3(\calf)=\calf$. However, our goals 
are different from those in the earlier papers, in that we want to develop 
tools which can be used in other situations within the framework of the 
general problem described above, and are using these Todd modules as test 
cases. 

The proof of Theorem \ref{ThA} is straightforward, following a program that 
also seems to work in many other cases. Set $Z=Z(S)$. We first show that 
$\calf=\gen{C_{\calf}(Z),N_{\calf}(A)}$. We then construct a special subgroup 
$Q\nsg S$ of exponent $3$ with $Z(Q)=[Q,Q]=Z$ (of order $3$ or 
$9$) and $Q/Z(Q)\cong E_{81}$, and show that $Q$ is normal in 
$C_{\calf}(Z)$. This is the hardest part of the proof, especially when 
$O^{3'}(\Aut_{\calf}(A))\cong2M_{12}$. Finally, we determine the different 
possibilities for $O^{3'}(\Out_{\calf}(Q))$, and show that this group 
together with $O^{3'}(\Aut_{\calf}(A))$ determines $O^{3'}(\calf)$ up to 
isomorphism. 

Theorem \ref{ThA} involves just one special case of the following general 
problem. Given a prime $p$, a finite group $\Gamma=O^{p'}(\Gamma)$, and a 
finite $\F_p\Gamma$-module $M$ (or more generally, a finite 
$\Z/p^k\Gamma$-module for some $k>1$), we say that a saturated fusion 
system $\calf$ over a finite $p$-group $S$ ``realizes'' $(\Gamma,M)$ if 
there is an abelian subgroup $A\le S$ such that $C_S(A)=A$, $A\nnsg\calf$, 
and $(O^{p'}(\autf(A)),A)\cong(\Gamma,M)$. We want to know whether a given 
module can be realized in this sense, and if so, list all of the distinct 
saturated fusion systems that realize it. 

In the papers \cite{indp1}, \cite{indp2}, and \cite{indp3}, we studied 
this question under the additional assumption that $|\Gamma|$ be a 
multiple of $p$ but not of $p^2$, and the answer in that case was 
already quite complicated. In this more general setting, all we can hope 
to do for now is to look at a few more cases, and try to develop some 
tools that can be used in greater generality. For example, in a second 
paper \cite{O-nonreal} still in preparation, we give some criteria 
for the nonrealizability of certain $\F_p\Gamma$-modules. As one 
application of those results, when $\Gamma\cong M_{11}$, $M_{12}$, or 
$2M_{12}$, we show that up to extensions by trivial modules, the only 
$\F_p\Gamma$-modules that can be realized in the above sense are the 
Todd modules of $M_{11}$ and $2M_{12}$ and their duals (when $p=3$), 
and the simple $10$-dimensional $\F_{11}[2M_{12}]$-modules. 

As pointed out by the referee, Theorem \ref{ThA} in this paper is closely 
related to the list of amalgams by Papadopoulos in \cite{Papadopoulos}. It 
seems quite possible that the results in this paper can be used to 
strengthen or generalize the main theorem in \cite{Papadopoulos}, but if 
so, that will have to wait for a separate (short) paper.

General definitions and properties involving saturated fusion systems are 
surveyed in Section \ref{s:background}, while the more technical results 
needed to carry out the programme described above are listed in Section 
\ref{s:general}. In Section \ref{s:Todd}, we set up some notation for 
working with Todd modules for $2M_{12}$ and $M_{11}$; notation which we 
hope might also be useful in other contexts. Case (i) of Theorem 
\ref{ThA} is proven in Section \ref{s:M12}, and the remaining cases in 
Section \ref{s:M10-11}. The $3$-local characterizations of the Conway 
groups and some others are given in Section \ref{s:conway}. We finish with 
two appendices: one containing a few general group theoretic results, and 
another more specifically focused on groups with strongly $p$-embedded 
subgroups. 

\noindent\textbf{Notation and terminology:} Most of our notation for 
working with groups is fairly standard. When $P\le G$ and $x\in 
N_G(P)$, we let $c_x^P\in\Aut(P)$ denote conjugation by $x$ on the 
left: $c_x^P(g)=\9xg=xgx^{-1}$ (though the direction of conjugation 
very rarely matters). Our commutators have the form 
$[x,y]=xyx^{-1}y^{-1}$. If $G$ is a group and $\alpha\in\Aut(G)$, then 
$[\alpha]\in\Out(P)$ denotes its class modulo $\Inn(G)$. If 
$\varphi\in\Hom(G,H)$ is a homomorphism, $Q$ is normal in both $G$ and 
$H$, and $\varphi(Q)=Q$, then $\varphi/Q\in\Hom(G/Q,H/Q)$ denotes the 
induced map between quotients. Also, $\sylp{G}$ is the set of Sylow 
$p$-subgroups of a finite group $G$, $\scrs(G)$ is the set of all subgroups 
of $G$, and $Z_2(G)$ is the second term in its upper central 
series ($Z_2(G)/Z(G)=Z(G/Z(G))$). 

Other notation used here includes:
\begin{itemize} 
\item $E_{p^m}$ is always an elementary abelian $p$-group of rank $m$; 
\item $p^{a+b}$ denotes a special $p$-group $P$ with $Z(P)=[P,P]\cong 
E_{p^a}$ and $P/Z(P)\cong E_{p^b}$; 
\item $p^{1+2m}_+$ (when $p$ is odd) is an extraspecial $p$-group 
of order $p^{1+2m}$ and exponent $p$; 
\item $A\circ B$ is a central product of groups $A$ and $B$; 
\item $A\rtimes B$ and $A.B$ are a semidirect product and an 
arbitrary extension of $A$ by $B$; 
\item $\UT_n(q)$ is the group of upper triangular $(n\times 
n)$-matrices over $\F_q$ with $1$'s on the diagonal; and 
\item $\GGL_n(q)$ and $\PGGL_n(q)$ denote the extensions of $\GL_n(q)$ and 
$\PGL_n(q)$ by their field automorphisms. 
\end{itemize}
Also, $2M_{12}$, $2A_n$, and $2\Sigma_n$ ($n=4,5,6$) denote nonsplit 
central extensions of $C_2$ by the groups $M_{12}$, $A_n$, and 
$\Sigma_n$, respectively. 

\noindent\textbf{Thanks: } The author would especially like to thank the 
referee for the many helpful suggestions, including some several involving 
potential connections with other papers. He would also like to thank the 
Isaac Newton Institute for its hospitality while the paper was being 
revised.


\section{Background}
\label{s:background}

We begin with a survey of the basic definitions and terminology involving 
fusion systems that will be needed here, such as normalizer fusion systems, 
the Alperin-Goldschmidt fusion theorem for fusion systems, and the model 
theorem. Most of these definitions and results are originally due to Puig 
\cite{Puig}. 


\newsubb{Basic definitions and terminology}{s:basic}

A \emph{fusion system} $\calf$ over a finite $p$-group $S$ is a 
category whose objects are the subgroups of $S$, and whose morphism sets 
$\homf(P,Q)$ are such that \medskip
\begin{itemize}
\item $\Hom_S(P,Q)\subseteq\homf(P,Q)\subseteq\Inj(P,Q)$ for all $P,Q\le 
S$; and 

\item every morphism in $\calf$ factors as an isomorphism in $\calf$
followed by an inclusion.
\end{itemize} 
For this to be very useful, more conditions are needed.

\begin{Defi} \label{d:s.f.s.}
Let $\calf$ be a fusion system over a finite $p$-group $S$.
\begin{enuma}
\item Two subgroups $P,P'\le{}S$ are \emph{$\calf$-conjugate} if 
$\isof(P,P')\ne\emptyset$, and two elements $x,y\in S$ are 
$\calf$-conjugate if there is $\varphi\in\homf(\gen{x},\gen{y})$ such that 
$\varphi(x)=y$. The $\calf$-conjugacy classes of $P\le S$ and $x\in S$ are 
denoted $P^\calf$ and $x^\calf$, respectively.

\item A subgroup $P\le S$ is \emph{fully normalized in $\calf$} \emph{(fully 
centralized in $\calf$)} if $|N_S(P)|\ge|N_S(Q)|$ ($|C_S(P)|\ge|C_S(Q)|$) 
for each $Q\in P^\calf$. 
		
\item The fusion system $\calf$ is \emph{saturated} if it satisfies the 
following two conditions:\medskip
\begin{itemize}

\item \textup{(Sylow axiom)} For each subgroup $P\le S$ fully normalized in 
$\calf$, $P$ is fully centralized and $\Aut_S(P)\in\sylp{\autf(P)}$.

\item \textup{(extension axiom)} For each isomorphism $\varphi\in\isof(P,Q)$ 
in $\calf$ such that $Q$ is fully centralized in $\calf$, $\varphi$ extends 
to a morphism $\4\varphi\in\homf(N_\varphi,S)$ where 
	\[ N_\varphi = \{ g\in N_S(P) \,|\, \varphi c_g \varphi^{-1} \in 
	\Aut_S(Q) \}. \]
\end{itemize}
\end{enuma}
\end{Defi}

In the following lemma, we describe another important property of fully 
normalized subgroups. 

\begin{Lem}[{\cite[Lemma I.2.6(c)]{AKO}}] \label{p:Hom(NSP,S)}
Let $\calf$ be a saturated fusion system over a finite $p$-group $S$. Then 
for each $P\le S$ and each $Q\in P^\calf\cap\calf^f$, there is 
$\psi\in\homf(N_S(P),S)$ such that $\psi(P)=Q$.
\end{Lem}

We next recall a few more classes of subgroups in a fusion system. As 
usual, for a fixed prime $p$, a proper subgroup $H$ of a finite group 
$G$ is \emph{strongly $p$-embedded} if $p\mid|H|$, and 
$p\nmid|H\cap\9xH|$ for each $x\in G\sminus H$. 

\begin{Defi} \label{d:subgroups}
Let $\calf$ be a fusion system over a finite $p$-group $S$. For 
$P\le S$, 
\begin{itemize}

\item $P$ is \emph{$\calf$-centric} if $C_S(Q)\le Q$ for each $Q\in P^\calf$; 

\item $P$ is \emph{$\calf$-essential} if $P$ is $\calf$-centric and fully 
normalized in $\calf$, and the group $\outf(P)=\autf(P)/\Inn(P)$ contains a 
strongly $p$-embedded subgroup; 



\item $P$ is \emph{weakly closed in $\calf$} if $P^\calf=\{P\}$; 

\item $P$ is \emph{strongly closed in $\calf$} if for each $x\in P$, 
$x^\calf\subseteq P$; and 

\item $P$ is \emph{normal in $\calf$} ($P\nsg\calf$) if each morphism in $\calf$ 
extends to a morphism that sends $P$ to itself. Let 
$O_p(\calf)\nsg\calf$ be the largest subgroup of $S$ normal in $\calf$.

\item $P$ is \emph{central in $\calf$} if each morphism in $\calf$ 
extends to a morphism that sends $P$ to itself via the identity. Let 
$Z(\calf)\nsg\calf$ be the largest subgroup of $S$ central in $\calf$.

\end{itemize}
\end{Defi}

Clearly, if $P$ is weakly closed in $\calf$, then it must be normal in 
$S$. 

It follows immediately from the definitions that if $P_1$ and $P_2$ are 
both normal in $\calf$, then so is $P_1P_2$. So $O_p(\calf)$ is 
defined, and a similar argument applies to show that $Z(\calf)$ is 
defined.

The following notation is useful when referring to some of these classes 
of subgroups.

\begin{Not} \label{n:F^c,F^f}
For each fusion system $\calf$ over a finite $p$-group $S$, define 
\begin{itemize} 
\item $\calf^f=\{P\le S\,|\,P~\textup{is fully normalized in $\calf$}\}$; 
\item $\calf^c=\{P\le S\,|\,P~\textup{is $\calf$-centric}\}$ and 
$\calf^{cf}=\calf^c\cap\calf^f$; and
\item $\EE\calf=\{P\le S\,|\,P~\textup{is $\calf$-essential}\}$.
\end{itemize}
\end{Not}

\newsubb{The Alperin-Goldschmidt fusion theorem for fusion systems}{s:AFT}

The following is one version of the Alperin-Goldschmidt fusion theorem for fusion 
systems. This theorem is our main motivation for defining 
$\calf$-essential subgroups here. 

\begin{Thm}[{\cite[Theorem I.3.6]{AKO}}] \label{t:AFT}
Let $\calf$ be a saturated fusion system over a finite $p$-group $S$. Then 
each morphism in $\calf$ is a composite of restrictions of automorphisms 
$\alpha\in\autf(R)$ for $R\in\EE\calf\cup\{S\}$.
\end{Thm}

Equivalently, Theorem \ref{t:AFT} says that 
$\calf=\gen{\autf(P)\,|\,P\in\EE\calf\cup\{S\}}$. Here, whenever 
$\calf$ is a fusion system over $S$, and $\scrx$ is a set of fusion 
subsystems and morphisms in $\calf$, we let $\gen{\scrx}$ denote the 
smallest fusion system over $S$ that contains $\scrx$. Since an 
intersection of fusion subsystems over $S$ is always a fusion system 
over $S$ (not necessarily saturated, of course), the subsystem 
$\gen\scrx$ is well defined. 

In fact, up to $\calf$-conjugacy, the essential subgroups form the 
smallest possible set of subgroups that generate $\calf$. 

\begin{Prop} \label{p:-AFT}
Let $\calf$ be a saturated fusion system over a finite $p$-group $S$, 
and let $\scrt$ be a set of subgroups of $S$ such that 
	$ \calf = \gen{\autf(P)\,|\,P\in\scrt}$. 
Then each $\calf$-essential subgroup $R<S$ is $\calf$-conjugate to a 
member of $\scrt$. 
\end{Prop}

\begin{proof} Fix $R\in\calf^f$ such that $R<S$ and 
$R^\calf\cap\scrt=\emptyset$, and set 
	\[ \autf^0(R) = \Gen{\alpha\in\autf(R) \,\big|\, 
	\alpha=\4\alpha|_R, ~\textup{some $\4\alpha\in\homf(P,S)$ where 
	$R<P\le S$} }. \]
We will prove that $\autf^0(R)=\autf(R)$. It will then follow that $R$ 
is not $\calf$-essential (see \cite[Proposition I.3.3(b)]{AKO}), thus 
proving the proposition.

Fix $\alpha\in\autf(R)$. By assumption, there are isomorphisms 
	\[ R = R_0 \RIGHT2{\alpha_1}{\cong} R_1 \RIGHT2{\alpha_2}{\cong} R_2 
	\RIGHT2{\alpha_3}{\cong} \cdots \RIGHT2{\alpha_k}{\cong} R_k = R \]
such that $\alpha=\alpha_k\circ\cdots\circ\alpha_1$, together with 
automorphisms $\beta_i\in\autf(P_i)$ for $1\le i\le k$ such that 
$\gen{R_{i-1},R_i}\le P_i\in\scrt$ and $\alpha_i=\beta_i|_{R_{i-1}}$. 

By Lemma \ref{p:Hom(NSP,S)} and since $R\in\calf^f$, for each 
$0\le i\le k$, there is $\chi_i\in\homf(N_S(R_i),N_S(R))$ such that 
$\chi_i(R_i)=R$, where we take $\chi_0=\chi_k=\Id_{N_S(R)}$. For each 
$1\le i\le k$, set 
	\[ \5R_{i-1}=N_{P_i}(R_{i-1}) \quad\textup{and}\quad
	\5\alpha_i= (\chi_i) \circ (\beta_i|_{\5R_{i-1}}) \circ 
	(\chi_{i-1}^{-1}|_{\chi_{i-1}(\5R_{i-1})}) \in \homf(\5R_{i-1},S). \]
Then $\5\alpha_i|_R=(\chi_i|_{R_i})\circ\alpha_i\circ(\chi_{i-1}^{-1}|_{R_{i-1}}) 
\in\autf(R)$ for each $i$.

For each $i$, $P_i>R_{i-1}$ since $P_i\in\scrt$ while $R_{i-1}\in 
R^\calf$ and $R^\calf\cap\scrt=\emptyset$. Hence $\5R_{i-1}>R$ for each 
$1\le i\le k$. By construction, 
$\alpha=(\5\alpha_k|_R)\circ\cdots\circ(\5\alpha_1|_R)$, and so 
$\alpha\in\autf^0(R)$. Since $\alpha\in\autf(R)$ was arbitrary, this 
proves that $\autf^0(R)=\autf(R)$, as claimed.
\end{proof}

The next two lemmas give different conditions for a subgroup to be 
normal in a fusion system. Both are consequences of Theorem \ref{t:AFT}.


\begin{Lem} \label{l:wcl-nm}
Let $\calf$ be a saturated fusion system over a finite $p$-group $S$. A 
subgroup $Q\le S$ is normal in $\calf$ if and only if it is weakly 
closed and contained in all $\calf$-essential subgroups. 
\end{Lem}

\begin{proof} This is essentially the equivalence 
(a\,$\Leftrightarrow$\,c) in \cite[Proposition I.4.5]{AKO}.
\end{proof}

In general, strongly closed subgroups in a saturated fusion system need 
not be normal. The next lemma describes one case where this does 
happen.

\begin{Lem}[{\cite[Corollary I.4.7(a)]{AKO}}] \label{l:scl-nm}
Let $\calf$ be a saturated fusion system over a finite $p$-group $S$. If 
$A\nsg S$ is an abelian subgroup that is strongly closed in $\calf$, then 
$A\nsg\calf$. 
\end{Lem}

\smallskip

\newsubb{Normalizer fusion subsystems and models}{s:NF(Q)}

If $\calf$ is a fusion system over a finite $p$-group $S$, then a 
\emph{fusion subsystem} $\cale\le\calf$ over a subgroup $T\le S$ is a 
subcategory $\cale$ whose objects are the subgroups of $T$, such that 
$\cale$ is itself a fusion system over $T$. For example, the full 
subcategory of $\calf$ with objects the subgroups of $T$ is a fusion 
subsystem of $\calf$. If we want our fusion subsystems to be saturated, 
then, of course, the problem of constructing them is more subtle. 

One case where this is straightforward is the construction of 
normalizers and centralizers of subgroups in a fusion system.

\begin{Defi} \label{d:NF(Q)}
Let $\calf$ be a fusion system over a finite $p$-group $S$. For each 
$Q\le S$, we define fusion subsystems $C_\calf(Q)\le 
N_\calf(Q)\le\calf$ over $C_S(Q)\le N_S(Q)$ by setting 
	\begin{align*} 
	\Hom_{C_\calf(Q)}(P,R) &= \bigl\{ \varphi|_P \,\big|\, 
	\varphi\in\homf(PQ,RQ),~ \varphi(P)\le R,~ \varphi|_Q=\Id_Q 
	\bigr\} \\
	\Hom_{N_\calf(Q)}(P,R) &= \bigl\{ \varphi|_P \,\big|\, 
	\varphi\in\homf(PQ,RQ),~ \varphi(P)\le R,~ \varphi(Q)=Q \bigr\}. 
	\end{align*}
\end{Defi}

It follows immediately from the definitions that a subgroup $Q\le S$ is 
normal or central in $\calf$ if and only if $N_\calf(Q)=\calf$ or 
$C_\calf(Q)=\calf$, respectively. 

\begin{Thm}[{\cite[Theorem I.5.5]{AKO}}] \label{t:NF(Q)}
Let $\calf$ be a saturated fusion system over a finite $p$-group $S$, 
and fix $Q\le S$. Then $C_\calf(Q)$ is saturated if $Q$ is 
fully centralized in $\calf$, and $N_\calf(Q)$ is saturated 
if $Q$ is fully normalized in $\calf$. 
\end{Thm}

We next look at models for constrained fusion systems, and in 
particular, for normalizer fusion subsystems of centric subgroups.

\begin{Defi} \label{d:model}
Let $\calf$ be a saturated fusion system over a finite $p$-group $S$.
\begin{enuma} 
\item The fusion system $\calf$ is \emph{constrained} if there is a 
subgroup $Q\le S$ that is normal in $\calf$ and $\calf$-centric; 
equivalently, if $O_p(\calf)\in\calf^c$. 
\item A \emph{model} for a constrained fusion system $\calf$ over 
$S$ is a finite group $M$ with $S\in\sylp{M}$, such that 
$S\in\sylp{M}$, $\calf_S(M)=\calf$, and $C_M(O_p(M))\le O_p(M)$. 

\end{enuma}
\end{Defi}

By the \emph{model theorem} (see \cite[Theorem III.5.10]{AKO}), every 
constrained fusion system has a model, unique up to isomorphism. We 
will need this only in the following situation.

\begin{Prop} \label{p:NF(Q)model}
Let $\calf$ be a saturated fusion system over a finite $p$-group $S$. 
Then for each $Q\in\calf^{cf}$, the normalizer fusion subsystem 
$N_\calf(Q)$ is constrained, and hence has a model: a finite group $M$ 
with $N_S(Q)\in\sylp{M}$ such that $Q\nsg M$, $C_M(Q)\le Q$, and 
$\calf_{N_S(Q)}(M)=N_\calf(Q)$. Furthermore, $M$ is unique in the following 
sense: if $M^*$ is another model for $N_\calf(Q)$, also with $Q\nsg M^*$ 
and $N_S(Q)\in\sylp{M^*}$, then $M\cong M^*$ via an isomorphism that 
restricts to the identity on $N_S(Q)$.
\end{Prop}

\begin{proof} The subsystem $N_\calf(Q)$ is constrained since the subgroup 
$Q$ is normal and $N_\calf(Q)$-centric. So by the model theorem 
\cite[Theorem III.5.10]{AKO}, it has a model, and any two models for 
$N_\calf(Q)$ are isomorphic via an isomorphism that is the identity on 
$N_S(Q)$. 
\end{proof}


\newsubb{Subsystems of index prime to \texorpdfstring{$p$}{p}}{s:Op'F}

We next turn to fusion subsystems of index prime to $p$. By analogy 
with groups, this really corresponds to subgroups of a finite group $G$ 
that contain $O^{p'}(G)$ (but are not necessarily normal). 

\begin{Defi} \label{d:Op'(F)}
Let $\calf$ be a fusion system over a finite $p$-group 
$S$. A fusion subsystem $\cale\le\calf$ has \emph{index prime to $p$} 
if $\cale$ is also a fusion system over $S$, and $\Aut_\cale(P)\ge 
O^{p'}(\autf(P))$ for each $P\le S$. 
\end{Defi}

There is clearly always a smallest fusion subsystem of $\calf$ of index 
prime to $p$: the subsystem $O^{p'}_*(\calf)$ over $S$ generated by the 
automorphism groups $O^{p'}(\autf(P))$. The corresponding result 
for saturated fusion subsystems is more subtle. 

\begin{Thm} \label{t:Op'(F)}
Let $\calf$ be a saturated fusion system over a finite $p$-group $S$. 
Then there is a (unique) smallest saturated fusion subsystem 
$O^{p'}(\calf)\le\calf$ of index prime to $p$. This has the property 
that for each $P\le S$ and each $\varphi\in\homf(P,S)$, there are 
morphisms $\varphi_0\in\Hom_{O^{p'}(\calf)}(P,S)$ and 
$\alpha\in\autf(S)$ such that $\varphi=\alpha\circ\varphi_0$. 
\end{Thm}

\begin{proof} See \cite[Theorem I.7.7]{AKO} or \cite[Theorem 
5.4]{BCGLO2} for the existence and uniqueness of $O^{p'}(\calf)$. The 
last statement follows from Lemma 3.4(c) in \cite{BCGLO2}, or since the 
map $\theta\:\Mor(\calf^c)\too\Gamma_{p'}(\calf)$ sends $\autf(S)$ 
surjectively. 
\end{proof}

In fact, the theorems in \cite{AKO} and in \cite{BCGLO2} cited above 
both describe the subsystem $O^{p'}(\calf)$ in more precise detail. 

\begin{Prop} \label{p:EOp'(F)}
For each saturated fusion system $\calf$ over a finite $p$-group $S$, 
we have $O^{p'}(\calf)^c=\calf^c$, $O^{p'}(\calf)^f=\calf^f$, and 
$\EE[*]{O^{p'}(\calf)}=\EE\calf$. 
\end{Prop}

\begin{proof} By Theorem \ref{t:Op'(F)}, if $P\le S$ and $Q\in 
P^\calf$, then there is $\alpha\in\autf(S)$ such that $\alpha(Q)\in 
P^{O^{p'}(\calf)}$. From this, it follows immediately that 
$O^{p'}(\calf)$ and $\calf$ have the same centric subgroups, and the 
same fully normalized subgroups. To see that they have the same 
essential subgroups, it remains to check that $\Out_{O^{p'}(\calf)}(P)$ 
has a strongly $p$-embedded subgroup if and only if $\Out_\calf(P)$ 
does, and this is shown in Lemma \ref{l:str.emb2}. 
\end{proof}

We also need the following result, which gives a more precise 
description of $O^{p'}(\calf)$, but under very restrictive conditions 
on $\calf$.


\begin{Prop} \label{p:<Op'NR>}
Let $\calf$ be a saturated fusion system over a finite $p$-group $S$, 
such that 
\begin{enumi} 
\item $\EE\calf\ne\emptyset$ and each member of $\EE\calf$ is weakly 
closed in $\calf$, and 
\item no intersection of two distinct members of $\EE\calf$ is 
$\calf$-centric.
\end{enumi}
Then 
\begin{enuma} 

\item $\Aut_{O^{p'}(N_\calf(R))}(P)=\bigl\{\alpha\in\autf(P)\,\big|\, 
\alpha|_R\in O^{p'}(\autf(R))\bigr\}$ for each $R\in\EE\calf$ and each 
$R\le P\le S$; and 

\item $\Aut_{O^{p'}(\calf)}(S)=\gen{\Aut_{O^{p'}(N_\calf(R))}(S)\,|\,
R\in\EE\calf}$. 

\end{enuma}
\end{Prop}

\begin{proof} For each $R\in\EE\calf$, set $\cale_R=O^{p'}(N_\calf(R))$. 

\smallskip

\noindent\textbf{(a) } Fix $R\in\EE\calf$, and let 
$H$ be a model for $N_\calf(R)$ (see Proposition \ref{p:NF(Q)model}). 
Then $O^{p'}(H)$ is a model for $\cale_R$, and an extension of $R$ by 
$O^{p'}(H/R)\cong O^{p'}(\outf(R))$. Hence 
	\[ \Aut_{\cale_R}(R)=\Aut_{O^{p'}(H)}(R)=O^{p'}(\Aut_H(R))
	=O^{p'}(\autf(R)). \]

Let $P$ be such that $R\le P\le S$. Then $\alpha\in\Aut_{\cale_R}(P)$ 
implies $\alpha|_R\in\Aut_{\cale_R}(R)=O^{p'}(\autf(R))$. 
Conversely, if $\alpha\in\autf(P)$ is such that $\alpha|_R\in 
O^{p'}(\autf(R))=\Aut_{\cale_R}(R)$, then by the extension axiom and 
since $\alpha|_R$ normalizes $\Aut_P(R)$, there is 
$\beta\in\Aut_{\cale_R}(P)$ such that $\beta|_R=\alpha|_R$. So by 
\cite[Lemma I.5.6]{AKO} and since $R\in\calf^c$, there is $x\in Z(R)$ 
such that $\alpha=\beta\circ c_x$, and hence $\alpha\in\Aut_{\cale_R}(P)$.

\smallskip

\noindent\textbf{(b) } Set 
	\[ \calf_0=\gen{O^{p'}(\autf(R))\,|\,R\in\EE\calf} 
	\qquad\textup{and}\qquad
	O^{p'}_*(\calf) = \gen{O^{p'}(\autf(P)) \,|\,P\le S} \]
as (not necessarily saturated) fusion systems over $S$. Thus 
$O^{p'}_*(\calf)$ is the minimal fusion subsystem in $\calf$ of index prime to 
$p$. For $P\in\calf^c$, since $P$ is contained in at 
most one member of $\EE\calf$ by (ii), the sets $\homf(P,S)$ and 
$\homf[0](P,S)$ and groups $\autf(P)$ and $\autf[0](P)$ are described 
as follows:
	\begin{Table}[ht]
	\[ \renewcommand{\arraystretch}{1.4} 
	\begin{array}{c|cc}
	\cale & \Hom_\cale(P,S) & \Aut_\cale(P) \\\hline
	\calf & \bigl\{\alpha|_P \,\big|\, \alpha\in\autf(R)\bigr\} &
	\bigl\{\alpha|_P \,\big|\, \alpha\in \autf(R),~ \alpha(P)=P \bigr\} \\
	\calf_0 & \bigl\{\alpha|_P \,\big|\, \alpha\in O^{p'}(\autf(R))\bigr\} &
	\bigl\{\alpha|_P \,\big|\, \alpha\in O^{p'}(\autf(R)),~ \alpha(P)=P \bigr\}
	\end{array} \]
	\caption{In each case, either $R$ is the unique member of 
	$\EE\calf$ such that $P\le R$, or $R=S$ if there is no such 
	member.} \label{tbl:Hom(P,S)}
	\end{Table}

In particular, this shows that $\autf[0](P)$ is normal of index prime 
to $p$ in $\autf(P)$ for each $P\in\calf^c$, and hence by \cite[Lemma 
I.7.6(a)]{AKO} that $\calf_0$ has index prime to $p$ in $\calf$. 
Thus $\calf_0=O^{p'}_*(\calf)$ (the inclusion $\calf_0\le 
O^{p'}_*(\calf)$ is immediate from the definitions). So 
	\begin{align*} 
	\Aut&_{O^{p'}(\calf)}(S) = \Gen{ \alpha\in\autf(S) \,\big|\, 
	\alpha|_P\in\Hom_{O^{p'}_*(\calf)}(P,S), ~\textup{some 
	$P\in\calf^c$} } \\ 
	&= \Gen{ \alpha\in\autf(S) \,\big|\, \alpha|_P\in 
	\Hom_{\calf_0}(P,S) ~\textup{some $P\in\calf^c$} } \\
	&= \Gen{\alpha\in\autf(S) \,\big|\, 
	\exists\, P\in\calf^c,~ P\le R\in\EE\calf\cup\{S\},~ \beta\in 
	O^{p'}(\autf(R)),~ \textup{s.t. $\alpha|_P=\beta|_P$} } 
	\\
	&= \Gen{\alpha\in\autf(S) \,\big|\, \alpha|_R\in 
	O^{p'}(\autf(R)) ~\textup{some $R\in\EE\calf\cup\{S\}$} } \\
	&= \gen{\Aut_{\cale_R}(S) \,|\, R\in\EE\calf} \,: 
	\end{align*}
the first equality by \cite[Theorem I.7.7]{AKO}, the second since 
$\calf_0=O^{p'}_*(\calf)$, the third by 
Table \ref{tbl:Hom(P,S)}, the fourth since $\alpha|_P=\beta|_P$ 
implies $\alpha|_R=\beta\circ c_x$ for some $x\in Z(P)$ (see \cite[Lemma 
I.5.6]{AKO}), and the last by (a) (applied with $P=S$). 
\end{proof}

One can also show that 
$O^{p'}(\calf)=\gen{O^{p'}(N_\calf(R))\,|\,R\in\EE\calf}$ under the 
hypotheses of Proposition \ref{p:<Op'NR>}. However, that will not be 
needed here.


\newsubb{Quotient fusion systems}{s:F/Q}

Quotient fusion systems of $\calf$ over $S$ are formed by dividing out 
by a subgroup of $S$, not by a fusion subsystem of $\calf$.

\begin{Defi} \label{d:F/Q}
Let $\calf$ be a fusion system, and assume $Q\le S$ is strongly 
closed in $\calf$. In particular, $Q\nsg S$. Let $\calf/Q$ be the 
fusion system over $S/Q$ where for each $P,R\le S$ containing $Q$, we 
set
	\begin{multline*} 
	\Hom_{\calf/Q}(P/Q,R/Q) =\\ \bigl\{\varphi/Q\in\Hom(P/Q,R/Q) 
	\,\bigl|\, \varphi\in\homf(P,Q), ~ (\varphi/Q)(gQ)=\varphi(g)Q 
	~\forall\,g\in P \bigr\}.
	\end{multline*}
\end{Defi}

We refer to \cite[Proposition II.5.11]{Craven} for the proof that 
$\calf/Q$ is saturated whenever $\calf$ is. In fact, the definition and 
saturation of $\calf/Q$ hold whenever $Q$ is weakly closed in $\calf$. 
This is not surprising, since we are looking only at morphisms in 
$\calf$ between subgroups containing $Q$, so that 
$\calf/Q=N_\calf(Q)/Q$.

If $Q$ is strongly closed in $\calf$, then every morphism 
$\varphi\in\homf(P,R)$, for arbitrary $P,Q\le S$, induces a (unique) 
morphism $\4\varphi\in\Hom(PQ/Q,RQ/Q)$. (Just note that $\varphi(P\cap 
Q)\le R\cap Q$.) A much deeper theorem states that each such morphism 
$\4\varphi$ also lies in $\calf/Q$. We refer to \cite[Theorem 
II.5.12]{AKO} and \cite[Theorem II.5.14]{Craven} for proofs of this 
result first shown by Puig. In this paper, however, we work with 
$\calf/Q$ only in the special case where $Q\nsg\calf$, in which case 
this property is automatic.

We will need the following lemma, comparing essential subgroups in 
$\calf$ and in $\calf/Z$ when $Z$ is central in $\calf$. 

\begin{Lem} \label{l:(F/Z)e}
Let $\calf$ be a saturated fusion system over a finite $p$-group $S$, and 
fix $Z\le Z(\calf)$. Then for each $R\le S$, $R\in\EE\calf$ if and only 
if $R\ge Z$ and $R/Z\in\EE[*]{\calf/Z}$.
\end{Lem}

\begin{proof} If $R\in\EE\calf$, then 
$R\in\calf^c$, and hence $R\ge Z(S)\ge Z$. So from now on, we always 
assume that $R\ge Z$. We will show that the following hold for each 
$R\le S$ containing $Z$: 
\begin{enuma} 

\item $R\in\calf^f$ if and only if $R/Z\in(\calf/Z)^f$; 

\item the natural map $\Psi\:\outf(R)\too\Out_{\calf/Z}(R/Z)$ is surjective and 
its kernel is a $p$-group; and 

\item $R/Z\in(\calf/Z)^c$ if and only if $R\in\calf^c$ and 
$\Psi$ is an isomorphism.

\end{enuma}
It follows immediately from (a), (b), and (c) and Definition \ref{d:subgroups} that 
$R\in\EE\calf$ if $R/Z\in\EE[*]{\calf/Z}$. Conversely, if $R\in\EE\calf$, 
then $O_p(\outf(R))=1$ since $\outf(R)$ has a strongly 
$p$-embedded subgroup (see \cite[Proposition A.7(c)]{AKO}), so $\Psi$ 
is an isomorphism, and $R/Z\in\EE[*]{\calf/Z}$ by (a), (b), and (c) again.

Point (a) is clear, since $(R/Z)^{\calf/Z}=\{P/Z\,|\,P\in R^\calf\}$, and 
$N_{S/Z}(P/Z)=N_S(P)/Z$ whenever $Z\le P\le S$. 

The natural map $\Psi\:\autf(R)\too\Aut_{\calf/Z}(R/Z)$ is surjective by 
definition of $\calf/Z$. If $[\alpha]\in\Ker(\Psi)$, where $[\alpha]$ is 
the class of $\alpha\in\autf(R)$, then for some $x\in R$, $\alpha c_x^R$ 
induces the identity on $R/Z$ and (since $Z\le Z(\calf)$) the identity on 
$Z$, and hence has $p$-power order by Lemma \ref{l:P0<..<P}. So 
$\Ker(\Psi)$ is a $p$-group, proving (b).

By (a), it suffices to prove (c) when $R\in\calf^f$ and 
$R/Z\in(\calf/Z)^f$. 
Assume $R/Z\in(\calf/Z)^c$. Then $C_S(R)/Z\le C_{S/Z}(R/Z)\le R/Z$, so 
$R\in\calf^c$. For each $[\alpha]\in\Ker(\Psi)$, the class of 
$\alpha\in\autf(R)$, we have $[\alpha]\in O_p(\outf(R))\le\Out_S(R)$, so 
$\alpha=c_x^R$ for some $x\in N_S(R)$ such that $c_x^R\in\Aut(R)$ induces an 
inner automorphism on $R/Z$. Hence $xZ\in(R/Z)C_{S/Z}(R/Z)$, so $xZ\in R/Z$ 
since $R/Z\in(\calf/Z)^c$, and $x\in R$. Thus $\alpha\in\Inn(R)$, and 
$\Psi$ is an isomorphism in this case.


Conversely, assume $R\in\calf^c$ and $\Psi$ is an isomorphism, and let 
$y\in N_S(R)$ be such that $yZ\in C_{S/Z}(R/Z)$. Then $[y,R]\le Z$, so 
$[c_y^R]\in\Ker(\Psi)=1$. So $c_y^R\in\Inn(R)$, and $y\in RC_S(R)=R$ since 
$R$ is $\calf$-centric. This shows that $C_{S/Z}(R/Z)\le R/Z$ and hence 
$R/Z\in(\calf/Z)^c$, finishing the proof of (c). 
\end{proof}

If $\calf$ is a saturated fusion system over $S$ and $P\le Q\le S$, 
then $P\nsg\calf$ and $Q\nsg\calf$ implies $Q/P\nsg\calf/P$: this follows 
easily from the definitions. However, $P\nsg\calf$ and 
$Q/P\nsg\calf/P$ need not imply that $Q\nsg\calf$, as is seen by the 
following example. Let $p$ be any prime, set $G=C_p\wr\Sigma_p$ (wreath 
product), fix $S\in\sylp{G}$ (so $S\cong C_p\wr C_p$), and set 
$\calf=\calf_S(G)$. Set $P=O_p(G)\cong E_{p^p}$. Then $P\nsg\calf$ and 
$S/P\nsg\calf/P$, but $S$ is not normal in $\calf$.

In the following lemma, we give two conditions under which $P\nsg\calf$ 
and $Q/P\nsg\calf/P$ does imply that $Q\nsg\calf$.

\begin{Lem} \label{l:Q/P<|F/P}
Let $\calf$ be a saturated fusion system over a finite $p$-group $S$, and 
let $P\le Q\le S$ be such that $P\nsg\calf$ and $Q/P\nsg\calf/P$. If 
$Q$ is abelian, or if $P\le Z(\calf)$, then $Q\nsg\calf$. 
\end{Lem}

\begin{proof} Since $Q/P$ is normal, it is strongly closed in 
$\calf/P$, and hence $Q$ is strongly closed in $\calf$. So if $Q$ is 
abelian, then it is normal by Lemma \ref{l:scl-nm}. If $P\le Z(\calf)$, 
then $Q$ is contained in all $\calf$-essential subgroups by Lemma 
\ref{l:(F/Z)e} and since $Q/P$ is contained in all $\calf/P$-essential 
subgroups (Lemma \ref{l:wcl-nm}), and so $Q\nsg\calf$ by Lemma 
\ref{l:wcl-nm} again.
\end{proof}


\section{General lemmas}
\label{s:general}

As noted in the introduction, in our general setting, we want to 
analyze a saturated fusion system $\calf$ over a finite $p$-group $S$ 
with an abelian subgroup $A\le S$ and $\Gamma=\autf(A)$, where the 
group $A$ and the action of $O^{p'}(\Gamma)$ are given. In this 
section, we give some of the tools that will be used in Sections 
\ref{s:M12} and \ref{s:M10-11} to do this.

In practice, we don't get very far without knowing that the subgroup 
$A$ is normal in $S$ and weakly closed in $\calf$, and this should 
perhaps be included in our general assumptions. But in many cases, it 
follows easily from the weaker assumptions on $A$ and $O^{p'}(\Gamma)$.

\begin{Lem} \label{l:A<|S}
Let $\calf$ be a saturated fusion system over a finite $p$-group $S$, 
and let $A\le S$ be such that no member of $A^\calf\sminus\{A\}$ is 
contained in $N_S(A)$. Then $A$ is weakly closed in $\calf$. 
\end{Lem}

\begin{proof} Assume otherwise: then $S>N_S(A)$, and hence 
$N_S(N_S(A))>N_S(A)$. Choose $x\in N_S(N_S(A))\sminus N_S(A)$. Then $\9x\!A\ne 
A$, contradicting the  assumption that $A$ not be $S$-conjugate to any 
other subgroup of $N_S(A)$. 
\end{proof}

The importance of $A$ being weakly closed in our general situation is 
illustrated by the following lemma.

\begin{Lem} \label{l:A-w.cl.}
Let $\calf$ be a saturated fusion system over a finite $p$-group $S$, and 
assume $A\nsg S$ is an abelian subgroup that is weakly closed in 
$\calf$. 
\begin{enuma} 

\item If $R\in\calf^f$, and $R\in Q^\calf$ for some $Q\le A$, then 
$R\le A$. 

\item For each $P,Q\le A$, $\homf(P,Q)=\Hom_{N_\calf(A)}(P,Q)$. Hence each 
$\varphi\in\homf(P,Q)$ extends to some $\4\varphi\in\autf(A)$.

\item No element of $C_S(A)\sminus A$ is $\calf$-conjugate to any element 
of $A$. 

\end{enuma}
\end{Lem}

\begin{proof} \textbf{(a) }  Assume $Q\le A$ and $R\le S$ are 
$\calf$-conjugate and $R\in\calf^f$. By the extension 
axiom, each $\psi\in\isof(Q,R)$ extends to some $\4\psi\in\homf(C_S(Q),S)$. 
Then $C_S(Q)\ge A$ since $A$ is abelian, $\4\psi(A)=A$ since $A$ is weakly 
closed in $\calf$, and so $R=\4\psi(Q)\le A$. 

\smallskip


\noindent\textbf{(b) } Assume $P,Q\le A$ and $\varphi\in\homf(P,Q)$, and 
choose $R\in P^\calf$ that is fully centralized in $\calf$. Then $R\le A$ 
by (a), and there is $\psi\in\isof(\varphi(P),R)$. By the extension axiom again, 
$\psi$ extends to $\5\psi\in\homf(A,S)$ and $\psi\varphi$ extends to 
$\5\varphi\in\homf(A,S)$, and $\5\psi(A)=A=\5\varphi(A)$ since $A$ is 
weakly closed. Then $\5\psi^{-1}\5\varphi\in\autf(A)$, and 
$(\5\psi^{-1}\5\varphi)|_P=\psi^{-1}(\psi\varphi)=\varphi$.

\smallskip

\noindent\textbf{(c) } Assume $x\in C_S(A)\sminus A$ is $\calf$-conjugate 
to $y\in A$. By (a), we can arrange that $\gen{y}\in\calf^f$, so 
by Lemma \ref{p:Hom(NSP,S)}, there is 
$\varphi\in\homf(N_S(\gen{x}),S)$ such that $\varphi(x)=y$. But $A\le 
N_S(\gen{x})$, $\varphi(A)=A$ since $A$ is weakly closed, and this is 
impossible since $\varphi(x)\in A$ and $x\notin A$. So no element in 
$C_S(A)\sminus A$ is $\calf$-conjugate to any element of $A$. 
\end{proof}

In many of the cases we want to consider, the assumptions we choose on $A$ 
and on $\Gamma$ imply that $O^{p'}(\calf)$ is simple (see, e.g., 
\cite[Definition I.6.1]{AKO}). For example, if $\calf$ is a saturated 
fusion system over $S$, and $A\nsg S$ is such that $C_S(A)=A$, and we set 
$\Gamma=\autf(A)$ and $\Gamma_0=O^{p'}(\Gamma)$, and assume also that 
$\Omega_1(A)$ is a simple $\F_p\GG$-module and $\Gamma_0/O_{p'}(\Gamma_0)$ 
is a simple group (and $\Gamma_0\ncong C_p$), then either $A\nsg\calf$ or 
the fusion system $O^{p'}(\calf)$ is simple. However, this will not be 
needed, and before proving it here, we would first have to define normal 
fusion subsystems.


\newsubb{Proving that 
\texorpdfstring{$\calf=\gen{N_\calf(A),C_\calf(Z)}$}{F=<NF(A),CF(Z)>}}
{s:<N,C>}

When analyzing fusion systems in our setting, we first check whether 
$\calf=\gen{N_\calf(A),C_\calf(Z)}$ for some choice of $Z\le Z(S)$. 
The following lemma will be our tool for doing this.

\begin{Prop} \label{p:R=C(Z)}
Let $\calf$ be a saturated fusion system over a finite $p$-group $S$, 
let $A\nsg S$ be an abelian subgroup that is weakly closed in $\calf$, 
and fix $1\ne Z\le Z(S)\cap A$. Then either 
$\calf=\gen{C_\calf(Z),N_\calf(A)}$, or there are $R\in\EE\calf$ and 
$\alpha\in\autf(R)$ such that $\alpha$ is 
not a morphism in $\gen{C_\calf(Z),N_\calf(A)}$, and such that 
$\alpha(Z)\nleq A$, $\alpha(Z)\in 
N_\calf(A)^f$, and $R=C_S(\alpha(Z))=N_S(\alpha(Z))$. 
\end{Prop}

\begin{proof} Set $\calf_0=\gen{C_\calf(Z),N_\calf(A)}$: the smallest 
fusion system over $S$ (not necessarily saturated) that contains both 
$C_\calf(Z)$ and $N_\calf(A)$. We first claim that 
	\beqq N_\calf(Z) \le \gen{C_\calf(Z),\autf(S)} \le \calf_0. 
	\label{e:NFZ<F0} \eeqq
The second inclusion is clear: $\autf(S)=\Aut_{N_\calf(A)}(S)$ since 
$A$ is weakly closed in $\calf$ by assumption. If 
$\varphi\in\Hom_{N_\calf(Z)}(P,Q)$, where $P,Q\ge Z$, then since $S=C_S(Z)$, 
$\varphi|_Z\in\autf(Z)$ extends to some $\alpha\in\autf(S)$ by the 
extension axiom, and $\varphi=\alpha\circ(\alpha^{-1}\varphi)$ where 
$\alpha^{-1}\varphi\in\Hom_{C_\calf(Z)}(P,S)$. This proves the first 
inclusion in \eqref{e:NFZ<F0}.

By Lemma \ref{p:Hom(NSP,S)} and since $Z\le Z(S)$ is fully normalized in 
$\calf$, for each $X\in Z^\calf$, there is $\psi_X\in\homf(N_S(X),S)$ such 
that $\psi_X(X)=Z$. Set 
	\[ \calz = \{ X\in Z^\calf \,|\, \psi_X\in\Mor(\calf_0) 
	\}. \]
If $\psi'\in\homf(N_S(X),S)$ is another morphism such that 
$\psi'(X)=Z$, then $\psi'\circ\psi_X^{-1}\in\Mor(N_\calf(Z))$, and 
hence $\psi'\in\Mor(\calf_0)$ if and only if $\psi_X\in\Mor(\calf_0)$ 
by \eqref{e:NFZ<F0}. So $\calz$ is independent of the choices of the 
$\psi_X$. 

If $X\in Z^\calf$ and $X\le A$, then $A\le N_S(X)$ and 
$\psi_X(A)=A$, so $\psi_X\in\Mor(\calf_0)$. Thus 
	\beqq X\in Z^\calf ~\textup{and}~ X\le A \quad\implies\quad 
	X\in\calz. \label{e:in_Z} \eeqq

If $\varphi\in\homf(P,S)$ is such that $P\ge Z$ and 
$X=\varphi(Z)\in\calz$, then $\varphi(P)\le C_S(X)$ since $P\le S=C_S(Z)$, so 
$\psi_X\circ\varphi$ is defined and in $N_\calf(Z)\le\calf_0$, and 
hence $\varphi=(\psi_X|_{\varphi(P)})^{-1}\circ(\psi_X\circ\varphi)$ 
is also in $\calf_0$. Thus 
	\beqq \textup{for each $\varphi\in\homf(P,S)$ with $Z\le P\le S$,} 
	\quad \varphi(Z)\in\calz \implies \varphi\in\Mor(\calf_0). 
	\label{e:in_X} \eeqq

Assume $\calf>\calf_0$. By Theorem \ref{t:AFT} (the Alperin-Goldschmidt fusion 
theorem), there are $R\in\EE\calf\cup\{S\}$ and $\alpha\in\autf(R)$ 
such that $\alpha\notin\Mor(\calf_0)$. Since 
$\autf(S)=\Aut_{\calf_0}(S)$ by \eqref{e:NFZ<F0}, we have 
$R\in\EE\calf$. Choose such $R$ and $\alpha$ with $|R|$ maximal. Since 
$R$ is $\calf$-centric, we have $R\ge Z(S)\ge Z$. Set $X=\alpha(Z)$; 
then $X\notin\calz$ by \eqref{e:in_X}, and hence $X\nleq A$ by 
\eqref{e:in_Z}. Also, $R\le C_S(X)\le N_S(X)$ since $R\le C_S(Z)=S$.

For each $Y\in Z^\calf\sminus\calz$, we have 
$\psi_Y\notin\Mor(\calf_0)$ by definition of $\calz$. Hence $\psi_Y$ is 
a composite of restrictions of automorphisms of members of 
$\EE\calf\cup\{S\}$ of order at least $|N_S(Y)|$, and at least one of 
these automorphisms is not in $\calf_0$. So by the maximality 
assumption on $R$, $|R|\ge|N_S(Y)|$ for all $Y\in Z^\calf\sminus\calz$, 
and in particular, for all $Y\in X^{N_\calf(A)}$. Since $R\le N_S(X)$, 
this shows that $X$ is fully normalized in $N_\calf(A)$, and also that 
$R=C_S(X)=N_S(X)$. 
\end{proof}

Note in particular the following special case of Proposition 
\ref{p:R=C(Z)}.

\begin{Cor} \label{c:R=C(Z)}
Let $\calf$ be a saturated fusion system over a finite $p$-group $S$, 
let $A\nsg S$ be an abelian subgroup that is weakly closed in $\calf$, 
and fix $1\ne Z\le Z(S)\cap A$. Assume that $A\nsg C_\calf(Z)$ but 
$A\nnsg\calf$. Then there are $R\in\EE\calf$ and $\alpha\in\autf(R)$ 
such that $\alpha(Z)\nleq A$, $\alpha(Z)\in N_\calf(A)^f$, and 
$R=C_S(\alpha(Z))=N_S(\alpha(Z))$. 
\end{Cor}

\begin{proof} By assumption, $C_\calf(Z)\le N_\calf(A)<\calf$. So 
$\gen{C_\calf(Z),N_\calf(A)}\ne\calf$, and the result follows from 
Proposition \ref{p:R=C(Z)}. 
\end{proof}


\newsubb{Normality of subgroups}
{s:Anot<|F}

The results in this subsection will be useful when showing that certain 
subgroups, especially abelian subgroups, are strongly closed or normal in 
a fusion system.

\begin{Lem} \label{l:nnorm}
Let $\calf$ be a saturated fusion system over a finite $p$-group $S$, 
and let $Q\nsg S$ be a normal subgroup that is not weakly closed in 
$\calf$. Then there are $P\in Q^\calf\sminus\{Q\}$, 
$R\in\EE\calf\cup\{S\}$, and $\alpha\in\autf(R)$ such that $R\ge Q$, 
$P=\alpha(Q)$, $R=N_S(P)$, $P\in N_\calf(Q)^f$, and $|R|\ge|N_S(U)|$ 
for all $U\in Q^\calf\sminus\{Q\}$. 
\end{Lem}

\begin{proof} 
Let $\scrw$ be the set of pairs $(R,\alpha)$ where 
$R\in\EE\calf\cup\{S\}$, $R\ge Q$, $\alpha\in\autf(R)$, and 
$\alpha(Q)\ne Q$. Since $Q$ is not weakly closed in $\calf$, there is 
$\varphi\in\homf(Q,S)$ such that $\varphi(Q)\ne Q$, and hence 
$\scrw\ne\emptyset$ by the Alperin-Goldschmidt fusion theorem (Theorem \ref{t:AFT}). 

Choose $(R,\alpha)\in\scrw$ such that $|R|$ is maximal. By Lemma 
\ref{p:Hom(NSP,S)}, for each $U\in Q^\calf\sminus\{Q\}$, there is a 
morphism $\varphi\in\homf(N_S(U),S)$ such that $\varphi(U)=Q$. By 
Theorem \ref{t:AFT} again, there is $(R_1,\alpha_1)\in\scrw$ such 
that $|R_1|\ge|N_S(U)|$, and $|R|\ge|R_1|$ by the maximality of $|R|$. 
Thus $|R|\ge|N_S(U)|$ for each $U\in Q^\calf\sminus\{Q\}$.

Now set $P=\alpha(Q)$. Then $P\nsg R$ since $Q\nsg R$, so $R\le 
N_S(P)$, with equality since we just saw $|R|\ge|N_S(P)|$. Also, 
$P\in N_\calf(Q)^f$ since $|R|\ge|N_S(U)|$ for each $U\in 
Q^\calf\sminus\{Q\}\supseteq P^{N_\calf(Q)}$. 
\end{proof}

The following is a more technical result that will be needed when proving 
that $Q/Z\nsg C_\calf(Z)/Z$ in case (i) of Theorem \ref{ThA}. 

\begin{Prop} \label{p:not.str.cl.1}
Let $\calf$ be a saturated fusion system over a finite $p$-group $S$, and 
let $A\nsg S$ be an abelian subgroup that is weakly closed in $\calf$ but 
not normal. Let $1=A_0<A_1<\dots<A_m=A$ be such that $[S,A_i]\le 
A_{i-1}$ for each $1\le i\le m$. Set $\cale_0=\calf$, and for each 
$1\le i\le m$, set $\4A_i=A_i/A_{i-1}$ and 
$\cale_i=C_{\cale_{i-1}}(\4A_i)/\4A_i$, regarded as a fusion system 
over $S/A_i$. (Note that $\4A_i\le Z(S/A_{i-1})$.)
Then there are $0\le \ell\le m-2$, $R\le S$, and 
$\alpha\in\Aut_{\calf}(R)$, such that 
\begin{itemize} 
\item $R\ge A_{\ell+1}$, $[\alpha,A_i]\le A_{i-1}$ for $1\le i\le\ell$, and 
$X\defeq\alpha(A_{\ell+1})\nleq A$; 
\item $R=N_S(X)$, $R/A_\ell=C_{S/A_\ell}(X/A_\ell)$, 
and $X/A_\ell\in N_{\cale_\ell}(A/A_\ell)^f$; and 
\item $R/A_\ell\in\EE[*]{\cale_\ell}$. 
\end{itemize}
\end{Prop}

\begin{proof} The fusion systems $\cale_i$ are all saturated by Theorem 
\ref{t:NF(Q)} and \cite[Proposition II.5.11]{Craven}, applied 
iteratively. Also, $A/A_{m-1}$ is weakly closed in $\cale_{m-1}$ since 
$A$ is weakly closed in $\calf$. All $\cale_{m-1}$-essential subgroups 
contain $Z(S/A_{m-1})\ge A/A_{m-1}$ since they are centric, so 
$A/A_{m-1}\nsg\cale_{m-1}$ by Lemma \ref{l:wcl-nm}. Since 
$A\nnsg\cale_0=\calf$ by assumption, there is $0\le \ell\le m-2$ such 
that $A/A_\ell\nnsg\cale_\ell$ and $A/A_{\ell+1}\nsg\cale_{\ell+1}$. 

We now apply Corollary \ref{c:R=C(Z)}, with $A/A_\ell$, $A_{\ell+1}/A_\ell$, and 
$\cale_\ell$ in the role of $A$, $Z$, and $\calf$. Here, 
$A_{\ell+1}/A_\ell\le Z(S/A_\ell)$ since $[A_{\ell+1},S]\le A_\ell$, while 
$A/A_\ell\nnsg\cale_\ell$ by assumption. Since $A/A_\ell$ is abelian, it 
is normal in $C_{\cale_\ell}(\4A_{\ell+1})$ by Lemma \ref{l:Q/P<|F/P} and since 
$A/A_{\ell+1}\nsg\cale_{\ell+1}=C_{\cale_\ell}(\4A_{\ell+1})/\4A_{\ell+1}$. So 
by Corollary \ref{c:R=C(Z)}, there are $R\le S$ containing 
$A_\ell$, and $\4\alpha\in\Aut_{\cale_\ell}(R/A_\ell)$, such that 
$R/A_\ell=C_{S/A_\ell}(\4\alpha(\4A_{\ell+1}))\in\EE[*]{\cale_\ell}$, and 
	\beqq X/A_\ell\defeq\4\alpha(\4A_{\ell+1})\nleq A/A_\ell, \quad 
	R/A_\ell=N_{S/A_\ell}(X/A_\ell), 
	\quad\textup{and}\quad X/A_\ell\in N_{\cale_\ell}(A/A_\ell)^f. 
	\label{e:n.s.cl.1} \eeqq
Also, $R/A_\ell\ge Z(S/A_\ell)\ge\4A_{\ell+1}$ since $R/A_\ell$ is 
$\cale_\ell$-centric, so $R\ge A_{\ell+1}$.

Set $\alpha_\ell=\4\alpha$, and choose 
$\alpha_i\in\Aut_{C_{\cale_i}(\4A_{i+1})}(R/A_i)\le 
\Aut_{\cale_i}(R/A_i)$ for decreasing indices $i=\ell-1,\ell-2,\dots,0$ so that 
$\alpha_i/\4A_{i+1}=\alpha_{i+1}$ for each $i<\ell$. Set 
$\alpha=\alpha_0\in\autf(R)$; then $[\alpha,A_i]\le 
A_{i-1}$ for each $i$ by by definition of the $\cale_i$, and 
$X=\alpha(A_{\ell+1})\nleq A$ since 
$X/A_\ell=\4\alpha(\4A_{\ell+1})\nleq A/A_\ell$. The other claims 
listed in the proposition follow easily from \eqref{e:n.s.cl.1}.
\end{proof}


\newsubb{Equalities between fusion systems}
{s:F1=F2}


We finish the section with two sets of conditions for showing that two fusion 
systems over the same $p$-group are equal. Proposition \ref{p:F1=F2-1} 
will be applied to the fusion systems encountered in Section 
\ref{s:M12}, and Proposition \ref{p:F1=F2-2} to those in Section 
\ref{s:M10-11}.

\begin{Prop} \label{p:F1=F2-1}
Let $\calf_1\ge\cale\le\calf_2$ be saturated fusion systems over a finite 
$p$-group $S$. Assume that $Q\nsg S$ is centric and normal in all three, 
and that $\autf[1](Q)=\autf[2](Q)$. Assume also that the homomorphism 
	\[ H^1(\outf[1](Q);Z(Q)) \Right4{} H^1(\Out_\cale(Q);Z(Q)) \]
induced by restriction is surjective. Then $\calf_1=\calf_2$. 
\end{Prop}

\begin{proof} Let $M_1\ge H\le M_2$ be models for 
$\calf_1\ge\cale\le\calf_2$ (Definition \ref{d:model}), where $S\le H$ 
is a Sylow $p$-subgroup of all three. Thus $M_1$ and $M_2$ are both 
extensions of $Q$ by $\outf[1](Q)=\outf[2](Q)$, and the difference of 
the two extensions (up to isomorphism) is represented by an element 
$\chi\in H^2(\outf[1](Q);Z(Q))$ (see \cite[Theorem IV.8.8]{MacL}). 
Also, $\chi$ vanishes after restriction to $H^2(\Out_\cale(Q);Z(Q))$ 
since $M_1$ and $M_2$ both contain $H$, so $\chi=0$ since 
$\Out_\cale(Q)$ has index prime to $p$ in $\outf[1](Q)$. Thus there is 
an isomorphism $\psi\:M_1\too M_2$ such that $\psi|_Q=\Id_Q$. Note that 
$\psi$ also induces the identity on $H/Q$ and on $S/Q$ since they 
inject into $\Aut(Q)$, but need not induce the identity on $S$. 

Set $\psi_0=\psi|_H\in\Aut(H)$. Consider the commutative diagram
	\[ \vcenter{\xymatrix@C=30pt{
	H^1(M_1/Q;Z(Q)) \ar[r]^-{\eta_1}_-{\cong} \ar[d]_{\rho_1} 
	& C_{\Aut(M_1)}(Q)/\Aut_{Z(Q)}(M_1) \ar[d]_{\rho_2} \\
	H^1(H/Q;Z(Q)) \ar[r]^-{\eta_2}_-{\cong} & 
	C_{\Aut(H)}(Q)/\Aut_{Z(Q)}(H)
	}} \] 
where $\eta_1,\eta_2$ are defined as in \cite[Lemma 1.2]{OV2}. Since 
$\rho_1$ is surjective by assumption, $\rho_2$ is also surjective. So 
there is $\alpha\in\Aut(M_1)$ such that $\alpha|_H=\psi_0c_z|_H$ for 
some $z\in Z(Q)$, and upon replacing $\alpha$ by $\alpha c_z^{-1}$, we 
can arrange that $\alpha|_H=\psi_0$. 

Now set $\varphi=\psi\alpha^{-1}\:M_1\xto{~\cong~}M_2$. Then 
$\varphi|_H=\psi_0\psi_0^{-1}=\Id_H$, and in particular, 
$\varphi|_S=\Id_S$. Since $M_1$ and $M_2$ are models for $\calf_1$ and 
$\calf_2$, we conclude that $\calf_1=\calf_2$. 
\end{proof}

The other criterion we give for two fusion systems to be equal applies 
only to fusion systems satisfying some very restrictive hypotheses, 
which are stated separately for easier reference.

\begin{Hyp} \label{h:F1=F2-2}
Let $\calf$ be a saturated fusion system over a finite $p$-group $S$. 
Assume $A,Q\nsg S$ are such that 
\begin{enumi} 
\item $\EE\calf=\{A,Q\}$; 
\item $A$ is abelian, $S=AQ$, and $C_S(A\cap Q)=A$; and 
\item $p\nmid\bigl|N_{\Aut(A)}(O^{p'}(\autf(A)))\big/O^{p'}(\autf(A))\bigr|$. 
\end{enumi}
\end{Hyp}

Note that $\calf=N_\calf(R)$ if $\EE\calf=\{R\}$ has order $1$, while 
$\calf=N_\calf(S)$ if $\EE\calf=\emptyset$. So the next proposition 
still holds if we assume $\EE\calf\subseteq\{A,Q\}$ instead of assuming 
equality. However, since the extra 
cases that would be added are rather trivial and will not be 
encountered in this paper, we decided to use the more restrictive 
version.

\begin{Prop} \label{p:F1=F2-2}
Let $\calf_1$ and $\calf_2$ be two saturated fusion systems over the same 
finite $p$-group $S$, and let $A,Q\nsg S$ be normal subgroups with 
respect to which Hypotheses \ref{h:F1=F2-2} hold for $\calf_1$ and 
for $\calf_2$. Assume also that 
$O^{p'}(N_{\calf_1}(A))=O^{p'}(N_{\calf_2}(A))$ and 
$O^{p'}(\Aut_{\calf_1}(Q))=O^{p'}(\Aut_{\calf_2}(Q))$. 
Then $O^{p'}(\calf_1)=O^{p'}(\calf_2)$. 
\end{Prop}

\begin{proof} If Hypotheses \ref{h:F1=F2-2} hold for $\calf_i$ ($i=1,2$), 
then they also hold for $O^{p'}(\calf_i)$ (note in particular that 
$\EE[*]{O^{p'}(\calf_i)}=\EE{\calf_i}$ by Proposition \ref{p:EOp'(F)}). So 
it suffices to prove the proposition when $\calf_i=O^{p'}(\calf_i)$ for 
$i=1,2$.

Since $S=AQ$ where $A$ and $Q$ are both properly contained in $S$, we have 
$Q\ngeq A$ and $A\ngeq Q$. Note that $Q$ is nonabelian, since otherwise 
$C_S(A\cap Q)=S$, contradicting \ref{h:F1=F2-2}(ii). Also, $A$ and $Q$ are 
weakly closed in $\calf_i$ for $i=1,2$, since otherwise, there would be 
$\alpha\in\Aut_{\calf_i}(S)$ with $\alpha(A)\ne A$ or $\alpha(Q)\ne Q$, 
which is impossible since $\alpha$ permutes the members of $\EE{\calf_i}$.

Set 
	\[ \Theta = \gen{\autf[1](S),\autf[2](S)} \le \Aut(S). \]
Fix $R\in\{A,Q\}$. Each element of $\Theta$ normalizes $R$ since $R$ is 
weakly closed in $\calf_1$ and in $\calf_2$. 
For each $\alpha\in\Theta$ such that $\alpha|_R=\Id_R$, $\alpha$ 
also induces the identity on $S/R$ since $C_S(R)\le R$ (since 
$R\in\EE{\calf_i}$ by 
\ref{h:F1=F2-2}(i)), and hence $\alpha$ has $p$-power order. Thus 
	\beqq \bigl\{\alpha\in\Theta \,\big|\, \alpha|_R=\Id_R \bigr\} 
	\le O_p(\Theta) \quad 
	\textup{(for $R\in\{A,Q\}$)}: \label{e:Op(Theta)} \eeqq
this subgroup is normal in $\Theta$ since all elements in $\Theta$ 
normalize $R$. 

By points (i) and (ii) in Hypotheses \ref{h:F1=F2-2} and since $A$ and $Q$ 
are weakly closed, the conclusions 
of Lemma \ref{p:<Op'NR>} hold for $\calf_1$ and $\calf_2$. (Note 
that $Q\cap A\notin\calf^c$ since it is strictly contained in the 
abelian group $A$.) By Lemma 
\ref{p:<Op'NR>}(b) and since $O^{p'}(\calf_i)=\calf_i$ for $i=1,2$ 
by assumption, 
	\beqq \Aut_{\calf_i}(S) = 
	\Gen{\Aut_{O^{p'}(N_{\calf_i}(A))}(S),\Aut_{O^{p'}(N_{\calf_i}(Q))}(S)} 
	\label{e:AutFi(S)1} \eeqq
for $i=1,2$. 

Again fix $R\in\{A,Q\}$. If $\alpha\in\Aut_{O^{p'}(N_{\calf_1}(R))}(S)$, 
then $\alpha|_R\in O^{p'}(\Aut_{\calf_1}(R))=O^{p'}(\Aut_{\calf_2}(R))$ by 
Lemma \ref{p:<Op'NR>}(a), so 
$\alpha|_R=\beta|_R$ for some $\beta\in\Aut_{\calf_2}(S)$ by the extension 
axiom and since $\alpha|_R$ is normalized by $\Aut_S(R)$. By Lemma 
\ref{p:<Op'NR>}(a) again, $\beta\in\Aut_{O^{p'}(N_{\calf_2}(R))}(S)$. 
Also, $\alpha^{-1}\beta\in O_p(\Theta)$ by \eqref{e:Op(Theta)} and since 
$\alpha|_R=\beta|_R$. Upon repeating this argument with the roles of 
$\calf_1$ and $\calf_2$ exchanged, we have shown that 
	\[ \Aut_{O^{p'}(N_{\calf_1}(R))}(S)O_p(\Theta) = 
	\Aut_{O^{p'}(N_{\calf_2}(R))}(S)O_p(\Theta). \] 
Together with \eqref{e:AutFi(S)1}, this implies that 
	\beqq \autf[1](S)O_p(\Theta) = \autf[2](S)O_p(\Theta). 
	\label{e:AutFi(S)2} \eeqq

For $R\in\{A,Q\}$, set 
	\[ \Gamma^{(R)}=O^{p'}(\autf[1](R))=O^{p'}(\autf[2](R)), \]
where the last two groups are equal by assumption. Then for $i=1,2$, 
	\beqq \autf[i](R) = \Gamma^{(R)} \cdot 
	\{ \alpha|_R \,|\, \alpha\in\autf[i](S) \} \label{e:Gamma1} \eeqq
by the Frattini argument and the extension axiom (and since $R\nsg S$). 

Set $\Theta^{(A)}=\gen{\autf[1](A),\autf[2](A)}$. Then 
$\Gamma^{(A)}\nsg\Theta^{(A)}$ since it is normal in each 
$\autf[i](A)$. Since $N_{\Aut(A)}(\Gamma^{(A)})/\Gamma^{(A)}$ has 
order prime to $p$ by \ref{h:F1=F2-2}(iii), we have 
$O^{p'}(\Theta^{(A)})=O^{p'}(\Gamma^{(A)})=\Gamma^{(A)}$. By \eqref{e:Gamma1}, 
for each $\alpha\in\autf[1](A)$, there are 
$\alpha_0\in\Gamma^{(A)}$ and 
$\5\alpha\in\autf[1](S)$ such that $\alpha=\alpha_0(\5\alpha|_A)$. 
By \eqref{e:AutFi(S)2}, there 
is $\5\beta\in\autf[2](S)$ such that $\5\alpha^{-1}\5\beta\in O_p(\Theta)$. 
Set $\beta=\alpha_0(\5\beta|_A)\in\autf[2](A)$. Then $\alpha^{-1}\beta = 
(\5\alpha^{-1}\5\beta)|_A$ has $p$-power order, hence lies in 
$O^{p'}(\Theta^{(A)})=\Gamma^{(A)}$, and we have shown that 
$\autf[1](A)\le\autf[2](A)$. A similar argument proves the opposite 
inclusion, and thus 
	\beqq \autf[1](A) = \autf[2](A). \label{e:AutFi(A)} \eeqq

For $i=1,2$, 
	\begin{align*} 
	\Aut_{\calf_i}(Q) &= \Gamma^{(Q)} \cdot 
	\bigl\{\alpha|_Q \,\big|\, \alpha\in\Aut_{\calf_i}(S) \bigr\} \\
	&= \Gamma^{(Q)} \cdot \Gen{ \bigl\{ \alpha|_Q \,\big|\, 
	\alpha\in\Aut_{O^{p'}(N_{\calf_i}(Q))}(S) \bigr\}
	\,,\, \bigl\{ \alpha|_Q \,\big|\, 
	\alpha\in\Aut_{O^{p'}(N_{\calf_i}(A))}(S) \bigr\} } \\
	&\le \Gamma^{(Q)} \cdot \Gen{ \Aut_{O^{p'}(N_{\calf_i}(Q))}(Q) 
	\,,\, \bigl\{ \alpha|_Q \,\big|\, 
	\alpha\in\Aut_{O^{p'}(N_{\calf_i}(A))}(S) \bigr\} } \\
	&= \Gamma^{(Q)} \cdot \bigl\{ \alpha|_Q \,\big|\, 
	\alpha\in\Aut_{O^{p'}(N_{\calf_i}(A))}(S) \bigr\} :
	\end{align*}
the first equality by \eqref{e:Gamma1}, the second by \eqref{e:AutFi(S)1}, 
and the last since $\Aut_{O^{p'}(N_{\calf_i}(Q))}(Q) =\Gamma^{(Q)}$ by 
Lemma \ref{p:<Op'NR>}(a). The opposite inclusion is clear, so 
	\beqq \Aut_{\calf_1}(Q)=\Aut_{\calf_2}(Q) \label{e:AutFi(Q)} \eeqq 
since $O^{p'}(N_{\calf_1}(A))= O^{p'}(N_{\calf_2}(A))$ by assumption. 

For $R\in\{A,Q\}$, consider the homomorphism 
	\[ \Theta = \gen{\Aut_{\calf_1}(S),\Aut_{\calf_2}(S)} 
	\Right4{\Psi_R} N_{\Aut_{\calf_1}(R)}(\Aut_S(R)) = 
	N_{\Aut_{\calf_2}(R)}(\Aut_S(R)), \]
where $\autf[1](R)=\autf[2](R)$ by \eqref{e:AutFi(A)} or 
\eqref{e:AutFi(Q)}, and where $\Psi_R$ is induced by restriction to $R$ and 
is surjective by the extension axiom. Hence $\Psi_R$ sends $O_p(\Theta)$ 
into the group $O_p(N_{\Aut_{\calf_i}(R)}(\Aut_S(R)))=\Aut_S(R)$. So for 
each $\beta\in O_p(\Theta)$, there are $g,h\in S$ such that 
$\beta|_A=c_h^A$ and $\beta|_Q=c_g^Q$. Then $\beta(c_g^S)^{-1}$ is the 
identity on $Q$ and conjugation by $hg^{-1}$ after restriction to $A$, so 
$hg^{-1}\in C_S(Q\cap A)=A$ by \ref{h:F1=F2-2}(ii), and 
$\beta(c_g^S)^{-1}|_A=\Id$. Since $S=AQ$ by \ref{h:F1=F2-2}(ii), this 
shows that $\beta=c_g^S$, and hence that $O_p(\Theta)=\Inn(S)$. So 
$\Aut_{\calf_1}(S)=\Aut_{\calf_2}(S)$ by \eqref{e:AutFi(S)2}. Since 
$\EE[*]{\calf_i}=\{A,Q\}$ by \ref{h:F1=F2-2}(i), this together with 
\eqref{e:AutFi(A)} and \eqref{e:AutFi(Q)} (and Theorem \ref{t:AFT}) shows 
that 
	\beq \calf_1 = \gen{\autf[1](S),\autf[1](A),\autf[1](Q)} 
	= \gen{\autf[2](S),\autf[2](A),\autf[2](Q)} = \calf_2. \qedhere 
	\eeq
\end{proof}


\section{Todd modules in characteristic 3}
\label{s:Todd}

We describe here the notation we use in Sections \ref{s:M12} and 
\ref{s:M10-11} to make computations involving Todd modules: first the Todd 
module for $2M_{12}$, and afterwards those for $M_{11}$ and $A_6\cong 
O^2(M_{10})$. 

\newsubb{The ternary Golay code and the group 
\texorpdfstring{$2M_{12}$}{2M12}}
{s:M12a}

We first set up notation for handling the ternary Golay code $\scrg$ and 
its automorphism group $2M_{12}$. Our notation is based on that used by 
Griess in \cite[Chapter 7]{Griess} to describe the ternary Golay code. 
We begin by fixing some very general notation for describing $n$-tuples of 
elements in a field. 

\begin{Not} \label{n:MonX(K)}
For a finite set $X=\{1,2,\dots,n\}$ and a field $K$, we regard $K^X$ as 
the vector space of maps $X\too K$, and let $\{e_i\,|\,i\in X\}$ be its 
canonical basis:
	\[ \{ e_i \,|\, i\in X\} \subseteq K^X \qquad\qquad 
	\textup{where } e_i(j)= \begin{cases} 1 & \textup{if $i=j$}\\
	0 & \textup{if $i\ne j$}
	\end{cases} \quad\textup{for $i,j\in X$.} \]
We also set $e_J=\sum_{j\in J}e_j$ for $J\subseteq X$. Let 
	\[ \Perm_X(K) \le \Mon_X(K) \le \Aut(K^X) \]
be the subgroups of permutation automorphisms and monomial automorphisms, 
respectively: automorphisms that permute the basis $\{e_i\}$ or the 
subspaces $\{Ke_i\}$, respectively. Thus if $|X|=n$, then 
$\Perm_X(K)\cong\Sigma_n$ and $\Mon_X(K)\cong K^\times\wr\Sigma_n$. Let 
	\[ \pi = \pi_{X,K} \: \Mon_X(K) \Right4{} \Perm_X(K) \]
be the canonical projection that sends a monomial automorphism to the 
corresponding permutation automorphism; thus $\Ker(\pi_{X,K})$ is the group of 
automorphisms that send each $Ke_i$ to itself. 
\end{Not}

Now set $I=\{1,2,3,4\}$, and regard $\F_3^I$ as the space of $4$-tuples of 
elements of $\F_3$ as well as that of functions $I\too\F_3$. 
Let $\scrt\subseteq\F_3^I$ be the \emph{tetracode subgroup}:
	\beqq \begin{split} 
	\scrt &= \{(a,b,b+a,b+2a)\,|\,a,b\in\F_3\} \\
	&= \bigl\{\xi\in\F_3^I \,\big|\, \xi(3)=\xi(1)+\xi(2),~ 
	\xi(4)=\xi(1)+\xi(3) \bigr\}. 
	\end{split} \label{e:tetra} \eeqq
Thus $\scrt$ is a 2-dimensional subspace of $\F_3^I$. By \cite[Lemma 
7.3]{Griess}, 
	\beqq \Aut(\scrt) \defeq 
	\bigl\{\alpha\in\Mon_I(\F_3) \,\big|\, \alpha(\scrt)=\scrt\bigr\}
	\cong\GL_2(3)\cong2\Sigma_4. \label{e:Aut(T)} \eeqq
More precisely, each linear automorphism of $\scrt$ extends to a unique 
monomial automorphism of $\F_3^I$; and each permutation of $I$ lifts to a 
monomial automorphism of $\F_3^I$, unique up to sign, that acts on $\scrt$.

Set $\Delta=\F_3\times I$, so that $\F_3^\Delta$ is a 12-dimensional 
vector space over $\F_3$. Define $C_1,C_2,C_3,C_4\in\F_3^\Delta$ by 
setting 
	\[ C_i=e_{(0,i)}+e_{(1,i)}+e_{(2,i)} \qquad \textup{for $i\in I$,} \]
and set $\scrc=\{C_i\,|\,i\in I\}$. Thus $e_\Delta=\sum_{i\in I}C_i$. Define 
	\[ \Grf\:\F_3^I \Right4{} \F_3^\Delta 
	\qquad\textup{by setting}\qquad
	\Grf(\xi) = \sum\nolimits_{i\in I}e_{(\xi(i),i)} \]
(the ``graph'' of $\xi$). 
Thus for each $(c,i)\in\Delta$, $\Grf(\xi)(c,i)=1$ if $c=\xi(i)$ and is 
zero otherwise. Finally, define $\scrg<\4\scrg<\F_3^\Delta$ by setting 
	\beqq \widebar{\scrg} = \Gen{\scrc\cup\Grf(\scrt)} 
	\qquad\textup{and}\qquad \scrg = \Gen{C_i+\Grf(\xi) \,\big|\, 
	i\in I,\xi\in\scrt}. \label{e:Golay} \eeqq
Finally, for $i,j\in I$ and $\xi\in\scrt$, we define 
	\[ C_{ij} = C_i-C_j\in\scrg \qquad\textup{and}\qquad 
	\grf\xi = \Grf(\xi)-\Grf(0) \in \scrg. \]

The $C_i$ are clearly linearly independent in $\widebar{\scrg}$. The 
relations 
	\beqq \Grf(\xi)+\Grf(\eta)+\Grf(\theta) 
	= \sum_{\substack{i\in I\\\xi(i)\ne\eta(i)}}C_i \quad 
	\textup{for all $\xi,\eta,\theta\in\scrt$ such that 
	$\xi+\eta+\theta=0$} \label{e:C-xi-relation} \eeqq
among the $C_i$ and $\Grf(\xi)$ are easily checked. 
So for any $\F_3$-basis $\{\xi_1,\xi_2\}$ of $\scrt$, 
	\begin{align*} 
	\widebar{\scrg} &= 
	\Gen{C_1,C_2,C_3,C_4,\Grf(0),\Grf(\xi_1),\Grf(\xi_2)} \\
	\scrg &= \Gen{C_{12}, C_{13}, C_{14}, \grf{\xi_1}, \grf{\xi_2}, 
	C_1+\Grf(0) }. 
	\end{align*}
These elements in each of these two sets are independent in $\F_3^\Delta$, 
and hence form bases for $\4\scrg$ and $\scrg$, respectively. So 
$\dim(\4\scrg)=7$ and $\dim(\scrg)=6$. 

The subspace $\scrg$ is the ternary Golay code. We refer to \cite[Lemmas 
7.8 \& 7.9]{Griess} for more details and more properties. Note in 
particular that $\scrg=\scrg^\perp$ under the standard inner product on 
$\F_3^\Delta$ (i.e., that for which the standard basis 
$\{e_{(c,i)}\,|\,(c,i)\in\Delta\}$ is orthonormal).

We next look at automorphisms of $\scrg$. 

\begin{Not} \label{n:transl}
\begin{enuma} 

\item Set $\twoMat12=\{ \xi\in\Mon_\Delta(\F_3) \,|\, 
\xi(\scrg)=\scrg \}$. 


\item For $\eta\in\F_3^I$, let $\trs\eta\in\Perm_\Delta(\F_3)$ be the 
translation that sends $e_{(c,i)}$ to $e_{(c+\eta(i),i)}$. Thus for 
$\xi\in\F_3^\Delta$, we have $\trs\eta(\xi)(c,i)=\xi(c-\eta(i),i)$. 

\item Fix $\alpha\in\Mon_I(\F_3)$, and let $\gee_i\in\F_3^\times$ ($i\in 
I$) and $\sigma\in\Sigma_I$ be such that $\alpha(e_i)=\gee_ie_{\sigma(i)}$ 
for all $i$. Let $\ttt(\alpha)\in\Perm_\Delta(\F_3)$ be the automorphism 
that sends $e_{(c,i)}$ to 
$e_{(\gee_ic,\sigma(i))}$. Thus for $\xi\in\F_3^\Delta$, we have 
$(\ttt(\alpha)(\xi))(c,i)=\xi(\gee_{\sigma^{-1}(i)}c,\sigma^{-1}(i))$. 

\item Define 
	\[ \NNnul = \trs\scrt \rtimes \ttt(Aut(\scrt)) = 
	\Gen{\trs\eta,\ttt(\alpha) \,\big|\, \eta\in\scrt, 
	\alpha\in\Aut(\scrt)} \le \twoMat12, \]
and set $\NN=\NNnul\times\{\pm\Id\}\le\twoMat12$. 

\end{enuma}
\end{Not}

By \cite[Proposition 7.29]{Griess}, $\twoMat12\cong 2M_{12}$. 

Note the following relations, for $\eta,\theta\in\F_3^I$, $i\in I$, and 
$\alpha\in\Mon_I(\F_3)$:
	\begin{align*} 
	\trs\eta(C_i) &= C_i & \ttt(\alpha)(C_i) &= C_{\pi(\alpha)(i)} \\
	\trs\eta(\Grf(\theta)) &= \Grf(\theta+\eta) & 
	\ttt(\alpha)(\Grf(\theta)) &= \Grf(\alpha(\theta)) .
	\end{align*}
To see the last equality, note that for $\alpha\in\Mon_I(\F_3)$ with 
$\gee_i\in\F_3^\times$ and $\sigma\in\Sigma_I$ as above, and for 
$\theta=\sum_{i\in I}\theta(i)e_i$ in $\F_3^I$, we have 
	\[ \ttt(\alpha)(\Grf(\theta)) = \sum_{i\in I} 
	\ttt(\alpha)(e_{(\theta(i),i)}) = \sum_{i\in 
	I}e_{(\gee_i\theta(i),\sigma(i))} = \Grf(\theta') \]
where $\theta'=\sum_{i\in I}\gee_i\theta(i)e_{\sigma(i)}=\alpha(\theta)$. 
In particular, these formulas show that the action of $\NNnul$ on 
$\F_3^\Delta$ sends $\4\scrg$ and $\scrg$ to themselves.

\begin{Lem} \label{l:H-max}
We have $\NN=N_{\twoMat12}(\trs\scrt)$, and this is a maximal subgroup 
of $\twoMat12$. 
\end{Lem}

\begin{proof} By construction, $\NN\le N_{\twoMat12}(\trs\scrt)$. 
Conversely, by \cite[Theorem 7.20]{Griess}, $\NN$ is the subgroup 
of all elements of $\twoMat12$ whose action on $\Delta$ permutes the 
columns $\F_3\times\{i\}$, and hence contains the normalizer of 
$\trs\scrt$. 

For the maximality of $\NN\le\twoMat12$ (or of 
$\NN/\{\pm\Id\}\cong E_9\rtimes\GL_2(3)$ in 
$\twoMat12/\{\pm\Id\}\cong M_{12}$), see \cite[p. 235]{Conway} or 
\cite[p. 8]{A-overgr}. Note that if we regard $M_{12}$ as a group of 
permutations of $12$ points, then $\NN/\{\pm\Id\}\cong 
M_9\rtimes\Sigma_3$ is the subgroup of those permutations that 
normalize a set of three of the points. 
\end{proof}

One easy consequence of Lemma \ref{l:H-max} is that 
$\NNnul=\twoMat12\cap\Perm_\Delta(\F_3)$. In other words, the elements of 
$\NNnul$ are the only ones in $\twoMat12$ that permute the coordinates in 
$\Delta$ without sign changes. But this will not be needed later.

To simplify later calculations, we next describe $\scrg$ and the action 
of $\NNnul$ on it in terms of $(3\times3)$ matrices over $\F_3$. 
In general, for a vector space $V$ over a field $K$, we let $S_2(V)$ denote 
its symmetric power 
	\[ S_2(V)=(V\otimes_KV)/\gen{(v\otimes w)-(w\otimes v) \,|\, v,w\in 
	V}. \]
For $v,w\in V$, let $[v\otimes w]\in S_2(V)$ denote the class of 
$v\otimes w\in V\otimes_KV$, and write $v^{\otimes2}=[v\otimes v]$ for 
short. When $\alpha\in\Aut_K(V)$, we let $S_2(\alpha)\in\Aut_K(S_2(V))$ be 
the automorphism $S_2(\alpha)([v\otimes w])=[\alpha(v)\otimes\alpha(w)]$. 

\begin{Defi} \label{d:Phi-Th}
\begin{enuma} 

\item Choose a map of sets $\lambda\:I\too\scrt$ such that for each 
$i\in I$, $\lambda(i)\ne0$ and $(\lambda(i))(i)=0$. Define a map of sets 
	\[ \Phi_0 \: \scrc\cup\Grf(\scrt) \Right5{} S_2(\scrt\oplus\F_3) \]
by setting 
	\[ \Phi_0(C_i)=(\lambda(i),0)^{\otimes2} \qquad\textup{and}\qquad 
	\Phi_0(\Grf(\xi)) = (\xi,1)^{\otimes2} \]
for all $i\in I$ and all $\xi\in\scrt$. 

\item Define $\Theta_*\:\NNnul\Right2{}\Aut(\scrt\oplus\F_3)$ by setting 
	\[ \Theta_*(\trs\eta\ttt(\alpha))(\xi,a) = (\alpha(\xi)+a\eta,a) \]
for each $\eta,\xi\in\scrt$, $\alpha\in\Aut(\scrt)$, and $a\in\F_3$. 

\end{enuma}
\end{Defi}

We now check that $\Phi_0$ and $\Theta_*$ extend to a natural isomorphism 
from the $\F_3\NNnul$-module $\scrg$ to the group $S_2(\scrt\oplus\F_3)$ 
with action of a certain subgroup of $\Aut(\scrt\otimes\F_3)$. 

\begin{Lem} \label{l:Ph-Th*}
\begin{enuma} 

\item The map $\Phi_0$ of Definition \ref{d:Phi-Th}(a) is independent of 
the choice of $\lambda$, and extends to a surjective homomorphism 
$\4\Phi\:\4\scrg\too S_2(\scrt\oplus\F_3)$. This in turn restricts to 
an isomorphism $\Phi_*$ from $\scrg$ onto $S_2(\scrt\oplus\F_3)$. 

\item The map $\Theta_*$ of Definition \ref{d:Phi-Th}(b) is an 
isomorphism from $\NNnul\cong\scrt\rtimes\Aut(\scrt)$ onto the group of all 
automorphisms of $\scrt\oplus\F_3$ that are the identity modulo 
$\scrt\oplus0$. 

\item For each $\beta\in\NNnul$ and each $\gamma\in\scrg$, 
	\beqq \Phi_*(\beta(\gamma)) = S_2(\Theta_*(\beta))(\Phi_*(\gamma)). 
	\label{e:HH-action} \eeqq
Thus $\Theta_*$ and $\Phi_*$ define 
an isomorphism from $\scrg$ as an $\F_3\NNnul$-module to 
$S_2(\scrt\oplus\F_3)$ with its natural structure as a module over 
$\Theta_*(\NNnul)<\Aut(\scrt\oplus\F_3)$. 

\end{enuma}
\end{Lem}

\begin{proof} \noindent\textbf{(a) } For each $i\in I$, the choice of 
$\lambda(i)$ is unique up to sign. So 
$\4\Phi(C_i)=(\lambda(i),0)^{\otimes2}$ is independent of the choice of 
$\lambda(i)$. 

We first check that $\sum_{i\in I}\Phi_0(C_i)=0$. It suffices to show that 
$\sum_{i\in I}\lambda(i)^{\otimes2}=0$ in $S_2(\scrt)$. Independently of our 
choices, $\{\lambda(i)\,|\,i\in I\}$ is a set of representatives of the four 
subspaces of dimension $1$ in $\F_3^2$. So the $\lambda(i)$ are permuted up 
to sign by each $\alpha\in\Aut(\scrt)$, and the sum of the 
$\lambda(i)^{\otimes2}$ is fixed by each such $\alpha$. Hence the sum must 
be zero. (Alternatively, this can be shown directly by choosing coordinates 
and then computing with matrices.)

We next check that the relations \eqref{e:C-xi-relation} hold for the 
images of the elements in $\scrc\cup\Grf(\scrt)$ under $\Phi_0$ as defined 
above. So fix $\xi,\eta,\theta\in\scrt$ such that $\xi+\eta+\theta=0$. If 
$\xi=\eta=\theta$, then \eqref{e:C-xi-relation} clearly holds. Otherwise, 
$\xi-\eta\ne0$, so there is a unique index $j\in I$ 
such that $(\xi-\eta)(j)=0$. Then $\xi-\eta=\pm\lambda(j)$, and so 
	\begin{align*} 
	(\xi,1)^{\otimes2}+(\eta,1)^{\otimes2}+(\theta,1)^{\otimes2} &= 
	(\xi,0)^{\otimes2}+(\eta,0)^{\otimes2}+(\theta,0)^{\otimes2} \\ 
	&= (\xi,0)^{\otimes2}+(\eta,0)^{\otimes2}+(-\xi-\eta,0)^{\otimes2} 
	= - (\xi-\eta,0)^{\otimes2} \\ 
	&= - (\lambda(j),0)^{\otimes2} 
	= \sum_{i\in I\sminus\{j\}}(\lambda(i),0)^{\otimes2}, 
	\end{align*}
where the first equality holds since $\xi+\eta+\theta=0$, and the last 
one since $\sum_{i\in I}\Phi_0(C_i)=0$. 

Thus $\Phi_0$ extends to a homomorphism defined on a vector space over 
$\F_3$ with basis $\scrc\cup\Grf(\scrt)$, modulo the subspace generated by 
the relations \eqref{e:C-xi-relation}. This quotient space is generated by 
the images of the $C_i$, as well as those of $0$, $\xi_1$, and $\xi_2$ for 
any basis $\{\xi_1,\xi_2\}$ of $\scrt$, hence has dimension $7$ and is 
isomorphic to $\4\scrg$. So $\Phi_0$ extends to a homomorphism $\4\Phi$ 
from $\4\scrg$ to $S_2(\scrt\oplus\F_3)$. 

Now, $\4\Phi(\gen{\scrc})=\gen{(\eta,0)^{\otimes2}}=S_2(\scrt\oplus0)$ since 
$\scrt^\#=\{\lambda(i)^{\pm1}\,|\,i\in I\}$. Hence 
	\[ \4\Phi(\NNnul) = S_2(\scrt\oplus0)
	\Gen{(\xi,1)^{\times2}\,\big|\,\xi\in\scrt} = S_2(\scrt\oplus\F_3). \]
Thus $\4\Phi$ is onto, and a comparison of dimensions shows that 
$\Ker(\4\Phi)=\gen{e_\Delta}$. Since $e_\Delta\notin\scrg$, $\4\Phi$ 
restricts to an isomorphism $\Phi_*$ from $\scrg$ to $\Sym_3(\F_3)$.

\smallskip

\noindent\textbf{(b) } One easily checks that $\Theta_*$ as defined above 
restricts to homomorphisms on $\{\trs\eta\,|\,\eta\in\scrt\}\cong\scrt$ and 
on $\Aut(\scrt)$. So it remains only to check conjugacy relations: for 
$\alpha\in\Aut(\scrt)$ and $\eta\in\scrt$, we have 
	\begin{align*} 
	\Theta_*(\alpha)\bigl(\Theta_*(\trs\eta)(\Theta_*(\alpha)^{-1}(\xi,a) 
	)\bigr)  
	&= \Theta_*(\alpha)(\alpha^{-1}(\xi)+a\eta,a) 
	= (\xi, a\cdot\alpha(\eta),a) \\
	&= \Theta_*(\trs{\alpha(\eta)})(\xi,a) 
	= \Theta_*(\alpha\circ\trs\eta\circ\alpha^{-1})(\xi,a) . 
	\end{align*}

Thus $\Theta_*$ is well defined on $\NNnul$, and it clearly defines an 
isomorphism onto the group of all $\beta\in\Aut(\scrt\oplus\F_3)$ that are 
the identity modulo $\scrt\oplus0$.

\smallskip

\noindent\textbf{(c) } For each $\xi,\eta\in\scrt$, $i\in I$, and 
$\alpha\in\Aut(\scrt)$, we have 
	\begin{align*} 
	\4\Phi(\trs\eta(\Grf(\xi))) &= (\xi+\eta,1)^{\otimes2} = 
	(\Theta_*(\trs\eta)(\xi,1))^{\otimes2} = 
	S_2(\Theta_*(\trs\eta))(\4\Phi(\Grf(\xi))) \\
	\4\Phi(\ttt(\alpha)(\Grf(\xi))) &= (\alpha(\xi),1)^{\otimes2} = 
	(\Theta_*(\ttt(\alpha))(\xi,1))^{\otimes2} = 
	S_2(\Theta_*(\ttt(\alpha)))(\4\Phi(\Grf(\xi))) \\ 
	\4\Phi(\trs\eta(C_i)) &= \4\Phi(C_i) = (\lambda(i),0)^{\otimes2} 
	= S_2(\Theta_*(\trs\eta))(\4\Phi(C_i)) .
	\end{align*}
Also, for all $\alpha\in\Aut(\scrt)$ inducing the permutation 
$\sigma\in\Sigma_I$, and all $i\in I$, 
	\begin{align*} 
	\4\Phi(\ttt(\alpha)(C_i)) = \4\Phi(C_{\sigma(i)}) 
	&= \bigl(\lambda(\sigma(i)),0\bigr)^{\otimes2} \\ 
	&= \bigl(\pm\Theta_*(\ttt(\alpha))(\lambda(i),0)\bigr)^{\otimes2} 
	= S_2(\Theta_*(\ttt(\alpha)))(\4\Phi(C_i)) 
	\end{align*}
where $\lambda(\sigma(i))=\pm \ttt(\alpha)(\lambda(i))$
by definition (and uniqueness up to sign) of the $\lambda(i)$. 
Since $\4\scrg=\gen{\scrc\cup\Grf(\scrt)}$ and 
$\NNnul=\gen{\trs\eta,\ttt(\alpha)\,|\,\eta\in\scrt,~\alpha\in\Aut(\scrt)}$, 
this proves \eqref{e:HH-action}.
\end{proof}


To simplify computations still farther, we now describe elements in $\NNnul$ 
and $\AA$ as $(3\times3)$-matrices over $\F_3$. Fix an 
isomorphism $\scrt\cong\F_3^2$ (e.g., by restriction to the first two 
coordinates), so that $\scrt\oplus\F_3$ is identified with $\F_3^3$ and 
$\Aut(\scrt\oplus\F_3)$ with $\GL_3(\F_3)$. 
We then identify $S_2(\scrt\oplus\F_3)$ with 
the group $\Sym_3(\F_3)$ of symmetric $(3\times3)$ matrices over $\F_3$, by 
sending the class $[v\otimes w]$ (for $v,w\in\F_3^3$) to $\frac12(v\cdot 
w^t+w\cdot v^t)$. More explicitly, 
	\[ \left[\Colthree{a}bc \otimes \Colthree{d}ef\right] 
	\quad\textup{is sent to}\quad  
	\Mxthree{ad}{(ae+bd)}{(af+cd)/2}{(ae+bd)/2}{be}
	{(bf+ce)/2}{(af+cd)/2}{(bf+ce)/2}{cf}. \]
Let 
	\begin{align} 
	\Phi\: \scrg &\Right5{\cong} \Sym_3(\F_3) \label{e:Phi} \\
	\Theta\: \NNnul &\Right5{\cong} \left\{\left. \mxthree{a}bcdef001 
	\,\right|\, \begin{array}{c}a,b,c,d,e,f\in\F_3\\ 
	ae-bd\ne0\end{array} \right\} \le \GL_3(\F_3) \label{e:Theta}
	\end{align}
be the composites of $\Phi_*$ and $\Theta_*$ with the isomorphisms 
induced by this identification $\scrt\cong\F_3^2$. Lemma 
\ref{l:Ph-Th*}(c) now takes the following form:

\begin{Lem} \label{l:Phi-Th}
For each $\beta\in\NNnul$ and each $\xi\in\scrg$, 
	\[ \Phi(\beta(\xi)) = \Theta(\beta)\Phi(\xi)\Theta(\beta)^t 
	\in \Sym_3(\F_3). \]
\end{Lem}

As a first, very simple application, we describe the Jordan blocks 
for actions on $\AA$.

\begin{Lem} \label{l:nr.J.bl.}
There are exactly two conjugacy classes of elements of order $3$ in 
$\twoMat12$: those in one class act on $\scrg$ with three Jordan blocks 
of lengths $1,2,3$, and those in the other with two Jordan blocks of 
length $3$. In particular, for each $x\in\twoMat12$ of order $3$, 
$\rk(C_{\scrg}(x))\le3$.
\end{Lem}

\begin{proof} Each element of order $3$ in $M_{12}$ is the image of a 
unique element of order $3$ in $2M_{12}$. So $\twoMat12$ has two 
conjugacy classes of elements of order $3$ since $M_{12}$ does (see, 
e.g., \cite[Exercise 7.34(ii)]{Griess}). With the help of Lemma 
\ref{l:Phi-Th}, it is straightforward to check that 
$\Theta^{-1}\left(\mxthree101010001\right)$ acts on $\scrg$ with three 
Jordan blocks of lengths $1,2,3$, and that 
$\Theta^{-1}\left(\mxthree110011001\right)$ acts with two Jordan blocks 
of length $3$. Thus these elements are in different classes, and each 
element of order $3$ in $\twoMat12$ is conjugate to one of them and 
acts on $\scrg$ in one of these two ways. The last statement holds 
since the rank of $C_{\scrg}(x)$ is equal to the number of Jordan blocks. 
(See also \cite[Exercise 7.37]{Griess}.) 
\end{proof}

The notation developed in this subsection is summarized in 
Table \ref{tbl:Todd-12}. 
	\begin{Table}[ht]
	\[ \renewcommand{\arraystretch}{1.5} 
	\begin{array}{|c|} \hline
	\GG = \twoMat12 = 
	\bigl\{ \alpha\in\Mon_\Delta(\F_3) \,\big|\, 
	\alpha(\scrg)=\scrg \bigr\} \cong 2M_{12} \\
	\NNnul = \trs\scrt \rtimes \ttt(\Aut(\scrt)) \le \twoMat12 \\
	\NN = \NNnul\times\{\pm\Id\} = \bigl\{ \alpha\in \twoMat12 
	\,\big|\, \textup{$\alpha$ 
	permutes the $\scrk_i$} \bigr\} \\[2pt] 
	\Theta\: \NNnul \Right3{\cong} \bigl\{ \mxtwo{A}v01 \bmid 
	A\in\GL_2(3), ~ v\in \F_3^2 \bigr\} \le \GL_3(3) \\
	\TT = \Theta^{-1}(\UT_3(\F_3)) \in \syl3{\NNnul} \subseteq \syl3\GG 
	\\
	\hline
	\AA = \Phi(\scrg) = \Sym_3(\F_3) \\
	\beta(X)=\Theta(\beta)X\Theta(\beta)^t \quad 
	\textup{for}~\beta\in\NNnul,~ X\in\AA 
	\\\hline
	\end{array} \]
	\caption{{Notation used for certain subgroups of 
	$\GG=\twoMat12$ and their action on $\AA=\Phi(\scrg)$.}} 
	\label{tbl:Todd-12}
	\end{Table}

\newsubb{Notation for the Todd modules of \texorpdfstring{$M_{11}$ and 
$A_6$}{M11 and A6}}{s:Td10-11}

We next set up notation to work with the Todd modules of $M_{11}$ and of 
$A_6\cong O^2(M_{10})$. In particular, we get explicit descriptions of the 
actions of certain subgroups of $A_6$ and $M_{11}$. 

Let $\scrg\le\F_3^\Delta$ be as in \eqref{e:Golay} and Notation 
\ref{n:transl}. By \cite[Lemma 7.12]{Griess}, $\scrg$ 
contains exactly $12$ pairs $\{\pm\theta\}$ of elements of weight 12. Three 
of those pairs lie in $\gen{\scrc}$: the elements of the form $\sum_{i\in 
I}\gee_iC_i$ for $\gee_i\in\F_3^\times$ and $\sum_{i\in I}\gee_i=0$. (The 
other nine have the form $\pm(e_\Delta+\Grf(\xi))$ for $\xi\in\scrt$.) By a 
direct check, for each basis $\{\xi,\eta\}$ of $\scrt$, the six elements 
	\beqq \bigl\{\pm((\xi,0)^{\otimes2}+(\eta,0)^{\otimes2}), ~ 
	\pm((\xi,0)^{\otimes2}-(\eta,0)^{\otimes2})\pm[(\xi,0)\otimes(\eta,0)] 
	\bigr\} \subseteq S_2(\scrt\oplus\F_3) \label{e:6pts} \eeqq
are the images of the six elements of weight 12 in $\gen\scrc$
under the isomorphism 
	\[ \Phi_* \: \scrg \Right4{\cong} S_2(\scrt\oplus\F_3) \]
of Lemma \ref{l:Ph-Th*}(a). We want to identify 
$M_{11}$ as the subgroup of elements in $\twoMat12$ that are the 
identity on one of these subspaces, and similarly for $M_{10}$.

To simplify these descriptions, we identify $\scrt$ with $\F_9$ via 
some arbitrarily chosen isomorphism. We adopt the following notation 
for elements of $\F_9$:
	\beqq \begin{split} 
	&\F_9=\F_3[i] ~\textup{where}~ i^2=-1 \\
	&\zeta=1+i ~\textup{of order $8$ in $\F_9^\times$} \\
	&\phi\in\Aut(\F_9)~: \quad\textup{$\phi(a+bi)=a-bi$ for 
	$a,b\in\F_3$.}
	\end{split} \label{e:F9} \eeqq
We also write $\4x=\phi(x)$ for $x\in\F_9$. 

\begin{Not} \label{n:L10-11}
Assume Notation \ref{n:transl} and Table 
\ref{tbl:Todd-12}, and choose an $\F_3$-linear isomorphism 
$\kappa\:\scrt\xto{~\cong~}\F_9$. 
Define elements $\theta_1,\theta_2,\theta_3\in 
S_2(\scrt)\le S_2(\scrt\oplus\F_3)$ by setting 
	\begin{align*} 
	\theta_1 &= S_2(\kappa)^{-1}([1\otimes1+i\otimes i]) \\
	\theta_2 &= S_2(\kappa)^{-1}([1\otimes1-i\otimes i+1\otimes i]) \\
	\theta_3 &= S_2(\kappa)^{-1}([1\otimes1-i\otimes i-1\otimes i]).
	\end{align*}
Set $\theta^*_i=\Phi_*^{-1}(\theta_i)\in\scrg$. By \eqref{e:6pts}, 
$\pm\theta^*_1$, $\pm\theta^*_2$, and $\pm\theta^*_3$ are elements of weight 
$12$ in $\scrg$, and the only ones in $\gen\scrc\cap\scrg$. 

Set $K_1=\gen{\theta^*_1}$ and $K_2=\gen{\theta^*_2,\theta^*_3}$, both 
subspaces of $\scrg$, and define 
	\[ \twoMat11 = N_{\twoMat12}(K_1) \qquad\textup{and}\qquad 
	\twoMat10 = N_{\twoMat12}(K_2). \]
Also, set $\Matnul{}\ell=O^{3'}(\twoMat{}\ell)$ and 
$\nnn{\ell}=\NN\cap\twoMat{}\ell$ for $\ell=10,11$, and set 
$\TT=\trs\scrt$. 

Finally, define $\lambda\:\F_9^\times\gen{\phi}\too\Aut(\scrt)$ 
by setting $\lambda(u)=\kappa^{-1}(x\mapsto ux)\kappa$ for 
$u\in\F_9^\times$ and $\lambda(\phi)=\kappa^{-1}\phi\kappa$. (Recall that 
we compose from right to left.) 
For $x\in\F_9$ and $u\in\F_9^\times$, set 
	\[ \dpar{x} = \trs{\kappa^{-1}(x)}\in \TT, 
	\quad \sbk{u} = \ttt(\lambda(u))\in\NN, 
	\quad\textup{and}\quad \sbk{\phi} = 
	\ttt(\lambda(\phi))\in\NN. \]
Also, for $\xi\in\NNnul$, we write $-\xi=\xi\cdot(-\Id)\in\NN$.
\end{Not}

For easy reference, we summarize in Table \ref{tbl:Mat} some of 
the basic properties of groups defined in Notation \ref{n:L10-11}.

\begin{Lem} \label{l:N10-11}
Assume Notation \ref{n:L10-11}. Then for $\ell=10,11$, 
$\Matnul{}\ell=C_{\twoMat12}(K_{12-\ell})=C_{\twoMat{}\ell}(K_{12-\ell})$, 
and the groups $\twoMat{}\ell$, $\Matnul{}\ell$, $\nnn{\ell}$, and $\TT$ are 
as described in Table \ref{tbl:Mat}. In particular, 
$\TT\in\syl3{\twoMat{}\ell}=\syl3{\Matnul{}\ell}$. 
\begin{Table}[ht]
\[ \renewcommand{\arraystretch}{1.4}
\addtolength{\arraycolsep}{2mm}
\begin{array}{c|cc} 
 & \ell=10 & \ell=11 \\\hline
\TT & \trs\scrt=\gen{\dpar{x}\,|\,x\in\F_9} 
& \trs\scrt=\gen{\dpar{x}\,|\,x\in\F_9} \\
\nnn{\ell} & \TT\gen{\sbk{\zeta},\sbk{\phi},-\Id} 
& \TT\gen{\sbk{\zeta},\sbk{\phi},-\Id} \\
\nnn{\ell}\cap\Matnul{}\ell & \TT\gen{{-}\sbk{i}} & 
\TT\gen{{-}\sbk{\zeta},\sbk{\phi}} \\
\Matnul{}\ell \cong & A_6 & M_{11} \\
\twoMat{}\ell/\Matnul{}\ell\cong & D_8 & C_2 
\end{array} \]
\caption{In particular, $\NNN0=\NNN1\cong (E_9\rtimes\SD_{16})\times C_2$.}
\label{tbl:Mat} 
\end{Table}
\end{Lem}

\begin{proof} By definition (see Notation \ref{n:transl}(d)), each element 
of $\NN$ normalizes the subspace $\gen{\scrc}\cap\scrg$, and hence 
permutes the six elements $\pm\theta_1,\pm\theta_2,\pm\theta_3$ (the 
only elements of weight $12$ in $\gen\scrc\cap\scrg$). Some of these 
actions are described in Table \ref{tbl:acttheta}.
\begin{Table}[ht]
\[ \renewcommand{\arraystretch}{1.2}
\begin{array}{c|cccc} 
g\in\NN  & \sbk{\zeta} & \sbk{i} & \sbk{\phi} & -\Id \\\hline 
\9g\theta^*_1 & -\theta^*_1 & \theta^*_1 & \theta^*_1 & -\theta^*_1 \\
\9g\theta^*_2 & -\theta^*_3 & -\theta^*_2 & \theta^*_3 & -\theta^*_2 \\
\9g\theta^*_3 & \theta^*_2 & -\theta^*_3 & \theta^*_2 & -\theta^*_3 
\end{array} \]
\caption{}
\label{tbl:acttheta} 
\end{Table}

Consider, for example, the case $\9{\sbk{\zeta}}\theta^*_2$. 
Set $\xi=\kappa^{-1}(1)$ and $\eta=\kappa^{-1}(i)$, where 
$\kappa\:\scrt\xto{~\cong~}\F_9$ is as in Notation \ref{n:L10-11}. 
Then 
	\[ \Phi_*(\theta_2^*) = \theta_2 = S_2(\kappa)^{-1} \bigl( [1\otimes1 
	- i\otimes i + 1\otimes i]\bigr) = 
	[\xi\otimes\xi - \eta\otimes\eta + \xi\otimes\eta ]. \]
Since $\zeta=1+i$ and $i\zeta=-1+i$, we get 
	\begin{align*} 
	\Phi_*(\9{\sbk{\zeta}}\theta^*_2) &= S_2(\kappa)^{-1}\bigl( 
	[(1+i)\otimes(1+i) - (-1+i)\otimes(-1+i) + (1+i)\otimes(-1+i)] 
	\bigr) \\
	&= [(\xi+\eta)\otimes(\xi+\eta) - (-\xi+\eta)\otimes(-\xi+\eta) 
	+ (\xi+\eta)\otimes(-\xi+\eta)] \\
	&= [ 4(\xi\otimes\eta) - \xi\otimes\xi + \eta\times\eta ] = 
	-\Phi_*(\theta^*_3).
	\end{align*}
Hence $\9{\sbk{\zeta}}\theta^*_2=-\theta^*_3$. The other computations are 
similar, but simpler in most cases.

Recall that $\NN=(\trs\scrt\rtimes\ttt(\Aut(\scrt)))\times\{\pm\Id\}$ 
(Notation \ref{n:transl}(d)), where 
$\Aut(\scrt)\cong\GL_2(3)\cong2\Sigma_4$ by \eqref{e:Aut(T)}. Since the 
element $\sbk{{-}1}=\sbk{i}^2$ centralizes $K_1K_2$ by Table \ref{tbl:acttheta}, 
each element of $\trs\scrt=[\sbk{{-}1},\trs\scrt]$ 
also centralizes $K_1K_2$. Also, each noncentral element of 
$O_2(\ttt(\Aut(\scrt)))=\gen{\sbk{i},\sbk{\zeta\phi}}\cong Q_8$ fixes one 
of the $\theta^*_i$ and sends the other two to their negative, and hence 
each element of order $3$ in $\ttt(\Aut(\scrt))$ acts by permuting the sets 
$\{\pm\theta^*_i\}$ ($i=1,2,3$) cyclically. 
From this, we conclude that $\NNN0=\NNN1$ is as described in Table 
\ref{tbl:Mat}, and also that 
	\[ C_{\NNN0}(K_2)=\TT\gen{-\sbk{i}} 
	\quad\textup{and}\quad
	C_{\NNN1}(K_1)=\TT\gen{-\sbk{\zeta},\sbk{\phi}}. \]
In particular, $\NNN0/C_{\NNN0}(K_2)\cong D_8$ and 
$\NNN1/C_{\NNN1}(K_1)\cong C_2$. 

It remains only to show that 
$\Matnul{}\ell=C_{\twoMat{}\ell}(K_{12-\ell})$. For $\ell=10$ or $\ell=11$, 
consider the action of $\twoMat{}\ell=N_{\twoMat12}(K_{12-\ell})$ on 
$\scrg/K_{12-\ell}$. Since $\Matnul10\cong O^{3'}(M_{10})\cong A_6$ and 
$\Matnul11\cong M_{11}$ by definition of $M_{10}$ and $M_{11}$ as 
permutation groups, and since $\dim(\scrg/K_{12-\ell})=4$ or $5$, 
respectively, this quotient is absolutely irreducible as an 
$\F_3\Matnul{}\ell$-module by Lemma \ref{l:4-5dim}. Hence 
$C_{\Aut(\scrg/K_{12-\ell})}(\Matnul{}\ell)=\{\pm\Id\}$, and so 
	\beqq |\nnn{\ell}/C_{\nnn{\ell}}(K_{12-\ell})| \le 
	|\twoMat{}\ell/C_{\twoMat{}\ell}(K_{12-\ell})| \le 
	|\twoMat{}\ell/\Matnul{}\ell| 
	\le 2\cdot|\Out(\Matnul{}\ell)| 
	\label{e:[M:M0]} \eeqq
We just saw that $|\nnn{10}/C_{\nnn{10}}(K_2)|=8=2\cdot|\Aut(A_6)|$ and 
$|\nnn{11}/C_{\nnn{11}}(K_1)|=2=2\cdot|\Aut(M_{11})|$, and so the 
inequalities in \eqref{e:[M:M0]} are all equalities. 
Hence $\Matnul{}\ell=C_{\twoMat{}\ell}(K_{12-\ell}) 
=C_{\twoMat12}(K_{12-\ell})$, and the descriptions of 
$\nnn{\ell}\cap\Matnul{}\ell$ and $\twoMat{}\ell/\Matnul{}\ell$ in Table 
\ref{tbl:Mat} all hold.
\end{proof}

As seen in Lemma \ref{l:4-5dim}, there are three different representations 
that appear under Hypotheses \ref{h:hyp11}: one of $A_6$ and two of 
$M_{11}$. We will refer to these throughout the rest of the section as the 
``$A_6$-case'' (when $\GGnul\cong A_6$), the ``$M_{11}$-case'' (when 
$\GGnul\cong M_{11}$ and $\AA$ is its Todd module), and the 
``$M_{11}^*$-case'' (when $\GGnul\cong M_{11}$ and $\AA$ is the dual Todd 
module). 

\begin{Lem} \label{l:Td10-11}
Assume Notation \ref{n:L10-11}. We summarize here the notation we use for 
the $\F_3\twoMat10$- and $\F_3\twoMat11$-modules we are working with, and 
describe explicitly the action of the subgroup $\NNN0$ or $\NNN1$. 
\begin{enuma} 

\item \textup{($A_6$-case)} We identify the Todd module for $\twoMat10$ with 
$\AAA0\defeq \F_3\times\F_9\times\F_3$ in such a way that $\NNN0$ 
acts as follows: 
	\begin{align*} 
	\9{\dpar{x}}\trp[a,b,c] &= \trp[a,b-ax,c+\Tr(x\4b)-aN(x)] & 
	\textup{for $x\in\F_9$} \\
	\9{\sbk{u}}\trp[a,b,c] &= \trp[a,ub,N(u)c] & 
	\textup{for $u\in\F_9^\times$} \\ 
	\9{\sbk{\phi}}\trp[a,b,c] &= \trp[a,\4b,c] \quad\textup{and}\quad 
	\9{-\Id}\trp[a,b,c]=\trp[{-}a,-b,-c].
	\end{align*}

\item \textup{($M_{11}$-case)} We identify the Todd module for $\twoMat11$ with 
$\AAA1\defeq \F_3\times\F_9\times\F_9$ in such a way that $\NNN1$ 
acts as follows: 
	\begin{align*} 
	\9{\dpar{x}}\trp[a,b,c] &= \trp[a,b-ax,c+bx+ax^2] & 
	\textup{for $x\in\F_9$} \\
	\9{\sbk{u}}\trp[a,b,c] &= \trp[a,ub,u^2c] & 
	\textup{for $u\in\F_9^\times$} \\
	\9{\sbk{\phi}}\trp[a,b,c] &= \trp[a,\4b,\4c] \quad\textup{and}\quad 
	\9{-\Id}\trp[a,b,c] = \trp[{-}a,-b,-c].
	\end{align*}

\item \textup{($M_{11}^*$-case)} We identify the dual Todd module for 
$\twoMat11$ with $\AAA1^*\defeq \F_9\times\F_9\times\F_3$ in such a way that 
$\NNN1$ acts as follows: 
	\begin{align*} 
	\9{\dpar{x}}\trp[a,b,c]&=\trp[a,b-ax,c+\Tr(bx+ax^2)] &
	\textup{for $x\in\F_9$} \\
	\9{\sbk{u}}\trp[a,b,c] &= \trp[u^{-2}a,u^{-1}b,c] &
	\textup{for $u\in\F_9^\times$} \\
	\9{\sbk{\phi}}\trp[a,b,c] &= \trp[\4a,\4b,c] \quad\textup{and}\quad 
	\9{-\Id}\trp[a,b,c] = \trp[{-}a,-b,-c].
	\end{align*}

\end{enuma}
\end{Lem}

\begin{proof} 
\noindent\textbf{(b) } Define 
	\[ \5\kappa_{11} \: S_2(\scrt\oplus\F_3)
	\Right6{} \AAA1 = \F_3 \times \F_9 \times \F_9 \]
by setting 
	\[ \5\kappa_{11}\bigl([(\xi,r)\otimes(\eta,s)]\bigr)= 
	\trp[rs,r\kappa(\eta)+s\kappa(\xi),\kappa(\xi)\cdot\kappa(\eta)]. 
	\]
This is surjective since $\AAA1$ is generated by the elements 
	\[ \5\kappa_{11}([(0,1)\otimes(\eta,s)])=\trp[s,\kappa(\eta),0] 
	\qquad\textup{and}\qquad
	\5\kappa_{11}([(1,0)\otimes(\eta,0)]) = 
	\trp[0,0,\kappa(\eta)]. \]
Also, $\5\kappa_{11}(\theta_1)=0$, so 
$\Ker(\5\kappa_{11}\circ\Phi)=\gen{\theta^*_1}=K_1$ since they both are 
$1$-dimensional. Thus the action of $\twoMat12$ on $\scrg$ induces an 
action of $\twoMat11=N_{\twoMat12}(K_1)$ on $\scrg/K_1\cong\AAA1$.

For $\theta\in\scrt$, $\trs\theta(\xi,r)=(\xi+r\theta,r)$ and 
$\trs\theta(\eta,s)=(\eta+s\theta,s)$. So if we set $x=\kappa(\theta)$ and 
$\trp[a,b,c]=\5\kappa_{11}([(\xi,r)\otimes(\eta,s)])$, then 
	\begin{align*} 
	\9{\dpar{x}}\trp[a,b,c] &=
	\5\kappa_{11}\bigl([(\xi+r\theta,r)\otimes(\eta+s\theta,s)]
	\bigr) \\ 
	&= \trp[rs,(r\kappa(\eta)+s\kappa(\xi))+2rs\kappa(\theta), 
	\kappa(\xi)\kappa(\eta)+\kappa(\theta)
	(r\kappa(\eta)+s\kappa(\xi))+rs\kappa(\theta)^2] \\
	&= \trp[a,b-ax,c+bx+ax^2].
	\end{align*}
The other formulas follow by similar (but simpler) arguments. 

\smallskip

\noindent\textbf{(c) } The description of the action of $\NNN1$ on 
${\AAA1}^*$ follows from that in (b), together with the relation 
$\gen{\9g\xi,\eta}=\gen{\xi,\9[2]{g^{-1}}\eta}$ for 
$\xi\in{\AAA1}^*$ and $\eta\in\AAA1$, where the nonsingular pairing 
	\[ {\AAA1}^* \times \AAA1 = (\F_9\times\F_9\times\F_3) \times 
	(\F_3\times\F_9\times\F_9) \Right6{\gen{-,-}} \F_3 \]
is defined by $\Gen{\trp[a,b,z],\trp[y,c,d]} = yz + \Tr(ad+bc)$. 

\smallskip

\noindent\textbf{(a) } This proof is similar to that of (b), except that 
$\5\kappa_{11}$ is replaced by the map 
	\[ \5\kappa_{10} \: S_2(\scrt\oplus\F_3) 
	\Right6{} \AAA0 = \F_3 \times \F_9 \times \F_3, \]
defined by setting 
	\[ \5\kappa_{10}\bigl([(\xi,r)\otimes(\eta,s)]\bigr) 
	= \trp[rs,r\kappa(\eta)+s\kappa(\xi),
	\Tr(\kappa(\xi)\cdot\4{\kappa(\eta)})]. \]
This is easily seen to be surjective. For $i=2,3$, we have 
	\[ \5\kappa_{10}(\theta_i^*) = \trp[0,0,\Tr(1\cdot1-i\cdot\4\imath 
	\pm1\cdot\4\imath)]=0, \]
and so $\Ker(\5\kappa_{10})=\gen{\theta^*_2,\theta^*_3}=K_2$ since they 
are both $2$-dimensional. So the action of $\twoMat12$ on $\scrg$ induces an 
action of $\twoMat10=N_{\twoMat12}(K_2)$ on $\scrg/K_2\cong\AAA1$.

The formulas for $\9{\dpar{x}}\trp[a,b,c]$, $\9{\sbk{u}}\trp[a,b,c]$, and 
$\9{\sbk{\phi}}\trp[a,b,c]$ follow from arguments similar to those used in 
case (b). 
\end{proof}

\section{The Todd module for \texorpdfstring{$2M_{12}$}{2M12}} 
\label{s:M12}

We are now ready to look at fusion systems that involve the Todd module 
for $2M_{12}$. Throughout the section, we refer to the following 
assumptions:

\begin{Hyp} \label{h:hyp12}
Set $p=3$. Let $\FF$ be a saturated fusion system over a finite $3$-group 
$\SS$, and let $\AA\le\SS$ be an elementary abelian subgroup such that 
$C_{\SS}(\AA)=\AA$. Set $\GG=\Aut_{\FF}(\AA)$ and $\GGnul=O^{3'}(\GG)$, and 
assume that $\rk(\AA)=6$ and $\GG_0\cong2M_{12}$. 
\end{Hyp}

The main result in this section is Theorem \ref{t:M12case}, where we show 
that if $\FF$ satisfies these hypotheses, then either $\AA\nsg\FF$, or 
$\FF$ is isomorphic to the 3-fusion system of the sporadic group $\Co_1$. 

Standard results in the representation theory of $2M_{12}$ show that in the 
above situation, $\AA$ must be the Todd module for $\GG=\GGnul$ or its dual. 
In fact, we can assume in all cases that it is the Todd module.

\begin{Lem} \label{l:A=Td12}
Assume Hypotheses \ref{h:hyp12}. Then $\GG=\GGnul\cong2M_{12}$, $\AA$ is 
the Todd module for $\GG$, and $\AA$ is absolutely irreducible as an 
$\F_3\GG$-module. 
\end{Lem}

\begin{proof} By \S\,4 and Table 5 in \cite{Humphreys}, the only 
$6$-dimensional faithful $\F_3\GGnul$-modules are the Todd module and its 
dual, and they are absolutely irreducible and not isomorphic. Also, 
$\Out(\GGnul)\cong\Out(M_{12})\cong C_2$, and composition with an 
outer automorphism of $\GGnul$ sends the Todd module to its dual. So the 
action of $\GGnul$ on $\AA$ does not extend to any extension of $\GGnul$ by 
an outer automorphism, and $\GG=\GGnul\cdot C_{\GG}(\GGnul)$. As subgroups 
of $\Aut(\AA)$, we have
	\[ C_{\GG}(\GGnul)\le \Aut_{\F_3\GGnul}(\AA)=\{\pm\Id\}
	=Z(\GGnul), \]
where $\Aut_{\F_3\GGnul}(\AA)=\{\pm\Id\}$ since $\AA$ is absolutely 
irreducible. Hence $\GG=\GGnul\cong2M_{12}$. 

Now, $\Out(\GG)\cong\Out(M_{12})\cong C_2$, and by \S\,4 in 
\cite{Humphreys} again, an outer automorphism of 
$\GG$ acts by exchanging the Todd module with its dual. So 
$(\GG,\AA^*)\cong(\GG,\AA)$ as pairs, and we can assume that $\AA$ is the 
Todd module for $\GG$.
\end{proof}

We next check that under Hypotheses \ref{h:hyp12}, $\AA$ is weakly closed 
in $\FF$ and $\SS$ splits over $\AA$. These are easy consequences of Lemma 
\ref{l:nr.J.bl.}.

\begin{Lem} \label{l:w.cl.12}
Assume that $\AA\le \SS$ and $\FF$ satisfy Hypotheses \ref{h:hyp12}, and 
let $\MM$ be a model for $N_{\FF}(\AA)$ (see Proposition 
\ref{p:NF(Q)model}). Then 
\begin{enuma} 

\item $\AA$ is weakly closed in $\FF$ and hence normal in $\SS$, and 

\item $\SS$ and $\MM$ both split over $\AA$. 

\end{enuma}
\end{Lem}

\begin{proof} By Lemma \ref{l:A=Td12}, we have 
$\Aut_{\FF}(\AA)\cong\twoMat12$, and $\AA\cong\scrg$ as 
$\F_3\twoMat12$-modules. 

\smallskip

\noindent\textbf{(a) } If $A^*<N_{\SS}(\AA)$ is such that $A^*\cong 
E_{3^6}$ and $A^*\ne \AA$, then for $x\in A^*\sminus \AA$, $\AA\cap A^*\le 
C_{\AA}(x)$, where $\rk(C_{\AA}(x))\le3$ by Lemma \ref{l:nr.J.bl.} and 
since $c_x^{\AA}$ has order $3$ in $\Aut_{\FF}(\AA)$. Hence 
$\rk(\Aut_{A^*}(\AA))\ge3$, which is impossible since 
$\rk(\Aut_{\SS^*}(\AA))=\rk_3(2M_{12})=2$. So $\AA$ is the only element of 
$\AA^{\FF}$ contained in $N_{\SS}(\AA)$. Hence $\AA$ is 
weakly closed in $\FF$ by Lemma \ref{l:A<|S}. 

\smallskip

\noindent\textbf{(b) } Choose $\theta\in\MM$ such that $c_\theta$ is the 
central involution in $\Aut_{\FF}(\AA)\cong2M_{12}$ (Lemma 
\ref{l:A=Td12}). Then $|\theta|=2$ or $6$, and after replacing $\theta$ 
by $\theta^3$ if necessary, we can assume $|\theta|=2$. Also, $\theta$ 
fixes at least one element in each coset $h\AA$ of $\AA$ in $\MM$ since 
the cosets have odd order. Hence $\MM=\AA C_{\MM}(\theta)$ and $\SS=\AA 
C_{\SS}(\theta)$, while $\AA\cap C_{\MM}(\theta)=1$ since $\theta$ acts as 
$-\Id$ on $\AA$. This proves that $C_{\MM}(\theta)$ and $C_{\SS}(\theta)$ 
are splittings of $\MM$ and $\SS$ over $\AA$. 
\end{proof}

We use throughout this section the notation set up in Section \ref{s:M12a} 
for working with the Todd module for $2M_{12}$, as summarized in Notation 
\ref{n:M12}. In Subsection \ref{s:M12b}, we set up notation for some of the 
subgroups of $\SS$ and $\GG$ that we have to work with. All of this is then 
applied in Subsection \ref{s:M12c} to prove Theorem \ref{t:M12case} 
describing fusion systems satisfying Hypotheses \ref{h:hyp12}.

\begin{Not} \label{n:M12}
Assume Hypotheses \ref{h:hyp12} and Notation \ref{n:transl}. Identify
	\[ \GG=\twoMat12 \cong 2M_{12} 
	\qquad\textup{and}\qquad \AA=\Phi(\scrg)=\Sym_3(\F_3), \]
where $\twoMat12$ is as in Notation \ref{n:transl}(a). Let 
$\NNnul\le\twoMat12$ be as in Notation 
\ref{n:transl}(d), set $\NN=\NNnul\times\{\pm\Id\}$, and let 
	\[ \Theta\:\NNnul \Right4{\cong} \left\{\left. \mxthree{a}bcdef001 
	\,\right|\,a,b,c,d,e,f\in\F_3,~ ae\ne bd \right\} \le \GL_3(\F_3) \]
be the isomorphism defined by \eqref{e:Theta}. Thus 
	\[ \beta(X)=\Theta(\beta)X\Theta(\beta)^t \]
for all $\beta\in\NNnul$ and $X\in\AA$ by Lemma \ref{l:Phi-Th}. Finally, define 
	\[ \TT = \Theta^{-1}(\UT_3(\F_3)) 
	\in\syl3{\NNnul} \subseteq\syl3{\GG}, \]
and set 
	\[ \MM = \AA\rtimes\GG 
	\qquad\textup{and}\qquad \SS = \AA\rtimes \TT \in \syl3\MM. \] 
\end{Not}

\newsubb{Some subgroups of \texorpdfstring{$\GG$}{G} and 
\texorpdfstring{$\SS$}{S}}{s:M12b}

We begin by listing the additional notation that will be needed; in 
particular, notation to describe the subgroups of index $3$ in $\TT$. 

\begin{Not} \label{n:M12-UWQ}
Define
	\[ Z = Z(\SS) = C_{\AA}(\TT) \qquad\textup{and}\qquad A_*=[\TT,\AA]. \]
Define elements $\eta_0,\eta_{\pm1},\eta_\infty,\5\eta\in \TT$ as follows: 
	\[ \eta_k =\Theta^{-1}\left(\mxthree11001k001\right)~
	\textup{(for $k\in\F_3$)}, 
	\qquad \eta_\infty 
	=\Theta^{-1}\left(\mxthree100011001\right),
	\qquad \5\eta 
	=\Theta^{-1}\left(\mxthree101010001\right).  \]
Thus $\TT=\gen{\eta_0,\eta_\infty}$ and $Z(\TT)=\gen{\5\eta}$. For each 
$k\in\F_3\cup\{\infty\}$, set
	\begin{align*} 
	U_k &= \Gen{\5\eta,\eta_k}\le \TT \\ 
	W_k &= \bigl\{ a\in \AA \,\big|\, [a,U_k]\le Z=Z(\SS) \bigr\} 
	\le \AA \qquad (\textup{so}~ W_k/Z=C_{\AA/Z}(U_k)) \\ 
	Q_k &= W_kU_k \le \SS 
	\end{align*}
For $k\in\F_3$, set 
	\[ \calq_k = \bigl\{ Q\le \SS \,\big|\, Q\cap \AA=W_k,~ Q\AA=U_k\AA 
	\bigr\} . \]
In addition, we set 
	\[ \5Q = A_* U_\infty \cong 3^{3+4}. \]
For $1\le i,j\le3$ and $x\in\F_3$, let $a_{ij}^x\in \AA=\Sym_3(\F_3)$ be the 
symmetric $(3\times3)$-matrix with $x$ in positions $(i,j)$ and $(j,i)$ (or 
$2x$ in position $(i,i)$ if $i=j$) and $0$ elsewhere, and set 
$a_{ij}=a_{ij}^1$. 
\end{Not}

The actions of the $\eta_k$ on $\AA$ are described explicitly in 
Table \ref{tbl:TonA}. 
\begin{Table}[ht]
\[ \renewcommand{\arraystretch}{1.8}
\begin{array}{|c||c|c|} \hline
\eta & \vphantom{\bigg|} \eta\left(\mxthree{t}uruvsrsa\right) & 
\left[\eta,\mxthree{t}uruvsrsa\right] \\\hline\hline

\eta_k=\mxthree11001k001~(k\in\F_3) & 
\vphantom{\Bigg|}\mxthree{t-u+v}{u+v+k(r+s)}{r+s}
{u+v+k(r+s)}{v-ks+ak^2}{s+ak}{r+s}{s+ak}a 
& \mxthree{-u+v}{v+k(r+s)}{s}
{v+k(r+s)}{-ks+ak^2}{ak}{s}{ak}0 \\\hline 

\eta_\infty=\mxthree100011001 & \vphantom{\bigg|}
\mxthree{t}{u+r}{r}{u+r}{v-s+a}{s+a}{r}{s+a}a & 
\mxthree{0}{r}{0}{r}{-s+a}{a}{0}{a}0 \\\hline 

\5\eta=\mxthree101010001 & \vphantom{\bigg|}
\mxthree{t-r+a}{u+s}{r+a}{u+s}{v}{s}{r+a}sa & 
\mxthree{-r+a}{s}{a}{s}{0}{0}{a}00 \\\hline 
\end{array}
\]
\caption{} \label{tbl:TonA}
\end{Table}

\begin{Lem} \label{l:Qk}
Assume Notation \ref{n:M12} and \ref{n:M12-UWQ}.
\begin{enuma} 

\item We have 
	\[ Z=\left\{\left.\mxthree{t}00000000\,\right|\,t\in\F_3\right\} 
	\qquad\textup{and}\qquad 
	A_*=\left\{\left.\mxthree{t}uruvsrs0\,\right|\,t,u,v,r,s\in\F_3
	\right\}, \]
and 
	\[ \Aut_{N_{\FF}(A_*)}(\AA) = \Aut_{N_{\GG}(A_*)}(\AA) 
	\qquad\textup{where}\qquad
	N_{\GG}(A_*) = \NN. \]

\item For each $k\in\F_3\cup\{\infty\}$, 
	\begin{align*} 
	W_k &= \begin{cases} 
	\left\{\left.\mxthree{t}uru{-kr}0r00 \,\right|\, 
	r,t,u\in\F_3\right\} & \textup{if $k\in\F_3$} \\[2ex] 
	\left\{\left.\mxthree{t}u0uv0000 \,\right|\, 
	t,u,v\in\F_3\right\} & \textup{if $k=\infty$} 
	\end{cases} & 
	C_{\AA}(U_k) &= \begin{cases} 
	Z & \textup{if $k\in\F_3$} \\
	W_\infty & \textup{if $k=\infty$} 
	\end{cases} \\[2mm]
	Q_k &\cong \begin{cases} 
	3^{1+4}_+ & \textup{if $k\in\F_3$} \\
	E_{3^5} & \textup{if $k=\infty$.} 
	\end{cases} & 
	N_{\SS}(Q_k) &= \begin{cases} 
	\SS & \textup{if $k=0$} \\
	A_*\TT <\SS & \textup{if $k\ne0$.}
	\end{cases} 
	\end{align*}


\item More generally, if $k\in\F_3$ and $Q\in\calq_k$, then 
$N_{\SS}(Q)\ge\AA$ if $k=0$, and $\AA\cap N_{\SS}(Q)=A_*$ if $k\ne0$. 

\end{enuma}
\end{Lem}

\begin{proof} The descriptions of $Z$ and $A_*$ follow immediately from the 
formulas in Notation \ref{n:M12}. From this, we see that $A_*=[\NNnul,\AA]$ 
and hence is normalized by $\NN$. Since $\NN$ is a maximal subgroup 
of $\GG$ by Lemma \ref{l:H-max}, it must be the full normalizer of $A_*$. 

The formulas in point (b) follow easily from those in Table \ref{tbl:TonA}. 
(Note, for each $k\in\F_3\cup\{\infty\}$, that $\TT$ normalizes $Q_k$ since 
it normalizes $U_k$ and $W_k$.) 

If $Q\in\calq_k$ for some $k\in\F_3$, then an element $a\in \AA$ normalizes 
$Q$ if and only if $[a,U_k]\le W_k$, which holds for all $a\in \AA$ if 
$k=0$, but only for $a\in A_*$ if $k=\pm1$. 
\end{proof} 

Note that for each $k\in\F_3$, the subgroup $W_k\gen{\5\eta,a_{23}\eta_k}$ 
lies in $\calq_k$ (since $(a_{23}\eta_k)^3\in C_{\AA}(\eta_k)\le W_k$), but 
is not extraspecial since $[\5\eta,a_{23}\eta_k]=[\5\eta,a_{23}]\in 
W_k\sminus Z$. Thus members of the $\calq_k$ need not be extraspecial. 
However, as shown in the next lemma, all subgroups of $\SS$ not in $\AA$ 
and isomorphic to $E_{3^5}$ or $3^{1+4}_+$ are members of $\calq_k$ for 
some $k$. 

\begin{Lem} \label{conj-Q0}
Assume Notation \ref{n:M12} and \ref{n:M12-UWQ}. 
\begin{enuma} 

\item There are exactly three abelian subgroups of $\SS$ of order $3^5$ 
not contained in $\AA$, and all of them are conjugate to $Q_\infty\cong 
E_{3^5}$ by elements of $\AA\sminus A_*$.

\item If $P\le \SS$ is extraspecial of order $3^5$, then 
$Z(P)=Z$, and $P\in\calq_k$ for some $k\in\F_3$. If in addition, $P$ is 
weakly closed in $N_{\FF}(Z)$, then $P=Q_0$. 

\item For each saturated fusion system $\cale$ over $\SS$ and each 
$k\in\F_3$, $Q_k$ is $\cale$-centric. 

\end{enuma}
\end{Lem}

\begin{proof} \textbf{(a) } Assume $B\le \SS$ is abelian and such that 
$B\nleq \AA$ and $|B|=3^5$. For each $\eta\in \SS\sminus \AA$, 
$\rk(C_{\AA}(\eta))\le3$ by Lemma \ref{l:nr.J.bl.}, so $\rk(B\AA/\AA)=2$ 
and $\rk(B\cap \AA)=3$. Thus $B\AA=U_k\AA$ for some 
$k\in\F_3\cup\{\infty\}$ such that $\rk(W_k)\ge\rk(C_{\AA}(U_k))\ge3$, and 
$k=\infty$ by Lemma \ref{l:Qk}(a). By the same lemma, $B\cap\AA=W_\infty$. 

Thus $B=W_\infty\gen{b_1\5\eta,b_2\eta_\infty}$ for some $b_1,b_2\in \AA$ 
uniquely determined modulo $W_\infty$. Since $[\5\eta,\eta_\infty]=1$ and 
$\AA\nsg \SS$, we have 
	\[ 1 = [b_1\5\eta,b_2\eta_\infty] 
	= b_1(\5\eta b_2\5\eta^{-1})(\eta_\infty b_1^{-1}\eta_\infty^{-1}) 
	b_2^{-1} 
	= [\5\eta,b_2][b_1,\eta_\infty], \]
and hence 
	\[ [\5\eta,b_2]=[\eta_\infty,b_1]\in[\5\eta,\AA]\cap[\eta_\infty,\AA]
	= \gen{a_{12}}. \]
So by Table \ref{tbl:TonA} again, $b_1\equiv a_{13}^x$ and 
$b_2\equiv a_{23}^x$ (mod $W_\infty$) for some $x\in\F_3$.

In particular, there are at most three subgroups of $\SS$ isomorphic to 
$E_{3^5}$ and not in $\AA$. Since $N_{\SS}(Q_\infty)=A_*\TT$ has index 3 in 
$\SS$, there are exactly three such subgroups, and they are all conjugate 
to $Q_\infty$ by elements of $\AA\sminus A_*$. More precisely, the three 
subgroups $W_\infty\gen{a_{13}^x\5\eta,a_{23}^x\eta_\infty}$ for $x\in\F_3$ 
all have the form $\9{\beta}Q_\infty$ for some $\beta\in\gen{a_{33}}$. 

\smallskip

\noindent\textbf{(b) } Assume that $P\le \SS$ is extraspecial of order 
$3^5$, and set $P_0=P\cap \AA$. Then $P_0$ and $P/P_0$ are both 
elementary abelian (since $[P,P]=Z(P)\le P_0$), and hence $P_0\cong E_{27}$ 
and $P/P_0\cong E_9$. So $P\AA=U_k\AA$ for some 
$k\in\F_3\cup\{\infty\}$, and $Z(P)\le C_{\AA}(U_k)$. Since 
$U_k=\gen{\5\eta,\eta_k}$ and $C_{\AA}(\5\eta)=W_\infty$, this means 
that $Z(P)\le C_{W_\infty}(\eta_k)$, and hence $Z(P)=Z$ if $k\in\F_3$ 
(while $C_{\AA}(U_\infty)=W_\infty$). So if $k\ne\infty$, then 
$[P_0,U_k]=Z$, and hence $P_0\le W_k$ in this case, with equality since 
$\rk(W_k)=3$ for each $k$ (Lemma \ref{l:Qk}). Thus $P\in\calq_k$ if 
$k\in\F_3$. 

Conjugation by the element $\mxtwo{-I}001\in N$ lies in 
$\Aut_{\FF}(\SS)=\Aut_{N_{\FF}(Z)}(\SS)$, and its action on $\SS$ exchanges 
the sets $\calq_1$ and $\calq_{-1}$. So no member of either of these is 
weakly closed in $N_{\FF}(Z)$. Each member of $\calq_0$ has the form 
$Q=W_0\gen{g_1\eta_0,g_2\5\eta}$ for some $g_1,g_2\in \AA$, and $-\Id\in N$ 
sends $Q$ to $W_0\gen{g_1^{-1}\eta_0,g_2^{-1}\5\eta}$. Since 
$c_{-\Id}\in\Aut_{\FF}(\SS)=\Aut_{N_{\FF}(Z)}(\SS)$, $Q$ is weakly closed 
only if $g_i\equiv g_i^{-1}$ (mod $W_0$) for $i=1,2$, which occurs only if 
$g_1,g_2\in W_0$ and hence $Q=Q_0$. Thus $Q_0$ is the only member of 
$\calq_0\cup\calq_1\cup\calq_{-1}$ that could be weakly closed in 
$N_{\FF}(Z)$. 

If $k=\infty$, then 
	\[ Z(P) \le C_{\AA}(U_\infty) \cap [\5\eta,\AA] \cap [\eta_\infty,\AA] 
	= \gen{a_{12}} \]
by Table \ref{tbl:TonA}, and so 
	\[ P_0 \le \bigl\{ a\in \AA \,\big|\, [U_\infty,a]\le Z(P) \bigr\} 
	= W_\infty \]
with equality since $\rk(W_\infty)=3=\rk(P_0)$. But $[U_\infty,W_\infty]=1$, 
so $W_\infty\le Z(P)$, a contradiction. 

\smallskip

\noindent\textbf{(c) } For each $k\in\F_3$ and each $Q\in\calq_k$, 
$C_{\SS}(Q)\le C_{\AA U_k}(W_k)=\AA$ since $Q_k=U_kW_k$ is extraspecial 
(Lemma \ref{l:Qk}(b)), and hence $C_{\SS}(Q)=C_{\AA}(U_k)=Z$ by the same 
lemma. Since $(Q_k)^{\FF}\subseteq\calq_0\cup\calq_1\cup\calq_2$ by (b), 
this proves that $Q_k$ is $\cale$-centric for each saturated fusion system 
$\cale$ over $\SS$. 
\end{proof}

Point (c) in Lemma \ref{conj-Q0} is not true if one replaces $Q_k$ (for 
$k\in\F_3$) by $Q_\infty$. If $\FF$ and $\SS$ satisfy Hypotheses 
\ref{h:hyp12}, then one can show that $\5Q\nsg C_{\FF}(W_\infty)$, and 
that $\Out_{C_{\FF}(W_\infty)}(\5Q)\cong2A_4$. (Since $\FF$ is 
isomorphic to the fusion system of $\Co_1$ by Theorem \ref{t:M12case}, 
this follows from the structure of $C_{\Co_1}(W_\infty)\cong\5Q.2A_4$.) 
The subgroup $\5Q$ contains exactly four elementary abelian subgroups 
of rank $5$ (the three described in Lemma \ref{conj-Q0} and $A_*$), and 
they are permuted transitively by $\Out_{C_{\FF}(W_\infty)}(\5Q)$. So 
$Q_\infty\in(A_*)^{\FF}$, and hence is not $\FF$-centric.

\newsubb{Fusion systems involving the Todd module for 
\texorpdfstring{$2M_{12}$}{2M12}}{s:M12c}

We now begin to apply results from Section \ref{s:general}. Recall that our 
goal is to describe all fusion systems that satisfy Hypotheses 
\ref{h:hyp12} with $\AA\nnsg\FF$.
 
\begin{Prop} \label{p:F=<N,C>12}
Assume Hypotheses \ref{h:hyp12} with $\GG=\twoMat12$ and $\AA$ as in 
Notation \ref{n:M12}, and set $Z=Z(\SS)$. Then 
$\FF=\gen{C_{\FF}(Z),N_{\FF}(\AA)}$.
\end{Prop}

\begin{proof} Assume otherwise. By Proposition \ref{p:R=C(Z)}, there are 
subgroups $X\in Z^{\FF}$ and $R\in\EE\FF$ such that $X\nleq \AA$, 
$R=C_{\SS}(X)=N_{\SS}(X)$, and $Z=\alpha(X)$ for some $\alpha\in\Aut_{\FF}(R)$. 
Fix $x\in X\sminus \AA$. In all cases, $R\cap \AA=C_{\AA}(X)=C_{\AA}(x)$, 
since $|X|=|Z|=3$ and hence $X=\gen{x}$. Also, $|x|=3$ since $x\in X\in 
Z^{\FF}$ where $Z$ has order $3$. Set $R_0=R\cap \AA$. 

\smallskip

\noindent\textbf{Case 1: } Assume first that $|R\AA/\AA|=3$, so that 
$R\AA=\AA\gen{x}$ and $R=C_{\SS}(X)=C_{\AA}(x)\gen{x}$. Then 
$\Aut_{\AA}(R)\cong C_{\AA/R_0}(x)\cong E_{3^m}$, where $m$ is the number 
of Jordan blocks of length at least $2$ for the action of $x$ on $\AA$, and 
$m=2$ by Lemma \ref{l:nr.J.bl.}. 

Thus $|\Out_{\AA}(R)|=9$. Since $\Out_{\AA}(R)$ acts trivially on $R_0$ and 
$|R:R_0|=3$, this contradicts Lemma \ref{l:filtered-c}.

\smallskip

\noindent\textbf{Case 2: } Assume that $|R\AA/\AA|=9$, and hence that 
$\Aut_R(\AA)=U_k$ for some $k\in\F_3\cup\{\infty\}$. If $k\in\F_3$, then 
$Z=C_{\AA}(R)<C_{\AA}(x)$ by Lemma \ref{l:Qk}(b), and hence 
$Z\le[R,C_{\AA}(x)]\le[R,R]$. Since $X\nleq[R,R]$, no automorphism of $R$ 
sends $X$ to $Z$.

Now assume $k=\infty$, so $R_0=C_{\AA}(x)=C_{\AA}(R)\cong E_{27}$ by Lemma 
\ref{l:Qk}(b) again. Also, $\Out_{\AA}(R)\cong C_{\AA/R_0}(U_\infty)\cong 
E_9$ (see Table \ref{tbl:TonA}). So by Lemma \ref{l:filtered}(b), for each 
characteristic subgroup $P\le R$, either $|P|\ge3^4$ or $|R/P|\ge3^4$. 
Since $|R|=3^5$, and since $R$ is not extraspecial by Lemma 
\ref{conj-Q0}(b), this implies that $R\cong E_{3^5}$. 

Set $B=\Out_{\AA}(R)\cong E_9$, so that $B\le\Out_{\SS}(R)$. Let 
$H<\outf(R)$ be a strongly 3-embedded subgroup that contains $\Out_S(R)$ 
(recall $R\in\EE\calf$), fix $g\in \outf(R)\sminus H$, and set 
$L=\gen{B,\9gB}$. Then $L\nleq H$ and $3\mid|H\cap L|$, so by Lemma 
\ref{l:str.emb.}(b), the subgroup $H\cap L$ is strongly $p$-embedded in 
$L$. 

Since $\rk(C_R(B))=3$ and $\rk(R)=5$, we have 
$\rk(C_R(L))=\rk\bigl(C_R(B)\cap C_R(\9gB)\bigr)\ge1$. Also, 
$\rk(R/C_R(L))\ge4$ by Lemma \ref{l:filtered}(b) again, so $\rk(C_R(L))=1$, 
and $R/C_R(L)$ is a faithful $4$-dimensional representation of $L$. For 
each $x\in B^\#$, $\rk([x,R])=\rk([x,U_\infty])=2$, and so $[x,R/C_R(L)]$ 
has rank 1 or 2, and $x$ acts on $R/C_R(L)$ with Jordan blocks of lengths 
$2+2$ or $2+1+1$. By Proposition \ref{p:str.emb.4}, 
$L\cong\SL_2(9)$ with the natural action on $R/C_R(L)$, and hence 
$\rk([x,R/C_R(L)])=2$ for each $x\in B^\#$. Thus $C_R(L)\cap[x,R]=1$ for 
each $x\in B^\#$. But this is impossible: from Table \ref{tbl:TonA}, we see 
that the subgroups $[x,R]$ are precisely the four subgroups of rank $2$ in 
$W_\infty\cong E_{27}$ that contain $\gen{a_{12}}$, and hence each element 
of $W_\infty$ lies in at least one of them. 

\smallskip

\noindent\textbf{Case 3: } Finally, assume that $|R\AA/\AA|>9$. Then 
$R\AA/\AA=\SS/\AA\cong3^{1+2}_+$, and $\AA X=\AA\gen{\5\eta}$. From Table 
\ref{tbl:TonA}, we see that $R_0=C_{\AA}(\5\eta)=Z\gen{a_{12},a_{22}}\cong 
E_{27}$. 

From the formulas in Table \ref{tbl:TonA} again, we see that 
$Z\gen{a_{12}}\le[\TT,R_0]\le[R,R]$, and hence that $Z\le[R,[R,R]]$. Since 
$[R,[R,R]]\le\AA$, it does not contain $X$, so no automorphism of 
$R$ sends $X$ to $Z$, contradicting our assumptions.
\end{proof}

We next show that $Q_0$ is normal in $C_{\FF}(Z)$. The following lemma 
is a first step towards doing this. From now on, we set $\QQ=Q_0$, 
since this subgroup plays a central role in studying these fusion 
systems satisfying Hypotheses \ref{h:hyp12}.

\begin{Lem} \label{Q0-w.cl.}
Assume Hypotheses \ref{h:hyp12}, and Notation \ref{n:M12} and 
\ref{n:M12-UWQ}, and set $\QQ=Q_0$. Then 
\begin{enuma} 
\item $\QQ$ is weakly closed in $\FF$; 

\item $\QQ$ is normal in $N_{N_{\FF}(\AA)}(Z)$; 

\item $C_{\GG}(Z)\cong E_9\rtimes\GL_2(3)$ and 
$N_{\GG}(U_0)=N_{\GG}(Z)\cong (E_9\rtimes\GL_2(3))\times C_2$; and

\item $Z$ and $W_0$ are the only proper nontrivial subspaces of $\AA$ 
invariant under the action of $C_{\GG}(Z)$. 

\end{enuma}
\end{Lem}

\begin{proof} \textbf{(c) } Since $Z=C_{\AA}(U_0)$ (see Table 
\ref{tbl:TonA}), we have $N_{\GG}(U_0)\le N_{\GG}(Z)$. Also, 
$N_{\GG}(U_0)\ge N_{\NN}(U_0)\cong \TT\rtimes E_8$, so the index of 
$N_{\GG}(U_0)$ in $\GG$ divides $880$. By \cite[Lemma 7.12 \& Exercise 
7.36]{Griess}, the orbits of $\GG$ acting on the projective space $P(\AA)$ 
have lengths 132, 220, and 12, so $Z$ must be in an orbit of length $220$, 
and hence $|N_{\GG}(Z)|=3^2\cdot96=|\NN|$. 


Recall (Lemma \ref{l:nr.J.bl.}) that there are two conjugacy classes of 
elements of order $3$ in $\GG$, differing by the number of Jordan blocks 
for their actions on $\AA$. Thus all elements in $U_0^\#$ and $U_\infty^\#$ 
are in one of the classes, while elements in $U_k\sminus\gen{\5\eta}$ for 
$k\in\{\pm1\}$ are in the other. Since $C_{\AA}(U_0)=Z$ while 
$C_{\AA}(U_\infty)=W_\infty$ by Lemma \ref{l:Qk}(b), $U_0$ and $U_\infty$ 
are not $\GG$-conjugate. 

As noted earlier (see \cite[\S4]{Humphreys}), while $\AA$ is not isomorphic 
to its dual $\AA^*$ as $\F_3\GG$-modules, the pairs $(\GG,\AA)$ and 
$(\GG,\AA^*)$ are isomorphic via an outer automorphism 
$\alpha\in\Aut(\GG)\sminus\Inn(\GG)$. Hence by Table \ref{tbl:TonA}, 
	\[ \rk(C_{\AA}(U_0))=1 \quad\textup{and}\quad 
	\rk(C_{\AA}(\alpha(U_0))) = \rk(C_{\AA^*}(U_0)) 
	= \rk(\AA/[U_0,\AA])=3, \] 
so $\alpha(U_0)$ is not $\GG$-conjugate to $U_0$. Since all elements of 
order $3$ in $\alpha(U_0)$ are conjugate to each other, $\alpha(U_0)$ must 
be $\GG$-conjugate to $U_\infty$. Thus $\alpha$ exchanges the classes of 
$U_0$ and $U_\infty$. 

By the description of the action of $\NN$ on $\AA$ in Notation 
\ref{n:M12}, $\NN$ normalizes the subgroup $A_*$ of index $3$ in 
$\AA$. So it also normalizes a subgroup of order $3$ in the dual space 
$\AA^*$, and hence $\alpha(\NN)\le N_{\GG}(X)$ for some $X\le\AA$ of 
order $3$. The length of the orbit of $X$ under the action of $\GG$ divides 
$|\GG:\NN|=220$, so $X$ is in the orbit of $Z$ by earlier remarks, and 
$\alpha(\NN)=N_{\GG}(X)$ is $\GG$-conjugate to $N_{\GG}(Z)$. Thus 
$N_{\GG}(Z)\cong\NN\cong(E_9\rtimes\GL_2(3))\times C_2$. Since $\NNnul$ 
acts via the identity on $\AA/A_*$, a similar argument shows that 
$C_{\GG}(Z)\cong\NNnul$. Finally, since $U_\infty=O_3(\NN)$ and 
$\alpha(U_\infty)$ is $\GG$-conjugate to $U_0$, we get that 
$O_3(N_{\GG}(Z))$ is $\GG$-conjugate to $U_0$, so 
$|N_{\GG}(U_0)|=|N_{\GG}(\alpha(U_\infty))|\ge|N_{\GG}(Z)|$. Since 
$N_{\GG}(U_0)\le N_{\GG}(Z)$, they must be equal.

\smallskip

\noindent\textbf{(d) } Since $C_{\GG}(Z)$ has index $2$ in 
$N_{\GG}(Z)=N_{\GG}(U_0)$ by (c), $Z$ and $W_0$ are both invariant under 
its action on $\AA$ (recall $W_0/Z=C_{\AA/Z}(U_0)$ by definition). We 
must show that there are no other invariant subgroups. 

As noted in the proof of (c), the action of $C_{\GG}(Z)$ on $\AA$ 
is (up to isomorphism) dual to the action of $\NNnul\cong E_9\rtimes\GL_2(3)$ 
on $\AA$. Set 
	\[ B=\Theta^{-1}\bigl(\bigl\{\mxtwo{A}001\,\big|\,
	A\in\GL_2(3)\bigr\}\bigr)<\NNnul. \]
Then $\AA$ splits as a direct sum of the three irreducible $\F_3B$-submodules 
	\[ W_\infty=\left\{ \left. \mxthree{a}b0bc0000 \right|  
	a,b,c\in\F_3 \right\}, \quad
	\left\{ \left. \mxthree00x00yxy0 \right|  
	x,y\in\F_3 \right\}, \quad
	\left\{ \left. \mxthree00000000z \right| z\in\F_3 \right\}, \]
of which only $W_\infty$ is $\NNnul$-invariant. Since $\NNnul=U_\infty B$, it now 
follows that the only proper nontrivial $\F_3\NNnul$-submodules are $W_\infty$ 
and $A_*$, and hence (after dualizing) that $\AA$ also has only two proper 
nontrivial $\F_3C_{\GG}(Z)$-submodules. 

\smallskip

\noindent\textbf{(b) } Since $\MM=\AA\rtimes\GG$ is a model for 
$N_{\FF}(\AA)$ (Lemma \ref{l:w.cl.12}(b)), it suffices to show that $\QQ\nsg 
N_{\MM}(Z)=\AA N_{\GG}(Z)$. Since $[\QQ,\AA] = [U_0,\AA] = W_0\le \QQ$, 
where the second equality holds by Table \ref{tbl:TonA}, we have $\AA\le 
N_{\MM}(\QQ)$. Also, $N_{\GG}(Z)=N_{\GG}(U_0)$ by (c), this group 
normalizes $W_0$ since $U_0$ normalizes $W_0=[U_0,\AA]$, and hence 
$N_{\GG}(Z)$ also normalizes $\QQ=U_0W_0$. So $\QQ\nsg\AA N_{\GG}(Z)$. 

\smallskip

\noindent\textbf{(a) } We first check that 
	\beqq \QQ^{\FF}\cap\calq_0 = \{\QQ\}. \label{e:Q0^F_calq0} \eeqq
Assume otherwise: assume $P\in \QQ^{\FF}\cap\calq_0$ and $P\ne \QQ$. By 
Lemma \ref{p:Hom(NSP,S)}, there is 
$\varphi\in\Hom_{\FF}(N_{\SS}(P),\SS)$ such that $\varphi(P)=\QQ$, and 
$\AA\le N_{\SS}(P)$ by Lemma \ref{l:Qk}(c). Then $\varphi(\AA)=\AA$ since 
$\AA$ is weakly closed (Lemma \ref{l:w.cl.12}(a)), and $\varphi(Z)=Z$ since 
$Z=Z(N_{\SS}(P))=Z(\SS)$. (Note that $N_{\SS}(P)=U_0\AA$ or $\SS$.) Thus 
$\varphi\in\Mor(N_{N_{\FF}(\AA)}(Z))$, so $\varphi(\QQ)=\QQ$ by (b), 
contradicting our assumption that $P\ne \QQ$. 

If $\QQ$ is not weakly closed, then by Lemma \ref{l:nnorm}, there are 
$R\in\EE\FF\cup\{\SS\}$, $\alpha\in\Aut_{\FF}(R)$, and $P\le R$ such that 
$R\ge \QQ$, $P=\alpha(\QQ)\ne \QQ$, and $R=N_{\SS}(P)$. Then 
$P\notin\calq_0$ by \eqref{e:Q0^F_calq0}, so by Lemma \ref{conj-Q0}(b), 
there is $k\in\{\pm1\}$ such that $P\in\calq_k$. By Lemma \ref{l:Qk}(c) 
again, $R\cap\AA=N_{\SS}(P)\cap\AA=A_*$. Also, $R\AA$ contains both 
$\QQ\AA=U_0\AA$ and $P\AA=U_k\AA$, so $R\AA=\SS$ and $|\SS/R|=3$. In 
particular, $R\nsg \SS$. 


We next claim that 
	\beqq \beta\in\Aut_{\FF}(R),~ \beta(A_*)=A_* ~\implies~ 
	\beta(\QQ)=\QQ. \label{e:a(A0)=A0} \eeqq
Fix such a $\beta$. Since $\beta(A_*)=A_*$ and $\AA$ is weakly 
closed, $\beta|_{A_*}$ extends to some 
$\5\beta\in\Aut_{\FF}(\AA)=\GG$ by Lemma \ref{l:A-w.cl.}(b). Also, 
$\beta(Z)=Z$ since $Z=Z(R)$, so $\5\beta\in N_{\GG}(Z)=N_{\GG}(U_0)$ by 
(c), and $\5\beta$ normalizes $C_{\AA/Z}(U_0)=W_0/Z$. So $\beta(W_0)=W_0$, 
hence $\beta(\QQ)\cap A_*=\beta(W_0)=W_0$, and $\beta(\QQ)\in\calq_0$ by 
Lemma \ref{conj-Q0}(b) again. So $\beta(\QQ)=\QQ$ by \eqref{e:Q0^F_calq0}, 
proving \eqref{e:a(A0)=A0}.

In particular, $\alpha(A_*)\ne A_*=R\cap\AA$ by \eqref{e:a(A0)=A0} and 
since $\alpha(\QQ)\ne \QQ$, so $\alpha(A_*)\nleq \AA$, and by Lemma 
\ref{conj-Q0}(a), $\alpha(A_*)$ is one of the three subgroups 
$\AA$-conjugate to $Q_\infty$. Since $R\nsg\SS$, all three of these 
subgroups are in the $\Aut_{\FF}(R)$-orbit of $\QQ$. In particular, 
$Q_\infty=U_\infty W_\infty\le R$, so $R\ge U_\infty \QQ A_*=\TT A_*$, with 
equality since both have index $3$ in $\SS$.

Let $\Aut_{\FF}^0(R)\le\Aut_{\FF}(R)$ be the stabilizer of $A_*$. We just 
saw that the $\Aut_{\FF}(R)$-orbit of $A_*$ consists of $A_*$ together with 
the three subgroups conjugate to $Q_\infty$ by elements of $\AA$. 
So $\Aut_{\FF}^0(R)$ has index $4$ in $\Aut_{\FF}(R)$. 
By \eqref{e:a(A0)=A0}, $\beta(\QQ)=\QQ$ for each $\beta\in\Aut_{\FF}^0(R)$, 
and hence the $\Aut_{\FF}(R)$-orbit of $\QQ$ has order at most $4$. Since 
$R\nsg\SS$, all three members of the $\AA$-conjugacy class of $P\in\calq_k$ 
lie in this orbit. Also, the element 
$\Theta^{-1}\left(\mxtwo{-I}001\right)\in N_{\NNnul}(\TT)\le\MM$ exchanges 
the two classes $\calq_1$ and $\calq_{-1}$ and normalizes $R=\TT A_*$, so 
the $\Aut_{\FF}(R)$-orbit of $\QQ$ has at least three members from each of 
these classes. Since this contradicts the earlier observation that the 
orbit has at most four members, we conclude that $\QQ$ is weakly closed in 
$\FF$. 
\end{proof}


\newcommand{\imposs}{}  

We are now ready to prove that $\QQ\nsg C_{\FF}(Z)$. 

\begin{Lem} \label{Q<|CF(Z)-M12}
Assume Hypotheses \ref{h:hyp12} and Notation \ref{n:M12-UWQ}, and again 
set $\QQ=Q_0$. Then $\QQ\nsg C_{\FF}(Z)$. 
\end{Lem}

\begin{proof} For $1\le i\le j\le3$, let $A_{ij}\le\AA$ be the subgroup of 
those elements represented by symmetric $(3\times3)$-matrices with entries 
$0$ except possibly in positions $(i,j)$ and $(j,i)$. We also set 
$\Delta=W_0A_{22}=W_\infty A_{13}$, since this ``triangular 
shaped'' subgroup appears frequently in the arguments below. 

Define inductively $Z=B_0<B_1<B_2<B_3<B_4=B=\QQ$ by setting 
$B_i/B_{i-1}=C_{\QQ/B_{i-1}}(\SS)$. Thus 
	\[ B_0=A_{11},\quad B_1=B_0A_{12},\quad 
	B_2=W_0=B_1A_{13},\quad B_3=B_2\gen{\5\eta},\quad 
	B_4=\QQ=B_3\gen{\eta_0}, \]
and $B_i\nsg\SS$ for each $i$ since $Z$ and $\QQ$ are normal. 

\newcommand{\refnr}[1]{{\bf(\!(#1)\!)}}

Assume $\QQ\nnsg C_{\FF}(Z)$. Then $\QQ/Z\nnsg C_{\FF}(Z)/Z$ by 
Lemma \ref{l:Q/P<|F/P} and since $Z\le Z(C_{\FF}(Z))$. 
By Proposition \ref{p:not.str.cl.1}, applied 
with $C_{\FF}(Z)/Z$ and $\QQ/Z$ in the role of $\calf$ and $A$, there are 
$\ell\le2$, $R\le \SS$, and $\alpha\in\Aut_{C_\calf(Z)}(R)$ such that 
\begin{enumerate}[\bf(\!(1)\!) ]

\item $R\ge B_{\ell+1}$, $\alpha(B_i)=B_i$ for all $i\le\ell$, and 
$X\defeq\alpha(B_{\ell+1})\nleq \QQ$; 

\item $R=N_S(X)$ and $R/B_\ell=C_{S/B_\ell}(X/B_\ell)$; and

\item if $\ell=0$, then $R\in\EE[*]{C_\calf(Z)}$ and 
$R/Z\in\EE[*]{C_\calf(Z)/Z}$. 

\end{enumerate}
Note, in \refnr3, that $R\in\EE[*]{C_\calf(Z)}$ by Lemma \ref{l:(F/Z)e} 
together with Proposition \ref{p:not.str.cl.1}.

We will show that this is impossible. 
Fix an element $t\in X\sminus \QQ=\alpha(B_{\ell+1})\sminus \QQ$. Thus 
$X=B_\ell\gen{t}$ (recall $B_\ell\nsg\calf_\ell$). 
Set $R_0=R\cap \AA$, so that $R_0=N_{\AA}(X)$ and 
$R_0/B_\ell=C_{\AA/B_\ell}(X/B_\ell)$ by \refnr2. We claim that 
\begin{enumerate}[\bf(\!(1)\!) ] \setcounter{enumi}{3}

\item $R\ngeq \AA$ and hence $R_0\ne \AA$ and $t\notin\AA$;

\item $|t|=3$; and 

\item $t\notin\5\eta \AA$ implies $R\le \AA\gen{t,\5\eta}$. 

\end{enumerate}

To see these, note first that if $R\ge\AA$, then 
$\alpha\in\Aut_{\calf_\ell}(R)\subseteq\Mor(C_{N_{\FF}(\AA)}(Z))$ 
since $\calf_\ell\le C_{\FF}(Z)$ and $\AA$ is weakly closed (Lemma 
\ref{l:w.cl.12}(a)). So $\alpha(R\cap \QQ)=R\cap \QQ$ since $\QQ\nsg 
N_{N_{\FF}(\AA)}(Z)$ by Lemma \ref{Q0-w.cl.}(b), contradicting the 
assumption that $t\in\alpha(B_{\ell+1})\sminus \QQ$. Hence $R\ngeq\AA$. 
Also, $B_\ell\le\AA$, while $X=B_\ell\gen{t}\nleq\AA$ since 
$\AA\ne R_0=N_{\AA}(X)$, so $t\notin\AA$, finishing the proof of 
\refnr4. 

Since $B_{\ell+1}\le \QQ$ has exponent $3$, so does $X=\alpha(B_{\ell+1})$. 
Hence $|t|=3$, proving \refnr5. If $t\notin\5\eta\AA$, then $R\AA/\AA\le 
C_{\SS/\AA}(t) =\gen{t\AA,\5\eta\AA}$, so $R\le\AA\gen{t,\5\eta}$, 
proving \refnr6.

Since $t\in\SS\sminus\AA$ by \refnr4, and each element in $\SS\sminus\AA$ 
is $\SS$-conjugate to an element of $\eta\AA$ for $\eta=\5\eta^{\pm1}$ or 
$\eta_k^{\pm1}$ for $k\in\F_3\cup\{\infty\}$, we can arrange that 
$t\in\eta\AA$ for $\eta\in\{\5\eta,\eta_\infty,\eta_0,\eta_{\pm1}\}$. 
The proof now splits up naturally into different cases, depending on the 
class $t\AA$ and on $\ell$. The following arguments, 
covering all possible pairs $(t\AA,\ell)$, are summarized in Table 
\ref{tbl:Q<|C(Z):M12}. 


\begin{small} 
\begin{Table}[ht] 
\[ \renewcommand{\arraystretch}{1.8} 
\renewcommand{\arraycolsep}{4pt}
\begin{array}{|c||c|c|c|} \hline

\ell & \ell=0 & \ell=1 & \ell=2 \\\hline

B_{\ell} & \left\{\mxthree*00000000\right\} & 
\left\{\mxthree**0\ast00000\right\} & 
\left\{\mxthree***\ast00*00\right\}_{\vphantom{|}} \\\hline\hline

\vcenter{\hbox{\rotatebox{60}{$t\in\5\eta\AA$}}} & 
\begin{cstack}[1.8]
R_0 = \Delta = \left\{\mxthree***\ast*0*00\right\} \\
\dbl{\alpha^{-1}(t)\in[R,R] ~\implies~ t\in[R,R] ~\implies}
{R=\Delta\gen{t,u,v} ~\textup{with}~ u\in \eta_0\AA, ~ 
v\in\eta_\infty\AA} \\
\left.\dbl{Z(R/B_1)=R_0\gen{t}/B_1\cong E_{27}}{Z(R/Z\gen{t})
=B_1\gen{t}/Z\gen{t} \cong C_3}\right\} \textup{imposs.}
\end{cstack} &
\begin{cstack}[1.4] 
R_0 =A_* = \left\{\mxthree***\ast****0\right\}^{\vphantom{|}} \\
\dbl{A_*\cong E_{3^5},~ \alpha(A_*)\nleq\AA}
{\hfill\implies R\ge A_*Q_\infty} \\
\alpha(A_*)\nsg R\implies R=A_*\TT \\
\textup{$R\ge \QQ$ w.cl.}\Rightarrow\alpha(\QQ)=\QQ
\end{cstack} &
\begin{cstack} 
R_0=\AA \\ 
\dbl{\textup{impossible}}{\textup{by \refnr4}}
\end{cstack} \\\hline


\vcenter{\hbox{\rotatebox{60}{$t\in\eta_0\AA$}}} & 
\begin{cstack}[1.6] 
R_0=\left\{\mxthree***\ast00*0*\right\}, ~ R=R_0\gen{t,u} \\
\textup{where} ~u\in\5\eta A_*,~ [t,u]=t^3=u^3=1 \\
\tpl{[A_{23},R]\le B_2\gen{t} = Z_2(R)}{[A_{23},B_2\gen{t}]\le ZA_{13} = 
Z(Z_2(R)}{\implies R\notin\EE[*]{C_\calf(Z)}~
\textup{(Lem. \ref{QcharP})}~\textup{imp. by \refnr3}} 
\end{cstack} &
\begin{cstack}[1.6] 
R_0=\left\{\mxthree***\ast*0*0*\right\} \\
R=R_0\gen{t,u} ~\textup{with}~ u\in\5\eta A_* \\ 
\dbl{\implies~\alpha(t)\in B_2\le[R,R]}{\textup{while} ~t\notin[R,R]} 
\end{cstack} & 
\begin{cstack} 
R_0=\AA \\ 
\dbl{\textup{impossible}}{\textup{by \refnr4}}
\end{cstack} \\\hline

\vcenter{\hbox{\rotatebox{80}{$t\in\eta_k\AA$, $k\in\{\pm1,\infty\}$} }} &
\begin{cstack}[1.0] 
R_0=W_k = \left\{\mxthree**r\ast{-kr}0r00\right\} 
~ \textup{or}~ \left\{\mxthree**0\ast*0000\right\} \\
R=R_0\gen{t,u}~\textup{where}~ [t,u]\in Z \\
R\in\calq_k,~ R/Z\cong E_{3^4} \\
\textup{Prop.\ref{p:str.emb.4} $\Rightarrow$} 
\Aut_{C_\calf(Z)/Z}(R/Z)\cong \textit{(P)SL}_2(q) \\
\dbl{|\Aut_{\SS/Z}(R/Z)|= |N_{\SS}(R)/R|=3^3} 
{\implies~ \Aut_{\SS/Z}(R/Z)\nleq\Aut_{C_\calf(Z)/Z}(R/Z)}
\end{cstack} 
& \multicolumn{2}{c|}{ \begin{cstack}[1.0] 
R_0=\Delta=\left\{\mxthree***\ast*0*00\right\}^{\vphantom{|}} \\
R=\Delta\gen{t,u} ~\textup{for some}\\
t\in\eta_k\AA,~ u\in\5\eta\AA,~ [t,u]\in Z \\
\textup{Set}~ x=\alpha^{-1}(t)\in B_{\ell+1}\sminus B_\ell: 
~\textup{then} \\
\dbl{C_R(x)\cong C_R(t)\le C_\Delta(\eta_k)\gen{t,u'} ~\textup{for}}
{\textup{$u'\in u\Delta$: nonabelian of order $3^4$}} \\
\dbl{C_R(x)\ge\Delta\cong E_{3^4} \quad \textup{if $\ell=1$}}
{C_R(x)\ge W_\infty\gen{x}\cong E_{3^4} \quad \textup{if $\ell=2$}}
\end{cstack} } \\\hline

\end{array}
\]
\caption{In all cases, $R_0=R\cap \AA$, where $R/B_\ell=C_{\SS/B_\ell}(t)$.
In the matrices used to describe $R_0$, a ``$*$'' means an arbitrary 
element of $\F_3$, independent of the other entries.} 
\label{tbl:Q<|C(Z):M12}
\end{Table}
\end{small}

\medskip

\noindent\boldd{$t\in\5\eta \AA$:} Since $[\5\eta,\AA]=B_2=[\eta_0,\AA]$, 
$[t,u]=1$ for some element $u\in\eta_0\AA$, and hence $R\ge R_0\gen{t,u}$.

\begin{plainlist}{8}

\item \boldd{If $\ell=0$,} then $R_0=\Delta$. So $\alpha^{-1}(t)\in 
B_1=[\Delta,\eta_0]=[R_0,u]\le[R,R]$, and hence $t\in[R,R]$. This implies 
that $R=\Delta\gen{t,u,v}$ for some $v\in\eta_\infty\AA$, and hence that 
$Z(R)=Z$ and $Z_2(R)=B_1\gen{t}\cong E_{27}$.

By the above relations, we have $Z(R/B_1)=\Delta\gen{t}/B_1\cong E_{27}$, 
while $Z(R/(Z\gen{t}))=B_1\gen{t}/Z\gen{t}\cong C_3$. So no 
$\alpha\in\Aut(R)$ sends $B_1$ into $Z\gen{t}$. $\imposs$

\item \boldd{If $\ell=1$,} then $R_0=A_*=\Delta A_{23}\cong E_{3^5}$. Set 
$E=\alpha(A_*)$. Then $t\in\alpha(B_2)\le E$, so $E\cong E_{3^5}$ is not 
contained in $\AA$, and $E$ is $\AA$-conjugate to 
$Q_\infty=W_\infty\gen{\5\eta,\eta_\infty}$ by Lemma \ref{conj-Q0}(a). 
Since $\5Q=A_*Q_\infty\nsg \SS$, this implies that $R\ge\5Q$. Thus 
$R=\5Q\gen{u}=A_*\gen{\5\eta,\eta_\infty,u}$, and has index $3$ in $\SS$. 

Let $a\in A_{33}$ be such that $u\in\eta_0aA_*$. The element $\eta_0$ 
normalizes both $A_*$ and $Q_\infty=W_\infty\gen{\5\eta,\eta_\infty}$. 
Hence $\eta_0$ normalizes each of the four subgroups of $\5Q$ isomorphic to 
$E_{3^5}$, while $A_{33}$ normalizes $A_*$ and permutes the other three 
transitively. Since $A_*\nsg R$, we must have $E=\alpha(A_*)\nsg R$, and 
this is possible only if $a=1$. Thus $R=\5Q\gen{\eta_0}=A_*\TT$.

In particular, $\QQ=B_2\gen{\5\eta,\eta_0}\le R$, and $\alpha(\QQ)=\QQ$ 
since $\QQ$ is weakly closed in $\FF$ by Lemma \ref{Q0-w.cl.}(a). This 
contradicts the assumption that $\alpha(B_2)=B_1\gen{t}\nleq \QQ$. 

\item \boldd{If $\ell=2$,} then $R_0=\AA$, contradicting \refnr4. 
$\imposs$

\end{plainlist}


\medskip

\noindent\boldd{$t\in\eta_0 \AA$:} Since $[\5\eta,\AA]=B_2=[\eta_0,A_*]$, 
$t$ commutes with some element $u\in\5\eta A_*$. Thus $R=R_0\gen{t,u}$ by 
\refnr6, where $u\in\5\eta A_*$, and $[t,u]=u^3=1$. 

\begin{plainlist}{8}

\item \boldd{If $\ell=0$,} then $R_0=B_2A_{33}$ (recall $W_0=B_2$). So 
$Z(R)=Z$, and $R/Z\cong3^{1+2}_+\times E_9$. Then 
$Z_2(R)=B_2\gen{t}\cong3^{1+2}_+\times C_3$ and 
$Z(Z_2(R))=Z(B_2\gen{t})=ZA_{13}$, and so both of these are characteristic 
in $R$. 

Since $[A_{23},R]\le B_2\le Z_2(R)$ and 
$[A_{23},Z_2(R)]=[A_{23},t]=A_{13}\le Z(Z_2(R))$ (and since 
$[A_{23},Z(Z_2(R))]=1$), we have $R\notin\EE[*]{C_\calf(Z)}$ by Lemma 
\ref{QcharP}, contradicting \refnr3.

\item \boldd{If $\ell=1$,} then $R_0=B_2A_{22}A_{33}\cong E_{3^5}$. So 
$\alpha(t)\in B_2\le[R_0,\gen{t,u}]\le[R,R]$, while $t\notin[R,R]$, a 
contradiction.

\item \boldd{If $\ell=2$,} then $R_0=\AA$, contradicting \refnr4. 
$\imposs$

\end{plainlist}


\medskip

\noindent\boldd{$t\in\eta_k \AA$ for $k=\infty,\pm1$:} We have $W_k\le 
R_0\le\Delta$ in all cases. Since $|t|=3$ by \refnr5, we have 
$t\in\eta_kA_*$, and $t\in\eta_k\Delta$ if $k=\pm1$. This follows from 
Lemma \ref{l:[x,[x,a]]}, together with the formulas in Table 
\ref{tbl:TonA}. So if $k=\pm1$, then $[\5\eta,t]\in[\5\eta,\Delta]=Z$, and 
we set $u=\5\eta\in R$. If $k=\infty$, then 
$[\5\eta,t]\in[\5\eta,A_*]=B_1$, and $[u,R_0\gen{t}]\le Z$ (hence $u\in R$) 
for some $u\in\5\eta A_{13}$. In all cases, $[t,u]\in Z$, and 
$R=R_0\gen{t,u}$ by \refnr6. 

\begin{plainlist}{8}  

\item \boldd{If $\ell=0$,} then $R_0=W_k$, and so $R\in\calq_k$, and 
$R/Z\cong E_{3^4}$ in all cases. Since $R/Z\in\EE[*]{C_\calf(Z)/Z}$ by 
\refnr3, the group 
$\Aut_{C_\calf(Z)/Z}(R/Z)\le\GL_4(3)$ has a strongly embedded subgroup, and 
hence $O^{3'}(\Aut_{C_\calf(Z)/Z}(R/Z))\cong\SL_2(9)$ or $\PSL_2(9)$ by 
Proposition \ref{p:str.emb.4}. So $\Aut_{\SS/Z}(R/Z)\cong N_{\SS}(R)/R\cong 
E_9$: a Sylow $3$-subgroup of $\textit{(P)SL}_2(9)$. 

In all cases, $N_{\SS}(R)\cap\AA=A_*$. If $k=\pm1$, then 
$N_{\SS}(R)=A_*\gen{t,\5\eta,\eta_0}$, so $|N_{\SS}(R)/R|=3^3$. If 
$k=\infty$, then $t\in Z(R)$ since $\alpha^{-1}(t)\in B_1\le Z(R)$, so 
$R\cong E_{3^5}$ and is $\SS$-conjugate to $Q_\infty$ by Lemma 
\ref{conj-Q0}(a). So $|N_{\SS}(R)/R|=|N_{\SS}(Q_\infty)/Q_\infty|=3^3$, and 
we also get a contradiction in this case.

\item \boldd{If $\ell=1$ or $2$,} then $R_0=\Delta$ and $R=\Delta\gen{t,u}$ 
where $u\in\5\eta A_{13}$ and $[t,u]\in Z$. Set $x=\alpha^{-1}(t)\in 
B_{\ell+1}\sminus B_\ell$. Then $C_R(x)\cong C_R(t)$, where either 
$C_R(t)=C_\Delta(t)\gen{t}\cong E_{27}$, or $C_\Delta(t)\gen{t,u}$ is 
nonabelian of order $3^4$. If $\ell=1$, then $x\in\AA$, so $C_R(x)\ge R_0\cong 
E_{3^4}$. If $\ell=2$, then $x\in\5\eta B_2\subseteq\5\eta\Delta$ (and $x\in 
R$), so $C_R(x)\ge W_\infty\gen{x}\cong E_{3^4}$. So this is impossible in 
either case. \qedhere

\end{plainlist} 
\end{proof}

We can now determine $\Out_{\FF}(\QQ)$. Let $\SP_4(3)\le\GL_4(3)$ 
denote the group of matrices that preserve a symplectic form up to 
sign. Thus $\SP_4(3)$ contains $\Sp_4(3)$ with index $2$. 

\begin{Lem} \label{l:Q0-props}
Assume Hypotheses \ref{h:hyp12} and Notation \ref{n:M12-UWQ}. Then 
	\[ \Out_{\FF}(\QQ) = \Out(\QQ) \cong \SP_4(3). \]
Also, 
	\begin{align*} 
	\Out_{N_{\FF}(\AA)}(\QQ) &\cong N_{\MM}(\QQ)/\QQ 
	= \AA N_{\MM}(U_0)/W_0U_0 \\
	&\cong (\AA/W_0)\rtimes(N_{\MM}(U_0)/U_0) 
	\cong E_{27}\rtimes(\GL_2(3)\times C_2).  
	\end{align*}
where the action of $C_{\MM}(U_0)/U_0\cong\GL_2(3)$ on 
$O_3(\Out_{N_{\FF}(\AA)}(\QQ))\cong\AA/W_0$ is irreducible. 
\end{Lem}

\begin{proof} The model $\MM$ for $N_\calf(\AA)$ is a semidirect product 
of $\AA$ by $\GG=\Aut_{\FF}(\AA)\cong2M_{12}$ (Lemmas \ref{l:A=Td12} and 
\ref{l:w.cl.12}(b)). Since $\QQ$ is weakly closed in $\FF$ by Lemma 
\ref{Q0-w.cl.}(a), we have $N_{\MM}(\QQ)=N_{\MM}(\AA U_0)=\AA 
N_{\GG}(U_0)$, where $N_{\GG}(U_0)\cong(E_9\rtimes\GL_2(3))\times C_2$ by 
Lemma \ref{Q0-w.cl.}(c). The description of $\Out_{N_{\FF}(\AA)}(\QQ) 
\cong N_{\MM}(\QQ)/\QQ$ is now immediate, where the action of 
$C_{\MM}(U_0)/U_0$ on $\AA/W_0$ is irreducible by Lemma \ref{Q0-w.cl.}(d).

Since $N_{\FF}(\AA)<\FF$ by assumption and 
$\FF=\gen{C_{\FF}(Z),N_{\FF}(\AA)}$ by Proposition \ref{p:F=<N,C>12}, 
we have $N_{\FF}(Z)>N_{N_{\FF}(\AA)}(Z)$. Since $\QQ$ is $\FF$-centric by 
Lemma \ref{conj-Q0}(c) and normal in $N_{\FF}(Z)$ by Lemma 
\ref{Q<|CF(Z)-M12}, $N_{\FF}(Z)$ constrained and 
$\Aut_{\FF}(\QQ)>\Aut_{N_{\FF}(\AA)}(\QQ)$. Since 
$\Out_{N_{\FF}(\AA)}(\QQ)$ is maximal in $\Out(\QQ)$, we conclude that 
$\Out_{\FF}(\QQ)=\Out(\QQ)\cong\SP_4(3)$. 
\end{proof}

We are now ready to identify all fusion systems satisfying Hypotheses 
\ref{h:hyp12}.

\begin{Thm} \label{t:M12case}
Let $\FF$ be a saturated fusion system over a finite $3$-group $\SS$ with a 
subgroup $\AA\le \SS$ such that $\AA\cong E_{3^6}$, $C_{\SS}(\AA)=\AA$, and 
$O^{3'}(\Aut_{\FF}(\AA))\cong 2M_{12}$. Assume also that $\AA\nnsg\FF$. Then 
$\AA\nsg \SS$, $\SS$ splits over $\AA$, and $\FF$ is simple and isomorphic to the 
$3$-fusion system of $\Co_1$. 
\end{Thm}

\begin{proof} By Lemma \ref{l:A=Td12}, $\Aut_{\FF}(\AA)\cong2M_{12}$ and 
acts on $\AA$ as the Todd module. By Lemma \ref{l:w.cl.12}, $\AA$ is normal 
in $\SS$ and weakly closed in $\FF$, and $\SS\cong\AA\rtimes \TT$ where 
$\TT\in\syl3{\GG}$ is defined in Notation \ref{n:M12}. So we are in the 
situation of Notation \ref{n:M12} and \ref{n:M12-UWQ}, and can use the 
terminology listed there. Set $\QQ=Q_0$; then $\QQ\nsg C_{\FF}(Z)$ by Lemma 
\ref{Q<|CF(Z)-M12}, and is the only subgroup of $\SS$ isomorphic to 
$3^{1+4}_+$ and weakly closed in $N_{\FF}(Z)$ by Lemma \ref{conj-Q0}(b). 

Set $G^*=\Co_1$, fix $S^*\in\syl3G$, and let 
$A^*\nsg S^*$ be the unique subgroup isomorphic to $E_{3^6}$. Set 
$Z^*=C_{A^*}(S^*)=Z(S^*)$. By \cite[Theorem 3.1]{Curtis-Co1} (see also the 
discussion about the subgroup $!333$ on p. 424), the fusion system 
$\calf_{S^*}(G^*)$ satisfies Hypotheses \ref{h:hyp12}. 

Let $\MM$ be a model for $N_{\FF}(\AA)$ (see Proposition 
\ref{p:NF(Q)model}), and set $\MM^*=N_{G^*}(A^*)$. By Lemmas 
\ref{l:A=Td12} and \ref{l:w.cl.12}(b), $\MM$ and $\MM^*$ are both 
semidirect products of $E_{3^6}$ by $2M_{12}$ acting as the Todd 
module, so there is an isomorphism $\varphi\:\MM^*\xto{~\cong~}\MM$ 
such that $\varphi(S^*)=\SS$. Set 
$\calf^*=\9\varphi(\calf_{S^*}(G^*))$. Thus $\calf^*$ is a fusion 
system over $\SS$ isomorphic to $\calf_{S^*}(G^*)$. We will show that 
$\calf^*=\calf$. By construction, $N_{\FF}(\AA)=N_{\calf^*}(\AA)$. 

Set $\calf_1=C_{\FF}(Z)$, $\calf_2=C_{\calf^*}(Z)$, and 
$\cale=C_{N_{\FF}(\AA)}(Z)$. Since $N_{\FF}(\AA)=N_{\calf^*}(\AA)$, $\cale$ is 
contained in $\calf_2$ as well as in $\calf_1$. All three of these are 
fusion systems over $\SS$, and $\QQ$ is centric and normal in each of them 
by Lemmas \ref{conj-Q0}(c) and \ref{Q<|CF(Z)-M12}. 
Also, $\outf[1](\QQ)=\outf[2](\QQ)\cong\Sp_4(3)$ since they have 
index $2$ in $\Out_{\FF}(\QQ)$ and $\Out_{\calf^*}(\QQ)$, respectively, where 
$\Out_{\FF}(\QQ)=\Out_{\calf^*}(\QQ)=\Out(\QQ)$ by Lemma \ref{l:Q0-props}. 

By Lemma \ref{l:Q0-props}, 
	\[ \Out_{N_{\FF}(\AA)}(\QQ) = 
	\Aut_{\AA}(\QQ)\rtimes(N_{\GG}(Z)/U_0) \cong 
	E_{27}\rtimes(\GL_2(3)\times C_2), \]
where the action of $C_{\GG}(Z)/U_0\cong\GL_2(3)$ on 
$\Aut_{\AA}(\QQ)\cong\AA/W_0$ is irreducible. In particular, 
$\Out_\cale(\QQ)$ has no normal subgroup of index $3$, and hence 
	\[ H^1(\Out_\cale(\QQ);Z(\QQ)) \cong 
	\Hom(E_{27}\rtimes\GL_2(3),\Z/3)=0.  \]
So $\calf_1=\calf_2$ by Proposition \ref{p:F1=F2-1}. 

Thus $C_{\FF}(Z)=C_{\calf^*}(Z)$ and $N_{\FF}(\AA)=N_{\calf^*}(\AA)$. Since 
$\FF=\gen{C_{\FF}(Z),N_{\FF}(\AA)}$ by Proposition \ref{p:F=<N,C>12} 
again, and similarly for $\calf^*$, we have $\FF=\calf^*$. 

The $3$-fusion system of $\Co_1$ was shown to be simple by Aschbacher 
\cite[16.10]{A-gfit} (see also \cite[Theorem A]{pprime}).
\end{proof}


\section{Todd modules for \texorpdfstring{$M_{10}$ and $M_{11}$}
{M10 and M11}}
\label{s:M10-11}

We now look at Todd modules for the Mathieu groups $M_{11}$ and $M_{10}$. 
More generally, rather than looking only at $M_{10}$-representations, we 
work with representations of extensions of $O^{3'}(M_{10})\cong A_6$. We 
want to determine all saturated fusion systems over finite $3$-groups which 
involve these modules. Throughout the section, we refer to the following 
hypotheses. 

\begin{Hyp} \label{h:hyp11}
Set $p=3$. Let $\FF$ be a saturated fusion system over a finite $3$-group 
$\SS$, and let $\AA\le\SS$ be an elementary abelian subgroup such that 
$C_{\SS}(\AA)=\AA$. Set $\GG=\Aut_{\FF}(\AA)$ and $\GGnul=O^{3'}(\GG)$, and 
assume that one of the following holds: 
\begin{enumi} 
\item $\rk(\AA)=4$ and $\GGnul\cong A_6$, or 
\item $\rk(\AA)=5$ and $\GGnul\cong M_{11}$.
\end{enumi}
\end{Hyp}

We will see in Lemma \ref{l:w.cl.} that $\AA$ is weakly closed in $\FF$ 
under these assumptions. 

The irreducible $\F_3A_6$- and $\F_3M_{11}$-modules are, of course, 
very well known. In particular, there are only three modules that we 
need to consider.

\begin{Lem} \label{l:4-5dim}
There are exactly one isomorphism class of faithful $4$-dimensional 
$\F_3A_6$-mod\-ules, and exactly two isomorphism classes of faithful 
$5$-dimensional $\F_3M_{11}$-modules. All of these modules are absolutely 
irreducible.
\end{Lem}

\begin{proof} We refer for simplicity to \cite[{p. [4]}]{modatlas} for 
the table of characters of $A_6$ in characteristic $3$: there are none 
of degree $2$, two of degree $3$ which are not realized as 
$\F_3A_6$-modules (since $\GL_3(3)$ has order prime to $5$), and one of 
degree $4$ which is realized (as the natural module for $A_6$). This 
proves the claim for $\F_3A_6$-modules.

By \cite[\S\,7A]{James}, there are exactly two isomorphism classes of 
irreducible $5$-dimensional $\4\F_3M_{11}$-modules, one the dual of the 
other. In both cases, these are the smallest degrees of nontrivial Brauer 
characters. It is well known that they can be realized as 
$\F_3M_{11}$-modules; we give one explicit construction in Lemma 
\ref{l:Td10-11}(b,c). 
\end{proof}

\noindent\textbf{Note:} Of the two distinct $5$-dimensional 
$\F_3M_{11}$-modules, what we call the ``Todd module'' is the one that has 
a set of eleven $1$-dimensional subspaces permuted by $M_{11}$. That one of 
the modules has this form is clear by the construction in Notation 
\ref{n:L10-11}.

As noted in the proof of Lemma \ref{l:4-5dim}, the $4$-dimensional 
$\F_3A_6$-module is the natural module for $A_6$: a subquotient of the 
$6$-dimensional permutation module. However, for our constructions here 
(e.g., when we want to extend it to an $\F_3\Aut(A_6)$-module), it will 
be easier to work with it as a quotient module of the Todd module for $2M_{12}$ 
described in Section \ref{s:M12}. 


\newsubb{Preliminary results}{s:Todd-lemmas}

The main goal in this subsection is to show that 
$\FF=\gen{C_{\FF}(Z),N_{\FF}(\AA)}$ whenever Hypotheses \ref{h:hyp11} 
hold (Proposition \ref{p:F=<N,C>11}). But we first describe more 
explicitly how the notation of Section \ref{s:Td10-11} is used in the 
situation of Hypotheses \ref{h:hyp11}. Recall 
that $\TT\in\syl3{\twoMat{}\ell}$ by Lemma \ref{l:N10-11}.

\begin{Not} \label{n:A6+M11}
Assume Hypotheses \ref{h:hyp11} and Notation \ref{n:L10-11} as well as 
the notation in Lemma \ref{l:Td10-11}. Identify $\GGnul$ with 
$\Matnul{}\ell=O^{3'}(\twoMat{}\ell)$ for $\ell=10$ or $11$ in such a 
way that $\TT=\Aut_{\SS}(\AA)$, and identify $\AA$ with $\aaa{\ell}$ or 
(in the $M_{11}^*$-case) with $\AAA1^*$. Thus $Z=Z(\SS)=C_{\AA}(\TT)$. 
Finally, set $A_*=[\SS,\AA]=[\TT,\AA]$. 
\end{Not}

For later reference, we collect in Table \ref{tbl:G0,A} some easy 
computations involving some of the subgroups of $\AA$ and $\GG$ defined 
above.

\begin{Table}[ht]
\[ \renewcommand{\arraystretch}{1.2}
\begin{array}{c|ccc} 
 & A_6\textup{-case} & M_{11}\textup{-case} & M_{11}^*\textup{-case} 
\\\hline
\AA & \F_3\times\F_9\times\F_3 & \F_3\times\F_9\times\F_9 & 
\F_9\times\F_9\times\F_3 \\
\bigl[s,\trp[a,b,c]\bigr] & \sm{$\trp[0,-ax,\Tr(\4bx){-}aN(x)]$} & 
\sm{$\trp[0,-ax,bx+ax^2]$} & \sm{$\trp[0,-ax,\Tr(bx+ax^2)]$} \\
A_*=[\TT,\AA] & 0\times\F_9\times\F_3 
& 0\times\F_9\times\F_9 & 0\times\F_9\times\F_3 \\ 
{[s,\AA]} & \{\trp[0,ax,c]\,|\,a,c\in\F_3\} & 
\dbl{\{\trp[0,ax,c]\,|\,\hfill}{~a\in\F_3,~c\in\F_9\}} & 
0\times\F_9\times\F_3 \\
C_{\AA}(\TT)=Z(\SS) & 0\times0\times\F_3 & 0\times0\times\F_9 
& 0\times0\times\F_3 \\
\sm{$C_{\AA}(s)=Z(\AA\gen{s})$} & \{\trp[0,b,c]\,|\,\Tr(b\4x)=0\} 
& 0\times0\times\F_9 & \{\trp[0,b,c]\,|\,\Tr(bx)=0\} \\
\textup{Jd. bl. lth. of $c_s$} & 3+1 & 3+2 & 3+2 
\end{array} \]
\caption{In all cases, $s\in\SS\sminus\AA$, and $x\in\F_9$ is such that 
$c_s=\dpar{x}\in \TT$. The last line gives the Jordan block lengths for the 
action of $s$ on $\AA$.}
\label{tbl:G0,A} 
\end{Table}

The next lemma gives a first easy consequence of the computations in Table 
\ref{tbl:G0,A}.

\begin{Lem} \label{l:w.cl.}
Assume that $\AA\le\SS$ and $\FF$ satisfy Hypotheses \ref{h:hyp11}. Then 
$\AA$ is weakly closed in $\FF$ and in particular is normal in $\SS$. 
\end{Lem}

\begin{proof} By Lemma \ref{l:4-5dim}, $\AA$ is one of the 
$\F_3\GGnul$-modules described in Lemma \ref{l:Td10-11}. From that 
lemma and Table \ref{tbl:G0,A}, we see that in all of these cases, 
$N_{\SS}(\AA)/\AA\cong E_9$, $|C_{\AA}(x)|=9$ for each $x\in 
N_{\SS}(\AA)\sminus \AA$, and $|\AA:C_{\AA}(N_{\SS}(\AA))|\ge3^3$. So $\AA$ 
is the unique abelian subgroup of index $9$ in $N_{\SS}(\AA)$, and hence by 
Lemma \ref{l:A<|S} is weakly closed in $\FF$. 
\end{proof}

The following properties will also be needed.

\begin{Lem} \label{l:[S,S]}
Assume Hypotheses \ref{h:hyp11} and Notation \ref{n:A6+M11}. 
\begin{enuma} 

\item In the $A_6$- and $M_{11}$-cases, for $x\in \SS\sminus \AA$ and 
$a\in \AA$, we have $(ax)^3=x^3$ if and only if $a\in A_*$. In all 
cases, $x\in\SS\sminus\AA$ and $a\in A_*$ implies $(ax)^3=x^3$.

\item In all cases, if $\AA\nnsg\FF$, then $[\SS,\SS]=A_*$. 

\end{enuma}
\end{Lem}

\begin{proof} \textbf{(a) } By Lemma \ref{l:[x,[x,a]]}, for $a\in\AA$ 
and $x\in\SS\sminus\AA$, $x^3=(ax)^3$ if and only if $[x,[x,a]]=1$; 
i.e., if $[x,a]\in C_{\AA}(x)$. By Table \ref{tbl:G0,A}, this holds if and 
only if $a\in A_*$ in the $A_6$- and $M_{11}$-cases, while $[x,A_*]=Z\le 
C_{\AA}(x)$ in the $M_{11}^*$-case.

\smallskip

\noindent\textbf{(b) } Assume otherwise: assume $[\SS,\SS]>A_*=[\SS,\AA]$. 
Then since $\SS/\AA\cong E_9$ in all cases, $[\SS,\SS]$ contains $A_*$ with 
index $3$.

Assume we are in the $M_{11}^*$-case. Thus $|\AA/A_*|=9$ (Table 
\ref{tbl:G0,A}), and hence $A_*<[\SS,\SS]<\AA$. By Lemma 
\ref{l:N10-11}, there is an element $-\sbk{i}\in N_{\GGnul}(\TT)$, and 
this extends to $\alpha\in\Aut_{\FF}(\SS)$ by the extension axiom. By the 
formulas in Lemma \ref{l:Td10-11}(c), no subgroup of index $3$ in $\AA$ 
and containing $A_*$ is normalized by $\alpha$. In particular, 
$\alpha([\SS,\SS])\ne[\SS,\SS]$, which is impossible. 

Now assume we are in the $A_6$- or $M_{11}$-case. Then $|\AA/A_*|=3$ by 
Table \ref{tbl:G0,A} again, so $[\SS,\SS]=\AA$, and $\SS/A_*$ is nonabelian 
of order $27$. Let $x\in \SS\sminus \AA$ and $y\in \SS\sminus \AA\gen{x}$ 
be arbitrary. Then $\SS=\AA\gen{x,y}$ and $[x,y]\in \AA\sminus A_*$. So 
$x^3\ne(\9yx)^3=\9y(x^3)$ by (a). In particular, $x^3\ne1$, and since $x$ 
was arbitrary, no element of $\SS\sminus \AA$ has order $3$. 

Assume $R\in\EE{\FF}$. Then $\AA\cap R=\Omega_1(R)$ is characteristic in 
$R$. For each $a\in N_{\AA}(R)\sminus R$, 
we have $[a,R]\le R\cap\AA$ and $[a,R\cap\AA]=1$, contradicting Lemma 
\ref{QcharP}. Thus $N_{\AA}(R)\le R$, so $N_{\AA R}(R)=R$, and hence 
$\AA\le R$. Thus each $\FF$-essential subgroup contains $\AA$, 
contradicting the assumption that $\AA\nnsg\FF$. 
\end{proof}

In Notation \ref{n:A6+M11}, we identified 
$O^{3'}(\GG)=O^{3'}(\twoMat{}\ell)$ (for $\ell=10$ or $11$). In fact, 
this extends to an inclusion $\GG\le\twoMat{}\ell$.

\begin{Lem} \label{l:GG<Mat}
Assume Hypotheses \ref{h:hyp11} and Notation \ref{n:A6+M11}. Then for 
$\ell=10,11$, $\nnn{\ell}=N_{\twoMat{}\ell}(\TT)$ and is a maximal 
subgroup of $\twoMat{}\ell$. Also, as subgroups of $\Aut(\AA)$, we have 
\begin{itemize} 
\item $\twoMat10 = N_{\Aut(\AA)}(\GGnul) \ge \GG$ if 
$\GGnul=\Matnul10\cong A_6$, and 
\item $\twoMat11 = N_{\Aut(\AA)}(\GGnul) \ge \GG$ if 
$\GGnul=\Matnul11\cong M_{11}$.
\end{itemize}
\end{Lem}

\begin{proof} For $\ell=10,11$, 
	\[ \nnn{\ell} = \NN\cap\twoMat{}\ell = 
	N_{\twoMat12}(\TT)\cap\twoMat{}\ell = N_{\twoMat{}\ell}(\TT), \]
where the second equality holds by Lemma \ref{l:H-max}. The maximality 
of $\nnn{\ell}$ in $\twoMat{}\ell$ is well known in both cases, but we 
note the following very simple argument. If $\nnn{\ell}$ is not maximal 
in $\twoMat{}\ell$, then since it has index 
$10$ or $55$ when $\ell=10$ or $11$, respectively, 
there is $\nnn{\ell}<H<\twoMat{}\ell$ where 
$[H:\nnn{\ell}]=n$ for $n\in\{2,5,11\}$. But then $H$ has exactly $n$ 
Sylow 3-subgroups where $n\equiv2$ (mod $3$), contradicting the Sylow
theorems.

Now let $\ell\in\{10,11\}$ be such that $\GGnul=\Matnul{}\ell$. Since $\AA$ 
is absolutely irreducible as an $\F_3\Matnul{}\ell$-module by Lemma 
\ref{l:4-5dim}, we have $C_{\Aut(\AA)}(\Matnul{}\ell)=\{\pm\Id\}$, and 
hence 
	\[ |\twoMat{}\ell/\Matnul{}\ell| \le 
	|N_{\Aut(\AA)}(\Matnul{}\ell)/\Matnul{}\ell| \le 
	2\cdot|\Out(\Matnul{}\ell)|. \]
These inequalities are equalities by Table \ref{tbl:Mat} (and since 
$|\Out(A_6)|=4$ and $|\Out(M_{11})|=1$), so 
$\twoMat{}\ell=N_{\Aut(\AA)}(\Matnul{}\ell)\ge\GG$. 
\end{proof}

We can now begin to apply some of the lemmas in Section \ref{s:general}.
 
\begin{Prop} \label{p:F=<N,C>11}
Assume Hypotheses \ref{h:hyp11} and Notation \ref{n:A6+M11}. Then 
$\FF=\gen{C_{\FF}(Z),N_{\FF}(\AA)}$.
\end{Prop}

\begin{proof} Assume otherwise, and recall that $\AA\nsg\SS$ by Lemma 
\ref{l:w.cl.}. By Proposition \ref{p:R=C(Z)}, there are subgroups $X\in 
Z^{\FF}$ and $R\in\EE{\FF}$ such that $X\nleq \AA$, 
$R=C_{\SS}(X)=N_{\SS}(X)$, and $Z=\alpha(X)$ for some 
$\alpha\in\Aut_{\FF}(R)$. Set $R_0=R\cap \AA$. 

Fix $x\in X\sminus \AA$. Then $|x|=3$, since $x\in X\cong Z$ and 
$Z\le\AA$ has exponent $3$. Also, $R_0=C_{\AA}(X)=C_{\AA}(x)$: 
since either $|X|=|Z|=3$ and hence $X=\gen{x}$, or else we are in the 
$M_{11}$-case and $C_{\AA}(x)=Z=C_{\AA}(\SS)$. Since $x$ acts on 
$\AA$ in all cases with two Jordan blocks (Table \ref{tbl:G0,A}), we 
have $|R_0|=|C_{\AA}(x)|=9$.

\smallskip

\noindent\textbf{Case 1: } Assume first that $|R\AA/\AA|=3$. Then 
$R=R_0\gen{x}$, and hence $|R|=27$. 

If we are in the $A_6$-case, then each member of the $\SS$-conjugacy class 
of $R$ has the form $C_{\AA}(y)\gen{y}=R_0\gen{y}$ for some $y\in 
x\AA$, and $y\in xA_*$ by Lemma \ref{l:[S,S]}(a) and since $y^3=1=x^3$. 
Since $C_{\AA}(x)$ has index $3$ in $A_*$, there are at most three such 
subgroups, so $|N_{\SS}(R)/R|\ge\frac13[S:R]=9$, contradicting Lemma 
\ref{l:filtered}(b).

In the $M_{11}$- and $M_{11}^*$-cases, 
$|N_{R\AA}(R)/R|=|C_{\AA/R_0}(x)|=9$, since $x$ acts on $\AA$ with 
Jordan blocks of length $3$ and $2$ (Table \ref{tbl:G0,A}). Thus 
$|\Out_{\AA}(R)|=9$. Since $\Out_{\AA}(R)$ acts trivially on $R_0$, and 
$|R/R_0|=3$, this contradicts Lemma \ref{l:filtered-c}.

\smallskip

\noindent\textbf{Case 2: } Now assume that $|R\AA/\AA|\cong E_9$. Thus 
$R\AA=\SS$ and $|R|=81$. 

Assume first we are in the $A_6$- or $M_{11}^*$-case. Then $|Z|=3$ and 
$Z=C_{\AA}(R)<C_{\AA}(x)$. So there are $y\in R\sminus \AA\gen{x}$ and 
$a\in C_{\AA}(x)\sminus C_{\AA}(R)$ such that $1\ne[y,a]\in Z$, and 
hence $Z\le[R,C_{\AA}(x)]\le[R,R]$. Since $X\nleq[R,R]$, no 
automorphism of $R$ sends $X$ to $Z$. 

Now assume we are in the $M_{11}$-case. Then $R_0=Z$ and 
$N_{\SS}(R)=RA_*$, so $|N_{\SS}(R)/R|=|A_*/Z|=9$, and hence $R\cong 
E_{81}$ by Lemma \ref{l:filtered}(b). Each element of order $3$ in 
$\Aut_{\SS}(R)$ acts on $R$ with Jordan blocks of length at most $2$, 
so by Proposition \ref{p:str.emb.4}, 
$O^{3'}(\Aut_{\FF}(R))\cong\SL_2(9)$ with the natural action on $R$. 
Also, each element of order $8$ in 
$N_{O^{3'}(\Aut_{\FF}(R))}(\Aut_{\SS}(R))$ restricts to an element 
$\alpha\in\Aut_{\FF}(Z)$ of order $8$ (note that $Z=[N_{\SS}(R),R]$), 
and this in turn extends to some $\beta\in\Aut_{\FF}(\SS)$ and hence to 
$\beta|_{\AA}\in\Aut_{\FF}(\AA)$ since $\AA$ is weakly closed in $\FF$ 
by Lemma \ref{l:w.cl.}. But 
$\Matnul11\le\Aut_{\FF}(\AA)\le\twoMat11\cong M_{11}\times C_2$ by 
Lemma \ref{l:GG<Mat}, so $\F_9^\times\gen{\phi}$ or its product with 
$\{\pm\Id\}$ is a Sylow 2-subgroup of $\Aut_{\FF}(\AA)$, and by Lemma 
\ref{l:Td10-11}(b), the subgroups of order $8$ in these groups do not 
act faithfully on $Z$. So this case is impossible. 
\end{proof}


\newsubb{The subgroup \texorpdfstring{$\QQ\nsg C_{\FF}(Z)$}{Q<CF(Z)}}{s:Q}

So far, we have shown that $\FF=\gen{N_{\FF}(\AA),C_{\FF}(Z)}$ in all 
cases where Hypotheses \ref{h:hyp11} hold. Our next step in studying 
these fusion systems is to prove that $C_{\FF}(Z)$ is constrained by 
constructing a normal centric subgroup $\QQ\nsg C_{\FF}(Z)$; and 
proving (as one consequence) that $\SS$ splits over $\AA$. 


\begin{Prop} \label{Q<|C(Z)}
Assume Hypotheses \ref{h:hyp11} where $\AA\nnsg\FF$. Then there is a unique 
special subgroup $\QQ\nsg \SS$ of exponent $3$ such that $Z(\QQ)=Z$, 
$\QQ\cap\AA=A_*$, and $\QQ/Z\cong E_{81}$, and $\EE[*]{C_{\FF}(Z)}=\{\QQ\}$. In 
particular, $\QQ\nsg C_{\FF}(Z)$, and $\QQ$ is weakly closed in $\FF$ and 
$\FF$-centric. 
\end{Prop}

\begin{proof} Assume Notation \ref{n:A6+M11}. Define 
	\begin{align*} 
	\calq &= \{ Q\le \SS \,|\, Q\cap\AA=A_*,~ Q/Z~\textup{abelian of 
	order $3^4$} \} \\
	\calq_0 &= \{ Q\in\calq \,|\, \textup{$Q$ of exponent $3$} \}.
	\end{align*}
Recall that $[\SS,\SS]=A_*$ by Lemma \ref{l:[S,S]}(b). Also, $\SS/A_*$ is 
elementary abelian by Lemma \ref{l:spec}(a), applied to the group $\SS/Z$ 
with center $A_*/Z$. 

We will prove that 
	\beqq \EE[*]{C_{\FF}(Z)}\subseteq\calq_0 \quad\textup{and}\quad 
	|\calq_0|\le1. \label{e:|Q0|le1} \eeqq
Since $\FF=\gen{C_{\FF}(Z),N_{\FF}(\AA)}$ by Proposition 
\ref{p:F=<N,C>11}, and since $\FF\ne N_{\FF}(\AA)$ (recall 
$\AA\nnsg\FF$ by assumption), we have $\EE[*]{C_{\FF}(Z)}\ne\emptyset$. 
So \eqref{e:|Q0|le1} implies that $\EE[*]{C_{\FF}(Z)}=\calq_0$ has order 
$1$, and for $Q\in\calq_0$, $Q\nsg C_{\FF}(Z)$ and $Q$ is weakly closed 
in $\calf$. By construction, $C_{\SS}(Q)=C_{\SS}(\TT)=Z$, so $Q$ is 
also $\calf$-centric.

It thus remains to prove \eqref{e:|Q0|le1}. Set $\4{\SS}=\SS/Z$ and 
similarly for subgroups and elements of $\SS$. In all cases, 
$Z(\4{\SS})=\4A_*\cong E_9$. 

Let $\rho\:Q/A_*\too Z$ be the homomorphism of Lemma \ref{l:spec}(b) that 
sends $gA_*$ to $g^3$. (Note that $\rho$ is defined on $\4Q=Q/Z$ in the 
lemma, but factors through $Q/A_*$ since $A_*$ is elementary abelian.)


\noindent\boldd{$A_6$- and $M_{11}$-cases: } Here, $|\AA/A_*|=3$, so 
$|\calq_0|\le|\calq|=1$ by Lemma \ref{l:spec}(c), applied with $\4{\SS}$ 
and $\4{A_*}$ in the role of $S$ and $Z$. Let $\QQ\in\calq$ be the unique 
element. Then $\EE[*]{C_{\FF}(Z)/Z}\subseteq\{\4\QQ\}$ by \cite[Lemma 
2.3(a)]{indp1} and since $\4\QQ$ is the unique abelian subgroup of index $3$ 
in $\4{\SS}$, and so $\EE[*]{C_{\FF}(Z)}\subseteq\{\QQ\}$ by Lemma 
\ref{l:(F/Z)e}.  

Since $\QQ$ is the only member of $\calq$, it is normalized by 
$\Aut_{\FF}(\SS)$. By Table \ref{tbl:Mat}, the element 
	\[ \beta_0 = \begin{cases} 
	-\sbk{i} \in \NNN0\cap\Matnul10\le\Aut_{\FF}(\AA) & \textup{in 
	the $A_6$-case} \\
	-\sbk{\zeta} \in \NNN1\cap\Matnul11\le\Aut_{\FF}(\AA) & \textup{in 
	the $M_{11}$-case} 
	\end{cases} \]
normalizes $\Aut_{\SS}(\AA)$, and hence extends to some 
$\beta\in\Aut_{\FF}(\SS)$. Also, by construction of $\NNN0=\NNN1$, $\beta$ 
permutes the cosets $gA_*$ for $g\in \QQ\sminus A_*$ --- in two orbits of length 
$4$ in the $A_6$-case, or one orbit of length $8$ in the $M_{11}$-case --- 
and $\rho$ is constant on each of these orbits. 

In the $A_6$-case, where $|Z|=3$, this implies that $\rho=1$ and hence 
$\QQ\in\calq_0$. In the $M_{11}$-case, where $|Z|=9$, it implies that 
either $\QQ\in\calq_0$, or all elements of $\QQ\sminus A_*$ have order $9$ 
and hence $A_*$ is characteristic in $\QQ$. But in that case, 
$\QQ\notin\EE[*]{C_{\FF}(Z)}$ by Lemma \ref{QcharP}, since for 
$a\in\AA\sminus A_*$, we have $[a,\QQ]\le A_*$ and $[a,A_*]=1$. We 
conclude that $\EE[*]{C_{\FF}(Z)}\subseteq\calq_0$ in either case, 
finishing the proof of \eqref{e:|Q0|le1}.


\smallskip

\noindent\boldd{$M_{11}^*$-case: } Now, $|\AA/A_*|=9$. Assume 
$R\in\EE[*]{C_{\FF}(Z)}$. Then $R\ge Z$ and $\4R\in\EE[*]{C_{\FF}(Z)/Z}$ 
by Lemma \ref{l:(F/Z)e}, and hence $\4R\ge Z(\4{\SS})=\4{A_*}$. If 
$\4R$ is not abelian, then $Z(\4R)=\4{A_*}$, so $\4{A_*}$ is 
characteristic in $\4R$, contradicting Lemma \ref{QcharP} since 
$[x,\4R]\le \4{A_*}$ and $[x,\4{A_*}]=1$ for each $x\in\4{\SS}\sminus 
\4R$. Thus $\4R$ is abelian, and is maximal abelian since it is 
$\calf/Z$-centric. So $R\in\calq\cup\{\AA\}$ by Lemma \ref{l:spec}(d), 
and $\EE[*]{C_{\FF}(Z)}\subseteq\calq\cup\{\AA\}$. 

Since $\NNN1\cong(E_9\rtimes\SD_{16})\times C_2$ is a maximal subgroup 
of $\twoMat11$ by Lemma \ref{l:GG<Mat} and normalizes $Z$ by Lemma 
\ref{l:Td10-11}(c), we see that 
$\Aut_{N_{\FF}(Z)}(\AA)=C_{\Aut_{\FF}(\AA)}(Z)$ has index $2$ in 
$\NNN1$ and hence contains $\TT$ as a normal subgroup. So 
$\AA\notin\EE[*]{C_{\FF}(Z)}$, and $\EE[*]{C_{\FF}(Z)}\subseteq\calq$. 

Assume $R$ is not of exponent $3$, and set $R_0=\Omega_1(R)$. Then $R_0$ 
has index $3$ in $R$ by Lemma \ref{l:spec}(b), and so $R_0/Z(R_0)\cong E_9$ 
where $Z(R_0)\le A_*$. Since $|\AA/A_*|=9$ and 
$9\nmid|\Aut(R_0/Z(R_0))|$, there is $x\in\AA\sminus A_*$ such that 
$[x,R_0]\le Z(R_0)$. Also, $[x,R]\le A_*\le R_0$ and $[x,Z(R_0)]=1$, and by 
Lemma \ref{QcharP}, this contradicts the assumption that 
$R\in\EE[*]{C_{\FF}(Z)}$. Thus $\EE[*]{C_{\FF}(Z)}\subseteq\calq_0$.

It remains to show that $|\calq_0|\le1$. Assume otherwise: assume $Q_1$ 
and $Q_2$ are both in $\calq_0$. Define $\psi\:\SS/\AA\too \AA/A_*$ by setting, 
for each $g\AA\in \SS/\AA$, $\psi(g\AA)=(g\AA\cap Q_1)^{-1}(g\AA\cap 
Q_2)\in \AA/A_*$. (Note that $g\AA\cap Q_i\in\SS/A_*$ for $i=1,2$.) Since 
$(g_1)^3=1=(g_2)^3$ for $g_i\in g\AA\cap Q_i$, and $g_2\in g_1\psi(g\AA)$, 
we have $[g,[g,\psi(g\AA)]]=1$ by Lemma \ref{l:[x,[x,a]]}. Using the 
formulas in Lemma \ref{l:Td10-11}(c), we identify $\SS/\AA$ and $\AA/A_*$ with 
$\F_9$, and through that identify $\psi$ with an additive homomorphism 
$\5\psi\:\F_9\too\F_9$ such that 
	\[ 0 = [\dpar{x},[\dpar{x},\trp[\5\psi(x),0,0]]] = 
	\trp[0,0,\Tr(x^2\5\psi(x)] \]
for each $x\in\F_9$. Thus $x^2\5\psi(x)\in i\F_3$, and 
	\[ \5\psi(x) \in \begin{cases} i\F_3 & \textup{if $x=\pm1,\pm i$} \\
	\F_3 & \textup{if $x=\pm\zeta,\pm\zeta^3$.}
	\end{cases} \]
Hence $\5\psi$ is not onto, and either $\5\psi(1)=\5\psi(i)=0$ or 
$\5\psi(\zeta)=\5\psi(\zeta^3)=0$. This proves that $\5\psi=0$ and hence 
$Q_1=Q_2$, and finishes the proof of \eqref{e:|Q0|le1}. 
\end{proof}

We list some of the properties of these subgroups $\QQ\nsg\SS$ in the 
Table \ref{tbl:Q} for easy reference. They follow immediately from the 
descriptions in Lemma \ref{l:Td10-11} and Proposition \ref{Q<|C(Z)}. 

\begin{Table}[ht]
\[ \renewcommand{\arraystretch}{1.4}
\begin{array}{l|cccccc} 
 & \GGnul\cong & \rk(\AA) & \rk(Z) & ~|\SS|~ & \QQ\cong & |\Out_{\SS}(\QQ)| 
\\\hline
\textup{$A_6$-case} & A_6 & 4 & 1 & 3^6 & 3^{1+4}_+ & 3 \\
\textup{$M_{11}$-case} & M_{11} & 5 & 2 & 3^7 & 3^{2+4} & 3 \\
\textup{$M_{11}^*$-case} & M_{11} & 5 & 1 & 3^7 & 3^{1+4}_+ & 9 
\end{array} \]
\caption{} 
\label{tbl:Q} 
\end{Table}

One easy consequence of Proposition \ref{Q<|C(Z)} is that 
$\SS\cong\AA\rtimes \TT$.

\begin{Cor} \label{c:Ssplit}
Assume Hypotheses \ref{h:hyp11} where $\AA\nnsg\FF$, and let $\MM$ be a 
model for $N_{\FF}(\AA)$ (see Proposition \ref{p:NF(Q)model}). Then $\SS$ 
and $\MM$ split over $\AA$. 
\end{Cor}

\begin{proof} Let $\QQ\nsg\SS$ be the special subgroup of exponent $3$ of 
Proposition \ref{Q<|C(Z)}. To prove that $\SS$ splits over $\AA$, it 
suffices to show that $\QQ$ splits over $\QQ\cap\AA=A_*$. If $|Z|=9$ (i.e., in 
the $M_{11}$-case), then we are in the situation of Lemma \ref{l:spec}(d), 
so there is $B\le \QQ$ abelian of index $9$ such that $B\cap 
A_*=Z$, and any complement in $B$ to $Z$ is a splitting of $\QQ$ over $A_*$.

If $|Z|=3$, then consider the space $\4\QQ=\QQ/Z$, with symplectic form $\bb$ 
defined by $\bb(xZ,yZ)=\sigma([x,y])$ for some 
$\sigma\:Z\xto{\cong}\F_3$. Following the standard procedure for 
constructing a symplectic basis for $\4\QQ$, we fix a basis $\{a_1,a_2\}$ for 
$A_*/Z$, choose $b_1\in\4\QQ\sminus a_1^\perp$, and choose 
$b_2\in\gen{a_1,b_1}^\perp\sminus\gen{a_2}$. Then $\{a_1,b_1,a_2,b_2\}$ is 
a basis for $\4\QQ$, and $\gen{b_1,b_2}\le\4\QQ$ is totally isotropic and lifts 
to a splitting of $\QQ$ over $A_*$. 

Since $\SS$ splits over $\AA$, it follows from Gasch\"utz's theorem (see 
\cite[(10.4)]{A-FGT}) that $\MM$ also splits over $\AA$.
\end{proof}

Recall that for $\ell=10,11$, we set $\TT=O_3(\nnn\ell)\cong E_9$, a Sylow 
$3$-subgroup of $\twoMat{}\ell$, and set 
$\Matnul{}\ell=O^{3'}(\twoMat{}\ell)$. Also, $\GG$ was chosen so that 
$\GGnul=\Matnul{}\ell$ (see Notation \ref{n:A6+M11}), and then 
$\GG\le\twoMat{}\ell$ by Lemma \ref{l:GG<Mat}. 

\begin{Not} \label{n:Gamma,T}
Assume Hypotheses \ref{h:hyp11} and Notation \ref{n:L10-11} and 
\ref{n:A6+M11}. Let $\MM$ be a model for $N_{\FF}(\AA)$, and set 
$\MMnul=O^{3'}(\MM)$. Then $\MM$ splits over $\AA$ by Corollary 
\ref{c:Ssplit}, and we identify 
	\[ \MM = \AA\rtimes\GG\le\AA\rtimes\twoMat{}\ell 
	\qquad\textup{and}\qquad
	\MMnul = \AA\rtimes\GGnul=\AA\rtimes\Matnul{}\ell, \]
where $\ell=10$ if $\GGnul\cong A_6$ and $\ell=11$ if $\GGnul\cong M_{11}$. 
Thus $\SS=\AA\rtimes \TT\in\syl3\MM$ and $\QQ=A_*\rtimes \TT\le\SS$. 
\end{Not}

One easily sees that $\QQ$ is special with $Z(\QQ)=Z$ and $\QQ/Z\cong 
E_{81}$. Also, $\QQ$ has exponent $3$ by Lemma \ref{l:[S,S]}(a), and 
hence is the subgroup described in Proposition \ref{Q<|C(Z)}. In 
particular, $\QQ\nsg C_\calf(Z)$, and $\EE{C_{\FF}(Z)}=\{\QQ\}$.

Recall Notation \ref{n:L10-11} and Lemma \ref{l:N10-11}: 
$\TT=\{\dpar{x}\,|\,x\in\F_9\}$, and 
	\[ \NNN0=\NNN1= \Gen{ \dpar{x}, \sbk{u},\sbk{\phi},-\Id 
	\,\big|\, x\in\F_9,~ u\in\F_9^\times } 
	\cong (E_9\rtimes\SD_{16})\times\{\pm\Id\}. \]

\begin{Lem} \label{l:(i-iv)}
Assume Hypotheses \ref{h:hyp11} and Notation \ref{n:Gamma,T}, and 
also that $\AA\nnsg\FF$. Then conditions \textup{(i)--(iii)} 
in Hypotheses \ref{h:F1=F2-2} hold for $\FF$, $\SS$, $\AA$, and $\QQ$. 
\end{Lem}

\begin{proof} Since $\MM$ is a model for $N_{\FF}(\AA)$, we have 
$\SS\in\syl3\MM$ and $\MM/\AA\cong\GG=\Aut_{\FF}(\AA)$. Each pair of 
distinct Sylow 3-subgroups of $\GGnul=O^{3'}(\GG)\cong A_6$ or $M_{11}$ 
intersects trivially. Hence for each subgroup $R$ such that $\AA<R<\SS$, 
$\SS$ is the unique Sylow $3$-subgroup of $\MM$ that contains $R$. So 
$1\ne\Out_{\SS}(R)\nsg\Out_{\MM}(R)=\Out_{\FF}(R)$, and hence 
$\Out_{\FF}(R)=\Out_{N_{\FF}(\AA)}(R)$ does not have a strongly 
$3$-embedded subgroup. Thus no such $R$ can be $N_{\FF}(\AA)$-essential, 
proving that $\EE[*]{N_{\FF}(\AA)}\subseteq\{\AA\}$. 

By Proposition \ref{p:F=<N,C>11}, $\FF=\Gen{N_{\FF}(\AA),C_{\FF}(Z)}$. 
Hence $\EE{\FF}\subseteq\EE[*]{N_{\FF}(\AA)}\cup\EE[*]{C_{\FF}(Z)}$ by 
Proposition \ref{p:-AFT}, while $\EE[*]{C_{\FF}(Z)}\subseteq\{\QQ\}$ by 
Proposition \ref{Q<|C(Z)}. So $\EE{\FF}\subseteq\{\AA,\QQ\}$. Also, 
$\AA\in\EE{\FF}$ by Lemma \ref{l:str.emb2} and since 
$\GGnul=O^{3'}(\Aut_{\FF}(\AA))\cong A_6$ or $M_{11}$ and hence has a 
strongly embedded subgroup, and $\QQ\in\EE{\FF}$ since otherwise $\AA$ 
would be normal in $\FF$. Thus $\EE{\FF}=\{\AA,\QQ\}$, proving 
\ref{h:F1=F2-2}(i).

Recall that $\QQ=A_*\TT$. So $\SS=\AA \QQ$, and $C_{\SS}(\QQ\cap 
\AA)=C_{\SS}(A_*)=\AA$ by the relations in Lemma \ref{l:Td10-11}. This 
proves \ref{h:F1=F2-2}(ii).

By Lemma \ref{l:4-5dim}, $\AA$ is absolutely irreducible as an 
$\F_3\GGnul$-module, where $\GGnul=O^{3'}(\Aut_{\FF}(\AA))$ as earlier. 
Thus the centralizer in $\Aut(\AA)$ of $\GGnul$ is $\{\pm\Id\}$. Since 
$\Out(A_6)$ and $\Out(M_{11})$ are 2-groups, 
$N_{\Aut(\AA)}(O^{p'}(\Aut_{\FF}(\AA)))/O^{p'}(\Aut_{\FF}(\AA))$ is also a 
2-group, and so \ref{h:F1=F2-2}(iii) holds. 
\end{proof}

The following notation for elements in $\QQ$ will be useful. 

\begin{Not} \label{not:Q}
For $a,b\in\F_9$, and $z\in\F_9$ (in the $M_{11}$-case) or $z\in\F_3$ (in 
the $A_6$- or $M_{11}^*$-case), set 
	\[ \Qtrp[a,b,z] = \trp[0,a,z]\dpar{b} \in A_*\TT=\QQ. \]
\end{Not}

Thus each element of $\QQ$ is represented by a unique triple $\Qtrp[a,b,z]$, 
for $a,b\in\F_9$ and $z\in\F_3$ or $\F_9$. We sometimes write 
$\Qtrp[a,b,*]\in \QQ/Z$ to denote the class of $\Qtrp[a,b,z]$ for arbitrary 
$z$. 

We list in Table \ref{tbl:<<->>} some of the relations among such triples: 
all of these are immediate consequences of the definition in Notation 
\ref{not:Q} and the relations in Lemma \ref{l:Td10-11}.
\begin{Table}[ht]
\[ \renewcommand{\arraystretch}{1.3}
\renewcommand{\halfup}[1]{\raisebox{2.2ex}[0pt]{#1}}
\begin{array}{c|ccc} 
 & \textup{$A_6$-case} & \textup{$M_{11}$-case} & \textup{$M_{11}^*$-case} 
\\\hline
& \multicolumn{3}{c}
{\sm{$\Qtrp[a,b,z]{\cdot}\Qtrp[c,d,y]=\Qtrp[a+c,b+d,{z+y+\mu(b,c)}]$\quad 
where}} \\
& \mu(b,c)=\Tr(\4bc) 
& \mu(b,c)=bc & \mu(b,c)=-\Tr(bc) \\\hline
\9{\trp[r,0,0]}\Qtrp[a,b,z] & \sm{$\Qtrp[a-br,b,z+rN(b)]$} 
& \sm{$\Qtrp[a-br,b,z+rb^2]$} & \sm{$\Qtrp[a+br,b,z+\Tr(rb^2)]$} \\
\9{\sbk{u}}\Qtrp[a,b,z] & \Qtrp[ua,ub,N(u)z] & \Qtrp[ua,ub,u^2z] 
& \Qtrp[u^{-1}a,ub,z] \\ 
\9{\sbk{\phi}}\Qtrp[a,b,z] & \Qtrp[\4a,\4b,z] & \Qtrp[\4a,\4b,\4z] 
& \Qtrp[\4a,\4b,z] \\
\9{-\Id}\Qtrp[a,b,z] & \Qtrp[{-}a,b,{-}z] & \Qtrp[{-}a,b,{-}z] 
& \Qtrp[{-}a,b,{-}z]
\end{array} \]
\caption{Here, $a,b,c,d\in\F_9$ and $u\in\F_9^\times$ in all cases, 
$z,y\in\F_3$ in the $A_6$- and $M_{11}^*$-cases, and $z,y\in\F_9$ in the 
$M_{11}$-case. Also, $r\in\F_3$ in the $A_6$- and $M_{11}$-cases, and 
$r\in\F_9$ in the $M_{11}^*$-case.} 
\label{tbl:<<->>} 
\end{Table}

The next two lemmas give more information about $\Out(\QQ)$ and 
$\Out_{\FF}(\QQ)$. We start with the case where $\GGnul\cong A_6$.

\begin{Lem} \label{l:Out(Q)10}
Assume Hypotheses \ref{h:hyp11}, and Notation \ref{n:A6+M11} and 
\ref{n:Gamma,T}, with $\GGnul\cong A_6$. Thus 
$\MMnul=\AA\rtimes\Matnul10\cong E_{81}\rtimes A_6$. 
Then each $\alpha\in N_{\Aut(\QQ)}(\Aut_{\SS}(\QQ))$ 
extends to some $\4\alpha\in\Aut(\MMnul)$.
\end{Lem}

\begin{proof} Since $\NNN0=N_{\twoMat10}(\TT)$ by Lemma \ref{l:GG<Mat}, 
we have 
	\beqq N_{\MMnul}(\SS) = \AA\rtimes(\NNN0\cap\Matnul10) = 
	\SS\gen{\beta} 
	\quad\textup{where}\quad
	\beta=-\sbk{i}\in\NNN0, \label{e:betadef} \eeqq
by Lemma \ref{l:N10-11}, and $\beta$ acts on $\SS$ via 
	\beqq \9\beta\bigl(\trp[a,b,c]\dpar{x}\bigr)=\trp[-a,-ib,-c]\dpar{ix}. 
	\label{e:c_beta} \eeqq

For calculations in $\Out(\QQ)$, we use Notation \ref{not:Q}, and the 
ordered basis 
	\[ \calb=\bigl\{\Qtrp[1,0,*],\Qtrp[i,0,*],\Qtrp[0,1,*],\Qtrp[0,i,*]
	\bigr\} \] 
for $\QQ/Z$. With respect to $\calb$, the symplectic form $\bb$ defined by 
commutators has matrix $\pm\mxtwo0I{-I}0$, and conjugation by $\trp[1,0,0]$ 
(a generator of $\Out_{\SS}(\QQ)$) has matrix $\mxtwo{I}{-I}0I$ by Table 
\ref{tbl:<<->>}. 

We identify $\Out(\QQ)$ with $\Aut(\QQ/Z,\pm\bb)$: the group of 
automorphisms of $\QQ/Z$ that preserve $\bb$ up to sign. We have 
	\[ N_{\Aut(\QQ/Z)}(\Out_{\SS}(\QQ)) = 
	N_{\GL_4(3)}\bigl(\Gen{\mxtwo{I}I0I}\bigr) = \bigl\{ 
	\mxtwo{A}X0{\pm A} \,\big|\, A\in\GL_2(3),~ X\in M_2(\F_3) \bigr\}, 
	\]
and hence 
	\beqq \begin{split} 
	N_{\Out(\QQ)}&(\Out_{\SS}(\QQ)) = \Gen{ 
	\mxtwo{I}X0I, \mxtwo{A}00A, \mxtwo{I}00{-I} 
	\,\big|\, A,X\in M_2(\F_3), ~ X=X^t,~ AA^t=\pm I } \\ 
	&= \Gen{ \mxtwo{I}X0I,~ \mxtwo{A}00A,~ \mxtwo{I}00{-I} 
	\,\big|\, X=X^t,~ A\in\{\mxtwo11{-1}1,\mxtwo100{-1}\} } \\
	&\cong E_{27}\rtimes(\SD_{16}\times C_2). 
	\end{split} \label{e:N(OutSQ)} \eeqq
Here, each element of the form $\mxtwo{A}00{\pm A}$ in 
$N_{\Out(\QQ)}(\Out_{\SS}(\QQ))$ is conjugation by some element of $\NNN0$, and 
hence extends to an automorphism of $\MMnul$.

It remains to prove the lemma for automorphisms of the form $\mxtwo{I}X0I$ 
when $X=X^t$. 
Define $\alpha_1,\alpha_2,\alpha_3\in\Aut(\SS)$ as follows. 
In each case, $\alpha_i|_{\AA}=\Id$, and $\omega_i\:\TT\too \AA$ is such that 
$\alpha_i(g)=\omega_i(g)g$ for all $g\in \TT$:
	\begin{align*} 
	\alpha_1\bigl(\trp[a,b,c]\dpar{x}\bigr) &= 
	\trp[a,b+x,c+N(x)]\dpar{x} & 
	\omega_1(\dpar{x}) &= \trp[0,x,N(x)] \\
	\alpha_2\bigl(\trp[a,b,c]\dpar{x}\bigr) &= 
	\trp[a,b+\4x,c-\Tr(x^2)]\dpar{x} &
	\omega_2(\dpar{x}) &= \trp[0,\4x,-\Tr(x^2)] \\
	\alpha_3\bigl(\trp[a,b,c]\dpar{x}\bigr) &= 
	\trp[a,b+i\4x,c+\Tr(ix^2)]\dpar{x} &
	\omega_3(\dpar{x}) &= \trp[0,i\4x,\Tr(ix^2)] .
	\end{align*}
Each of the $\alpha_i$ is seen to be an automorphism of $\SS$ by checking the 
cocycle condition 
	\[ \omega_i(\dpar{x+y})=\omega_i(\dpar{x})
	+\9{\dpar{x}}\omega_i(\dpar{y}) \]
on $\omega_i$. (Note the relation $N(x+y)=(x+y)(\4x+\4y)=N(x)+N(y)+\Tr(\4xy)$.) 
The class of $\alpha_i|_{\QQ}$ as an automorphism of $\QQ/Z$ has 
matrix $\mxtwo{I}{X}0I$ for $X=I$, $\mxtwo100{-1}$, or $\mxtwo0110$, 
respectively, and thus the classes $[\alpha_i|_{\QQ}]$ generate 
$O_3(N_{\Out(\QQ)}(\Out_{\SS}(\QQ)))$ 
by \eqref{e:N(OutSQ)}. Since $\alpha_1$ is conjuation by $\trp[1,0,0]$, 
it extends to $\MMnul$. For $i=2,3$, the automorphism 
$\alpha_i$ extends to $\SS\gen{\beta}$ since $[\alpha_i,c_\beta]=1$ in 
$\Aut(\SS)$: this follows upon checking the relation 
$\9\beta\omega_i(\dpar{x})=\omega_i(\dpar{\9\beta x})$ using \eqref{e:c_beta}.

Recall that $\GGnul=\Matnul10\cong A_6$. Then 
$N_{\GGnul}(\TT)=\TT\gen{\beta}$ (see \eqref{e:betadef}), and the 
cohomology elements $[\omega_1],[\omega_2],[\omega_3]\in H^1(\TT;\AA)$ 
are all stable under the action of $\beta$. Since $\TT\in\syl3{\GGnul}$ 
is abelian, fusion in $\GGnul\cong A_6$ among subgroups of $\TT$ is 
controlled by $N_{\GGnul}(\TT)=\TT\gen{\beta}$, and hence the 
$[\omega_i]$ are stable under all fusion in $\GGnul$. So they are 
restrictions of elements of $H^1(\GGnul;\AA)$ by the stable elements 
theorem (see \cite[Theorem XII.10.1]{CE} or \cite[Theorem 
III.10.3]{Brown}), and each $\alpha_i$ extends to an automorphism 
$\4\alpha_i$ of $\MMnul=\AA\rtimes \GGnul$ that is the identity on 
$\AA$. 
\end{proof}

The next lemma is needed to handle the cases where $\GGnul\cong M_{11}$.

\begin{Lem} \label{l:Out(Q)11}
Assume Hypotheses \ref{h:hyp11} and Notation \ref{n:Gamma,T}, where 
$\GGnul\cong M_{11}$. Let $\QQ\nsg\SS$ be as in Proposition 
\ref{Q<|C(Z)}, set $\Delta=\Out_{\FF}(\QQ)$ and 
$\Delta_0=O^{3'}(\Delta)$. 
\begin{enuma} 

\item If we are in the $M_{11}$-case (i.e., if $|Z(\SS)|=9$), then there is 
$\4\gamma\in\Aut_{\FF}(\SS)$ of order $2$ that acts on 
$\QQ/Z$ via $(x\mapsto x^{-1})$. For each such $\4\gamma$, if we set 
$\gamma=[\4\gamma|_{\QQ}]\in\Out_{\FF}(\QQ)$, then 
	\[ \Delta \le C_{\Out(\QQ)}(\gamma) \cong \GGL_2(9). \]
If, furthermore, $1\ne U_0<U\in\syl3{C_{\Out(\QQ)}(\gamma)}$, 
and if $\xi\in C_{\Out(\QQ)}(\gamma)$ has $2$-power order and 
acts on $U$ by $(x\mapsto x^{-1})$, then for $H\cong2A_4$ or 
$H\cong2A_5$, there is a unique subgroup $X\le C_{\Out(\QQ)}(\gamma)$ 
isomorphic to $H$, containing $U_0$, and normalized by $\xi$.

\item If we are in the $M_{11}^*$-case (i.e., if $|Z(\SS)|=3$), then 
there is $\4\gamma\in\Aut_{\FF}(\SS)$ of order $4$ such that 
$[\4\gamma|_{\QQ}]\in\outf(\QQ)$ centralizes $\Out_{\SS}(\QQ)$. 
For each such $\4\gamma$, 
	\[ \Delta_0 = O^{3'}(C_{\Out(\QQ)}(\4\gamma|_{\QQ}))\cong \SL_2(9). \]

\end{enuma}
\end{Lem}

\begin{proof} Recall that $\MM=\AA\rtimes\GG$ is a model for 
$N_{\FF}(\AA)$, and $\MMnul=O^{3'}(\MM)=\AA\rtimes\GGnul$. 

\smallskip

\noindent\textbf{(a) } Assume we are in the $M_{11}$-case. By Lemma 
\ref{l:N10-11} and Table \ref{tbl:<<->>}, the element 
$\sbk{{-}1}\in\NNN1\cap \Matnul11\le\MM$ acts on $\QQ/Z$ via $(x\mapsto 
x^{-1})$. Set $\4\gamma=c_{\sbk{-1}}\in\Aut_{\FF}(\SS)$; thus 
$\4\gamma$ has order $2$ and inverts $\QQ/Z$. 


Now let $\4\gamma\in\Aut_{\FF}(\SS)$ be an arbitrary element of order 
$2$ that acts on $\QQ/Z$ via $(x\mapsto x^{-1})$, and set 
$\gamma=[\4\gamma|_{\QQ}]\in\Delta=\Out_{\FF}(\QQ)$. Since 
$\QQ\cong\UT_3(9)$ by the relations in Lemma \ref{l:Td10-11}(b), we can 
apply Lemma \ref{l:UT3(q)} to the group $\Out_{\FF}(\QQ)\le\Out(\QQ)$. 
By Lemma \ref{l:UT3(q)}(a,c) and since $\gamma\in\Delta$ has order $2$ 
and inverts all elements of $\QQ/Z$, we have 
$C_{\Out(\QQ)}(\gamma)\cong\GGL_2(9)$. By the same lemma and since 
$O_3(\Delta)=1$, $\Delta$ is sent isomorphically into $\Aut(\QQ/Z)$, 
and hence (since $\gamma$ is sent to $Z(\Aut(\QQ/Z))$) we have 
$\gamma\in Z(\Delta)$. So $\Delta\le C_{\Out(\QQ)}(\gamma)$.

Now fix subgroups $1\ne U_0<U\in\syl3{C_{\Out(\QQ)}(\gamma)}$, and an 
element $\xi\in C_{\Out(\QQ)}(\gamma)$ of $2$-power order that acts on 
$U$ by $(x\mapsto x^{-1})$. In particular, $|U|=9$ and $|U_0|=3$. Since 
$O^{3'}(C_{\Out(\QQ)}(\gamma))\cong\SL_2(9)\cong2A_6$, there is a 
surjective homomorphism $\Psi\:O^{3'}(C_{\Out(\QQ)}(\gamma))\too A_6$ 
with kernel of order $2$ such that $\Psi(U_0)$ is generated by a 
$3$-cycle. (Recall that $A_6$ has an outer automorphism that exchanges 
the two classes of elements of order $3$.) Also, $c_\xi$ induces (via 
$\Psi$) an automorphism $\xi'$ of $A_6$. Since $\xi'$ has $2$-power 
order and inverts all elements in $\Psi(U)$, it must be inner, and 
conjugation by a product of two disjoint transpositions. So there is a 
unique subgroup $\4X\le A_6$ that contains $\Psi(U_0)$, is normalized 
by $\xi'$, and is isomorphic to $H/Z(H)$ (i.e., to $A_4$ or $A_5$). 
Thus $X=\Psi^{-1}(\4X)$ is the unique subgroup satisfying the 
corresponding conditions in $\Out(\QQ)$. 

\smallskip

\noindent\textbf{(b) } Assume we are in the $M_{11}^*$-case. 
By Lemma \ref{l:GG<Mat} (and Notation \ref{n:A6+M11}), 
	\[ \Matnul11 = \GGnul \le \GG \le \twoMat11, \]
where $[\twoMat11{:}\Matnul11]=2$ by Table \ref{tbl:Mat}. By Table 
\ref{tbl:Mat} and Lemma \ref{l:GG<Mat}, 
$N_{\twoMat11}(\TT)/\TT=\NNN1/\TT\cong\SD_{16}\times C_2$, and hence this 
group has two subgroups of order $8$, generated by $\sbk{\zeta}$ and 
$-\sbk{\zeta}$, of which only the subgroup $\gen{-\sbk{\zeta}}$ lies in 
$\GGnul$. By Table \ref{tbl:<<->>}, these elements act on $\QQ/Z$ as 
follows:
	\beqq \9{\sbk{\zeta}}\Qtrp[a,b,*] = \Qtrp[\zeta^{-1}a,\zeta b,*] 
	\qquad\textup{and}\qquad
	\9{-\sbk{\zeta}}\Qtrp[a,b,*] = \Qtrp[\zeta^3a,\zeta b,*] . 
	\label{e:zeta.Q} \eeqq
By comparing characteristic polynomials or traces for the actions of 
the $\zeta^i$ on $\F_9$, we see that $\QQ/Z$ splits as a sum of two 
nonisomorphic irreducible $\F_3C_8$-modules under the action of 
$\gen{\sbk{\zeta}}$, while the two summands under the action of 
$\gen{-\sbk{\zeta}}$ are isomorphic.

Set $U=\Out_{\SS}(\QQ)=\Out_{\AA}(\QQ)\in\syl3\Delta$. 
Since $U\cong E_9$ and all elements of order $3$ in $U$ are 
in class $\3c$ or $\3d$ (see Table \ref{tbl:<<->>}), we have 
	\[ \Delta_0\cong2A_6\cong\SL_2(9) \] 
by Lemma \ref{l:Out(3^1+4)}. In 
particular, there is an element $\gamma_0\in N_{\Delta_0}(U)$ of order 
$8$ that acts on $\QQ/Z$, as an $\F_3C_8$-module, with two irreducible 
summands not isomorphic to each other. By the extension axiom, 
$\gamma_0$ extends to $\4\gamma_0\in\Aut_{\FF}(\SS)$, and 
$\4\gamma_0|_{\AA}\in N_{\GG}(\TT)$ has order $8$. By comparison with the 
formulas in \eqref{e:zeta.Q}, we see that $\4\gamma_0|_{\AA}$ must be 
conjugate to $\sbk{\zeta}$, and hence does not lie in $\GGnul$. Thus 
$\GG>\GGnul$, and hence $\GG=\twoMat11\cong M_{11}\times C_2$. So 
$c_{-\sbk{i}}^{\SS}\in\Aut_{\FF}(\SS)$, it has order $4$ and acts on $\AA$ by 
$\9{-\sbk{i}}\trp[r,s,t]=\trp[r,is,-t]$ (see Lemma \ref{l:Td10-11}(c)), 
and hence centralizes $U=\Out_{\SS}(\QQ)\cong\AA/A_*$. 

Now let $\4\gamma\in\Aut_{\FF}(\SS)$ be an arbitrary automorphism of 
order $4$ that centralizes $U=\Out_{\SS}(\QQ)$. Since 
$\Aut(\Delta_0)\cong\Aut(2A_6)\cong\Aut(A_6)$ where $\Out(A_6)\cong 
E_4$, and since each outer automorphism of $\Sigma_6$ exchanges 
$3$-cycles with products of disjoint $3$-cycles, we have 
$C_{\Aut(\Delta_0)}(U)\cong C_{\Sigma_6}(V)=V$ for 
$V\in\syl3{\Sigma_6}$. Since $\4\gamma|_{\QQ}\in\Delta$ acts on 
$\Delta_0$ and centralizes $U$ (and since $\4\gamma$ has order prime to 
$3$), we conclude that $c_{\4\gamma}^{\Delta_0}=\Id_{\Delta_0}$ and 
hence $\Delta_0\le C_{\Out(\QQ)}(\4\gamma)$. 

From the list in \cite{Dickson-Sp4(3)} of subgroups of $\PSp_4(3)$, we see 
that $\Delta_0\cong\SL_2(9)\cong2A_6$ has index $2$ in a maximal 
subgroup of $\Sp_4(3)$, and hence index $4$ in a maximal subgroup of 
$\Out(\QQ)\cong\SP_4(3)$. So $\Delta_0=O^{3'}(C_{\Out(\QQ)}(\4\gamma))$. 
\end{proof}


\newsubb{Fusion systems involving the Todd modules for 
\texorpdfstring{$M_{10}$ and $M_{11}$}{M10 and M11}}
{s:M11}

We are now ready to state and prove our main theorem on fusion systems 
satisfying Hypotheses \ref{h:hyp11}.

\begin{Thm} \label{t:A6+M11}
Let $\FF$ be a saturated fusion system over a finite $3$-group $\SS$, with 
a subgroup $\AA\le \SS$. Set $\GGnul=O^{3'}(\Aut_{\FF}(\AA))$, and assume 
that either
\begin{enumi} 
\item $\AA\cong E_{3^4}$ and $\GGnul\cong A_6$; or 
\item $\AA\cong E_{3^5}$ and $\GGnul\cong M_{11}$.
\end{enumi}
Assume also that $\AA\nnsg\FF$. Then $\AA\nsg 
\SS$, $\SS$ splits over $\AA$, $\FF$ is almost simple, and either 
\begin{enuma} 

\item $\GGnul\cong A_6$ and $O^{3'}(\FF)$ is isomorphic to the $3$-fusion 
system of one of the groups $U_4(3)$, $U_6(2)$, $\McL$, or $\Co_2$; or 

\item $\GGnul\cong M_{11}$, $|Z(\SS)|=9$, and $O^{3'}(\FF)$ is isomorphic 
to the $3$-fusion system of $\Suz$ or $\Ly$; or 

\item $\GGnul\cong M_{11}$, $|Z(\SS)|=3$, and $\FF$ is isomorphic to the 
$3$-fusion system of $\Co_3$.

\end{enuma}
\end{Thm}

(Note that (a), (b), and (c) correspond to the $A_6$-, $M_{11}$-, and 
$M_{11}^*$-cases, respectively.)

\begin{proof} By Lemma \ref{l:w.cl.}, $\AA\nsg \SS$ and is weakly 
closed in $\FF$. By the same lemma, $\AA$ is the unique 4-dimensional 
$\F_3A_6$-module if $\GGnul\cong A_6$, and $\AA$ is the Todd module or 
its dual if $\GGnul\cong M_{11}$. Also, $\SS$ splits over $\AA$ by 
Corollary \ref{c:Ssplit} and since $\AA\nnsg\FF$. So we are in the 
situation of Notation \ref{n:A6+M11} and \ref{n:Gamma,T}, and can use 
the terminology listed there. 

By Proposition \ref{Q<|C(Z)}, there is a unique special subgroup 
$\QQ\nsg\SS$ of exponent $3$ such that $Z(\QQ)=Z=Z(\SS)$, $\QQ\cap\AA=A_*$, 
and $\QQ/Z\cong E_{81}$. Also, $\EE[*]{C_{\FF}(Z)}=\{\QQ\}$, so $\QQ\nsg 
C_\calf(Z)$. Set $\GG=\Aut_{\FF}(\AA)$, $\Delta=\Out_{\FF}(\QQ)$, 
$\GGnul=O^{3'}(\GG)$, and $\Delta_0=O^{3'}(\Delta)$ for short. 

If $|Z|=3$ (i.e., if we are in the $A_6$- or $M_{11}^*$-case), then 
$\QQ\cong3^{1+4}_+$, and by Table \ref{tbl:Q}, $\Out_{\SS}(\QQ)\cong\SS/\QQ$ has 
order $3$ (if $\GGnul\cong A_6$) or $9$ (if $\GGnul\cong M_{11}$). Also, 
all elements of order $3$ in $\GGnul$ act on $\QQ/Z$ with two Jordan blocks 
of length $2$ (see Table \ref{tbl:<<->>}), and hence they have class $\3c$ or 
$\3d$ in $O^{3'}(\Out(\QQ))\cong\Sp_4(3)$ by Lemma \ref{l:3ABCD}. So by Lemma 
\ref{l:Out(3^1+4)}, $\Delta_0$ is isomorphic to $2A_4$, $2A_5$, $(Q_8\times 
Q_8)\rtimes C_3$, or $2^{1+4}_-.A_5$ if $\GGnul\cong A_6$, while 
$\Delta_0\cong2A_6$ if $\GGnul\cong M_{11}$. 

If $|Z|=9$, then $\GGnul\cong M_{11}$ and $\AA$ is its Todd module. 
Also, $\QQ\cong\UT_3(9)$ by the relations in Lemma \ref{l:Td10-11}(b). So 
$\Aut(\QQ)/O_3(\Aut(\QQ))\cong\GGL_2(9)$ by Lemma \ref{l:UT3(q)}(a,b). Since 
$O_3(\Delta_0)=1$ (recall $\QQ\in\EE{\FF}$ and hence $\Out(\QQ)$ has a 
strongly $3$-embedded subgroup), $\Delta_0$ is isomorphic to a subgroup 
of $\SL_2(9)$. The subgroups of $\SL_2(9)$ are well known, and since 
$\Out_{\SS}(\QQ)\cong\SS/\QQ$ has order $3$, we have $\Delta_0\cong2A_4$ or 
$2A_5$.

\begin{Table}[ht]
\[ \renewcommand{\arraystretch}{1.2}
\begin{array}{|c||c|c|c|c|c|c|c|} \hline
\GGnul & \multicolumn{4}{c|}{A_6} & \multicolumn{3}{c|}{M_{11}} \\\hline
\Delta_0 & 2A_4 & 2A_5 & (Q_8\times Q_8)\rtimes C_3 
& 2^{1+4}_-.A_5 & 2A_4 & 2A_5 & 2A_6 \\\hline
G^* & U_4(3) & \McL & U_6(2) & \Co_2 & \Suz & \Ly & \Co_3 \\\hline
\end{array} \]
\caption{} 
\label{tbl:G*} 
\end{Table}

Thus in all cases, $(\GGnul,\Delta_0)$ is one of the pairs listed in 
the first two rows of Table \ref{tbl:G*}. Let $G^*$ be the finite 
simple group listed in the table corresponding to the pair 
$(\GGnul,\Delta_0)$, and fix $S^*\in\syl3{G^*}$. If $G^*\cong U_4(3)$, 
then it has maximal parabolic subgroups of the form $E_{81}\rtimes A_6$ 
and $3^{1+4}_+.2\Sigma_4$, so $\calf_{S^*}(G^*)$ satisfies Hypotheses 
\ref{h:hyp11}, and there are subgroups $A^*,Q^*\nsg S^*$ such that 
$A^*\cong\AA$, $Q^*\cong \QQ$, $O^{3'}(\Aut_{G^*}(A^*))\cong A_6$, and 
$O^{3'}(\Aut_{G^*}(Q^*))\cong 2A_4$. In all of the other cases, we 
refer to the tables in \cite[pp. 7--40]{A-overgr}, which show that 
$\calf_{S^*}(G^*)$ also satisfies Hypotheses \ref{h:hyp11} with 
subgroups $A^*\cong\AA$ and $Q^*\cong \QQ$ such that 
$O^{3'}(\Aut_{G^*}(A^*))\cong\GGnul$ and 
$O^{3'}(\Aut_{G^*}(Q^*))\cong\Delta_0$.

Let $\MM$ be a model for $N_{\FF}(\AA)$ (see Proposition 
\ref{p:NF(Q)model}), and set $\MM^*=N_{G^*}(A^*)$. 
By Corollary \ref{c:Ssplit}, applied to $\FF$ and to 
$\calf_{S^*}(G^*)$, we have $O^{3'}(\MM)\cong\AA\rtimes\GGnul\cong 
O^{3'}(\MM^*)$. Choose an isomorphism 
$\varphi\:O^{3'}(\MM^*)\xto{~\cong~}O^{3'}(\MM)$ such that 
$\varphi(A^*)=\AA$ and $\varphi(S^*)=\SS$, and set 
$\calf^*=\9\varphi(\calf_{S^*}(G^*))$. Then $\calf^*$ is a fusion 
system over $\SS$ isomorphic to $\calf_{\SS^*}(G^*)$, and we 
will apply Proposition \ref{p:F1=F2-2} to show that 
$\calf^*=O^{3'}(\FF)$. 

The fusion system $\calf_{S^*}(G^*)$ is simple in all cases by 
Proposition 4.1(b), Proposition 4.5(a), or Table 4.1 in \cite{pprime}. 
(See also (16.3) and (16.10) in \cite{A-gfit}, which cover almost all 
cases.) So $\calf^*=O^{3'}(\calf^*)$. By construction, 
$O^{3'}(N_{\FF}(\AA))=O^{3'}(N_{\calf^*}(\AA))$. By Lemma 
\ref{l:(i-iv)}, the fusion systems $\FF$ and 
$\calf^*$ both satisfy Hypotheses 
\ref{h:F1=F2-2} with respect to $\AA,\QQ\nsg\SS$. So by Proposition 
\ref{p:F1=F2-2}, to show that $O^{3'}(\FF)=\calf^*$, it remains to show 
that $O^{3'}(\Out_{\FF}(\QQ))=O^{3'}(\Out_{\calf^*}(\QQ))$, and this will be 
shown by considering the three cases separately. Set 
	\[ \Gamma^*=\Aut_{\calf^*}(\AA),\quad 
	\Delta^*=\Out_{\calf^*}(\QQ),\quad 
	\Gamma_0^*=O^{3'}(\Gamma^*),\quad\textup{and}\quad 
	\Delta_0^*=O^{3'}(\Delta^*), \]
and note that $\Delta_0\cong\Delta_0^*$ in all cases by the choice of 
$G^*$.


\noindent\boldd{The $A_6$-case: } Since $\Delta_0\cong\Delta_0^*$ are 
both subgroups of $\Out(\QQ)$ with the same Sylow $3$-subgroup 
$\Out_{\SS}(\QQ)$, Lemma \ref{l:Out(3^1+4)} applies to show that they are 
conjugate in $\Out(\QQ)$, and hence $\Delta_0=\9{\gamma_0}\Delta_0^*$ for some 
$\gamma_0\in N_{\Aut(\QQ)}(\Aut_{\SS}(\QQ))$. By Lemma \ref{l:Out(Q)10}, 
$\gamma_0$ extends to some $\gamma\in\Aut(H_0)$, and $\gamma(\SS)=\SS$ 
since $\SS=\QQ\AA$. So upon replacing $\calf^*$ by 
$\9{(\gamma|_{\SS})}\calf^*$, we can arrange that $\Delta_0^*=\Delta_0$ 
without changing $\Gamma_0^*$.

\noindent\boldd{The $M_{11}$-case: } Let 
$\gamma\in\Aut_{\FF}(\SS)=\Aut_{\calf^*}(\SS)$ be as in Lemma 
\ref{l:Out(Q)11}(a): $\gamma$ has order $2$, and $\gamma|_{\QQ}$ acts 
on $\QQ/Z$ by inverting all elements. Then $\Delta_0,\Delta_0^*\le 
O^{3'}(C_{\Out(\QQ)}(\gamma|_{\QQ}))\cong\SL_2(9)\cong2A_6$ by that lemma. 

Set $U_0=\Out_{\SS}(\QQ)\cong C_3$, and let 
$U\in\syl3{C_{\Out(\QQ)}(\gamma|_{\QQ})}$ be the (unique) Sylow 3-subgroup 
that contains $U_0$. Set $h=-\sbk{\zeta}\in\NNN1\cap\Matnul11<\MMnul$ 
(see Lemma \ref{l:N10-11}), and set 
$\4\xi=c_h^{\SS}\in\Aut_{\FF}(\SS)=\Aut_{\calf^*}(\SS)$. Since $\QQ$ is 
weakly closed in $\FF$ and in $\calf^*$ by Proposition \ref{Q<|C(Z)}, 
we have $\xi\defeq[\4\xi|_{\QQ}]\in\Delta\cap\Delta^*$. So 
$\Delta_0\cong\Delta_0^*$ both contain $U_0$ and are normalized by 
$\xi$, and they are both isomorphic to $2A_4$ or $2A_5$. Hence 
$\Delta_0=\Delta_0^*$ by the last statement in Lemma \ref{l:Out(Q)11}(a).

\noindent\boldd{The $M_{11}^*$-case: } By Lemma 
\ref{l:Out(Q)11}(b), applied to either fusion system $\FF$ or 
$\calf^*$, there is $\4\gamma\in\Aut_{\FF}(\SS)=\Aut_{\calf^*}(\SS)$ of 
order $4$ such that $\4\gamma|_{\QQ}$ commutes with $\Aut_{\SS}(\QQ)$. 
By the same lemma, for any such $\gamma$, we have 
$\Delta_0=C_{\Out(\QQ)}(\gamma)=\Delta_0^*$. Also, in this case, since 
$G^*\cong\Co_3$, we have $\Out(\calf^*)\cong\Out(G^*)=1$ by 
\cite[Proposition 3.2]{O-sportame}, and hence $\FF=O^{3'}(\FF)$. 
\end{proof}

The automizers of the subgroups $\AA$ and $\QQ$ in each case of Theorem 
\ref{t:A6+M11} are described more explicitly in Table \ref{tbl:A6,M11}. We 
refer again to \cite[pp. 7--40]{A-overgr} in all cases except that of 
$U_4(3)$. 
\begin{Table}[ht]
\[ \renewcommand{\arraystretch}{1.2}
\begin{array}{l|ccc|ccc}
& \AA & \GGnul & \QQ & \GG=\Aut_{\FF}(\AA) & \Delta=\Out_{\FF}(\QQ) & G \\\hline
&&&& A_6 & 2\Sigma_4 & U_4(3) \\
&&&& \Sigma_6 & (Q_8\times Q_8)\rtimes \Sigma_3 & U_6(2) \\
\halfup[2.0]{\textup{$A_6$-case}} & \halfup[2.0]{E_{3^4}} & \halfup[2.0]{A_6} & 
\halfup[2.0]{3^{1+4}_+} & M_{10} & 2\Sigma_5 & \McL \\
&&&& (A_6\times C_2).E_4 & 2^{1+4}_-.\Sigma_5 & \Co_2 \\\hline
&&&& M_{11} & (2A_4\circ D_8).C_2 & \Suz \\
\halfup[2.0]{\textup{$M_{11}$-case}} & \halfup[2.0]{E_{3^5}} & 
\halfup[2.0]{M_{11}} & 
\halfup[2.0]{3^{2+4}} & M_{11}\times C_2 & (2A_5\circ C_8).C_2 & \Ly \\\hline
\textup{$M_{11}^*$-case} & E_{3^5} & M_{11} & 3^{1+4}_+ & 2\times M_{11} & 
(2A_6\circ C_4).C_2 & \Co_3 
\end{array} \]
\caption{In all cases, $\FF$ is a fusion system over $\SS=\AA\rtimes \TT$, 
and is realized by the group $G$. Also, $\AA\nsg\SS$ is abelian with 
$C_{\SS}(\AA)=\AA$ and $Z=Z(\SS)$, $\GG=\Aut_{\FF}(\AA)$, and 
$\GGnul=O^{3'}(\GG)$. The subsystem $C_G(Z)$ is constrained with 
$\QQ=O_3(C_G(Z))$ and $Z=Z(\QQ)$.} 
\label{tbl:A6,M11} 
\end{Table}

Note that by \cite[Theorem A(a,d)]{BMO1}, the $3$-fusion system of $U_6(2)$ 
is isomorphic to those of $U_6(q)$ for each $q\equiv2,5$ (mod $9$), and to 
those of $L_6(q)$ for each $q\equiv4,7$ (mod $9$). Thus $U_6(2)$ could be 
replaced by any of these other groups in the statement of Theorem 
\ref{t:A6+M11}.


\section{Some 3-local characterizations of the Conway groups} 
\label{s:conway}

We finish with some new $3$-local characterizations of the three Conway 
groups, $U_6(2)$, and McLaughlin's group. In each case, the new result is 
obtained by combining an earlier characterization of the some group with 
the classifications of fusion systems in Theorem \ref{t:M12case} or 
\ref{t:A6+M11}. It seems likely that one could get stronger results with a 
little more work, but we prove here only ones that follow easily from 
Theorems \ref{t:M12case} and \ref{t:A6+M11} together with the earlier 
characterizations.

We first combine Theorem \ref{t:M12case} with the 3-local characterization 
of $\Co_1$ shown by Salarian \cite{Salarian}, to get the following slightly 
simpler characterization. 

\begin{Thm} \label{t:3-loc.Co1} 
Let $G$ be a finite group. Assume $A\le S\in\syl3G$ are such that 
\begin{enum1} 

\item $A\cong E_{3^6}$, $C_G(A)=A$, and $N_G(A)/A\cong2M_{12}$; 

\item $A$ is not strongly closed in $S$ with respect to $G$; and 

\item $O_{3'}(C_G(Z(S)))=1$ and $|O_3(C_G(Z(S)))|>3$. 

\end{enum1}
Then $G\cong\Co_1$.
\end{Thm}

\begin{proof} By Salarian's theorem \cite[Theorem 1.1]{Salarian}, to show 
that $G\cong\Co_1$, it suffices to find subgroups $H_1,H_2\ge S\in\syl3G$ 
that satisfy the following three conditions:
\begin{enumi}
\item $H_1=N_G(Z(O_3(H_1)))$, $O_3(H_1)\cong3^{1+4}_\pm$, 
$H_1/O_3(H_1)\cong\Sp_4(3)\rtimes C_2$, and $C_{H_1}(O_3(H_1))=Z(O_3(H_1))$;
\item $O_3(H_2)\cong E_{3^6}$ and $H_2/O_3(H_2)\cong2M_{12}$; and 
\item $(H_1\cap H_2)/O_3(H_2)$ is an extension of an elementary abelian 
group of order $9$ by $\GL_2(3)\times C_2$.
\end{enumi}
Set $Z=Z(S)$, $H_1=N_G(Z)$ and $H_2=N_G(A)$. Since $H_1,H_2\ge S$ (recall 
$A\nsg S$ by assumption), it suffices to prove (i)--(iii).

Set $\calf=\calf_S(G)$. Then $A\nnsg\calf$ by (2), and hence $\calf$ is 
isomorphic to the fusion system of $\Co_1$ by (1) and Theorem 
\ref{t:M12case}. In particular, $S$ is isomorphic to the $3$-group $\SS$ of 
Notation \ref{n:M12} and \ref{n:M12-UWQ}, so we can identify $S$ with $\SS$ 
and use the notation defined there for subgroups of $\SS$. 

Condition (ii) holds by (1). Also, $(H_1\cap 
H_2)/O_3(H_2)=N_{H_2}(Z)/A\cong N_{\autf(A)}(Z)$ where 
$N_{\autf(A)}(Z)\cong(E_9\rtimes\GL_2(3))\times C_2$ by Lemma 
\ref{Q0-w.cl.}(c), so (iii) holds. 

Set $P=O_3(C_G(Z))$. Then $|P|>3$ by (3), so $P>Z$. Also, $P\nsg 
C_\calf(Z)$, so $P\le O_3(C_\calf(Z))=Q_0\cong3^{1+4}_+$ by Lemma 
\ref{Q<|CF(Z)-M12}. The action of $\Out_{C_\calf(Z)}(Q_0)\cong\Sp_4(3)$ on 
$Q_0/Z\cong E_{81}$ is irreducible, and hence $P=Q_0$. Thus 
$Q_0=O_3(C_G(Z))=O_3(H_1)$ since $C_G(Z)$ is normal of index at most $2$ in 
$H_1=N_G(Z)$. 

Now, $Q_0$ is $\calf$-centric by Lemma \ref{conj-Q0}, so 
$Z=Z(Q_0)\in\syl3{C_G(Q_0)}$, and hence $C_G(Q_0)=K\times Z(Q_0)=K\times Z$ 
for some $K$ of order prime to $3$. Also, $K\nsg C_G(Z)$ since $Q_0\nsg 
C_G(Z)$, so $K\le O_{3'}(C_G(Z))=1$ by (3). Thus $C_{H_1}(Q_0)=Z=Z(Q_0)$, 
and hence $H_1/Q_0\cong\outf(Q_0)$. Since $\outf(Q_0)\cong\Sp_4(3){:}2$ 
by Lemma \ref{l:Q0-props}, this finishes the proof of (i), 
and hence the proof of the theorem.
\end{proof}

The following $3$-local characterization of $\Co_3$ simplifies 
slightly that shown by Korchagina, Parker, and Rowley. 

\begin{Thm} \label{t:3-loc.Co3}
Let $G$ be a finite group. Assume $A\le S\in\syl3G$ are such that 
\begin{enum1}

\item $A\cong E_{3^5}$, $C_G(A)=A$, $|Z(S)|=3$, and 
$O^{3'}(N_G(A)/A)\cong M_{11}$; 

\item $A$ is not strongly closed in $S$ with respect to $G$; and 

\item $O_{3'}(C_G(Z(S)))=1$ and $|O_3(C_G(Z(S)))|>3$. 

\end{enum1}
Then $G\cong\Co_3$.
\end{Thm}

\begin{proof} By the theorem of Korchagina, Parker, and Rowley 
\cite[Theorem 1.1]{KPR}, to show that $G\cong\Co_3$, it suffices to 
find subgroups $M_1,M_2\le G$ and $A\le S$ that satisfy the following 
two conditions:
\begin{enumi}
\item $M_1=N_G(Z(S))$ is of the form $3^{1+4}_+.C_2.C_2.\PSL_2(9).C_2$; and 
\item $M_2=N_G(A)$ is of the form $E_{3^5}\rtimes(C_2\times M_{11})$.
\end{enumi}
Set $Z=Z(S)$, $M_1=N_G(Z)$ and $M_2=N_G(A)$; we claim that 
(i) and (ii) hold for this choice of subgroups. 

Set $\calf=\calf_S(G)$. Then $A\nnsg\calf$ by (2). By Table 
\ref{tbl:G0,A} and since $|Z|=3$ by (1), $A$ is the dual Todd module 
for $O^{3'}(\autf(A))\cong M_{11}$. Hence $\calf$ is isomorphic to the 
fusion system of $\Co_3$ by Theorem \ref{t:A6+M11}(c). In particular, 
$S$ is isomorphic to the $3$-group $\SS$ of Notation \ref{n:A6+M11}, so 
we can identify $S$ with $\SS$ and use the notation defined there for 
subgroups of $\SS$. 

Condition (ii) holds by (1), and since $N_G(A)/A\cong\autf(A)\cong 
M_{11}\times C_2$ by Table \ref{tbl:A6,M11}.

Set $P=O_3(C_G(Z))$. Then $|P|>3$ by (3), so $P>Z$. Also, $P\nsg 
C_\calf(Z)$, so $P\le O_3(C_\calf(Z))=Q\cong3^{1+4}_+$ by Proposition 
\ref{Q<|C(Z)}. Since $5\bmid|\SL_2(9)|$, the action of 
$\Out_{C_\calf(Z)}(Q)\cong\SL_2(9)$ on $Q/Z\cong E_{81}$ is 
irreducible, and hence $P=Q$. Thus $Q=O_3(C_G(Z))=O_3(M_1)$ since 
$C_G(Z)$ is normal of index at most $2$ in $M_1=N_G(Z)$. 

Now, $Q$ is $\calf$-centric by Proposition \ref{Q<|C(Z)}, so 
$Z=Z(Q)\in\syl3{C_G(Q)}$, and hence $C_G(Q)=K\times Z(Q)=K\times 
Z$ for some $K$ of order prime to $3$. Also, $K\nsg C_G(Z)$ since 
$Q\nsg C_G(Z)$, so $K\le O_{3'}(C_G(Z))=1$ by (3). Thus 
$C_{M_1}(Q)=Z=Z(Q)$, and hence $M_1/Q\cong\outf(Q)$. Since $\calf$ is the 
fusion system of $\Co_3$, and since $\outf(Q)\cong2(A_6\times C_2).C_2$ by 
Table \ref{tbl:A6,M11},
this finishes the proof of (i), and hence the proof of the theorem. 
\end{proof}

Finally, we combine Theorem \ref{t:A6+M11} with results of Parker, Rowley, 
and Stroth, to get some new $3$-local characterizations of $\McL$ and 
$U_6(2)$ as well as of $\Co_2$.

\begin{Thm} \label{t:3-loc.81.A6}
Let $G$ be a finite group, fix $S\in\syl3G$, and set $Z=Z(S)$. Assume 
$A\le S$ is such that 
\begin{enum1}

\item $A\cong E_{3^4}$, $C_G(A)=A$, and 
$O^{3'}(N_G(A)/A)\cong A_6$;

\item $A$ is not strongly closed in $S$ with respect to $G$; and 

\item $O_{3'}(C_G(Z))=1$ and $|O_3(C_G(Z))|>3$. 
\end{enum1}
Then $O_3(N_G(Z))\cong3^{1+4}_+$ and $C_G(O_3(C_G(Z)))=Z$. Also, 
the following hold, where $k$ denotes the index of $O^{3'}(N_G(A)/A)$ in 
$N_G(A)/A$:
\begin{enuma} 

\item If $5\bmid|C_G(Z)|$, then $G$ is isomorphic to $\McL$, 
$\Aut(\McL)$, or $\Co_2$, depending on 
whether $k=2$, $4$, or $8$, respectively. 

\item If $5\nmid|C_G(Z)|$, $|O_2(C_G(Z)/O_3(C_G(Z)))|\ge2^6$, and $k\le4$, 
then $G\cong U_6(2)$ or $U_6(2)\rtimes C_2$ when $k=2$ or $4$, respectively.

\end{enuma}
\end{Thm}

\begin{proof} Set $\calf=\calf_S(G)$. Then $A\nnsg\calf$ by (2). So by (1) 
and Theorem \ref{t:A6+M11}(a), $O^{3'}(\calf)$ is isomorphic to the fusion 
system of $\Co_2$, $U_4(3)$, $\McL$, or $U_6(2)$. 

Set $Q=O_3(C_\calf(Z))$: an extraspecial group of order $3^5$ with $Z(Q)=Z$ 
by Proposition \ref{Q<|C(Z)}. We claim that $Q/Z$ is a simple 
$\F_3\outf(Q)$-module. Assume otherwise, and consider the elements 
$a=\trp[1,0,0]\in S$ and $\beta=[c_{-\sbk{i}}]\in\outf(Q)$ in the notation 
of Tables \ref{tbl:Mat} and \ref{tbl:<<->>}. Assume $0\ne V<Q/Z$ is a 
proper nontrivial submodule, and choose $0\ne x\in V$. If $x\notin 
C_{Q/Z}(a)$, then the elements $[a,x]$, $\beta([a,x])$, $x$, $\beta(x)$ all 
lie in $V$ and generate $Q/Z$ (see Table \ref{tbl:<<->>}), contradicting 
the assumption that $V<Q/Z$. Thus $V\le C_{Q/Z}(a)$, with equality since 
$V\ge \gen{x,\beta(x)} = C_{Q/Z}(a)$. But if $C_{Q/Z}(a)$ were a submodule, 
then by Lemma \ref{QcharP}, $Q$ would not be $\calf$-essential, 
contradicting Proposition \ref{Q<|C(Z)}.

Set $P=O_3(C_G(Z))$. Then $P>Z$ by (3), and $P\le Q$ since $P\nsg 
C_\calf(Z)$. Also, $P/Z$ is an $\F_3\outf(Q)$-submodule of $Q/Z$, so 
$P=Q\cong3^{1+4}_+$ since $Q/Z$ is simple. 

Now, $Q$ is $\calf$-centric by Proposition \ref{Q<|C(Z)}, so 
$Z=Z(Q)\in\syl3{C_G(Q)}$, and hence $C_G(Q)=K\times Z(Q)=K\times Z$ for 
some $K$ of order prime to $3$. Also, $K\nsg C_G(Z)$ since $Q\nsg 
C_G(Z)$, so $K\le O_{3'}(C_G(Z))=1$ by (3). Thus $C_G(Q)=Z=Z(Q)$, and 
hence $C_G(Z)/Q\cong\outf(Q)$. 

If $5\bmid|C_G(Z)/Q|=|\outf(Q)|$, then by Table \ref{tbl:A6,M11} again, 
$O^{3'}(\calf)$ is the fusion system of $\McL$ or $\Co_2$. In the 
former case, $O^{3'}(N_G(Z))\cong3^{1+4}_+.2A_5$ and 
$C_G(O_3(C_G(Z)))=C_G(Q)\le Q$, so conditions (i)--(iii) in 
\cite[Theorem 1.1]{PS-McL} all hold, and $G\cong\McL$ or $\Aut(\McL)$ 
by that theorem (with $k=2$ or $4$).

If $O^{3'}(\calf)$ is the fusion system of $\Co_2$, then by Table 
\ref{tbl:A6,M11},
\begin{enumi}

\item $Q=O_3(C_G(Z))$ is extraspecial of order $3^5$, $O_2(C_G(Z)/Q)$ is 
extraspecial of order $2^5$, and $C_G(Z)/O_{3,2}(C_G(Z))\cong A_5$; and 

\item $Z$ is not weakly closed in $S$ with respect to $G$.

\end{enumi}
So $G\cong\Co_2$ by a theorem of Parker and Rowley \cite[Theorem 1.1]{PR-Co2}.
Also, $k=8$ in this case. 

If $5\nmid|C_G(Z)|$, $|O_2(C_G(Z)/Q)|\ge2^6$, and $k\le4$, then by Table 
\ref{tbl:A6,M11}, $C_G(Z)/Q$ contains $2A_4$ with index $k$ or $(Q_8\times 
Q_8)\rtimes C_3$ with index $k/2$, and the first would imply 
$|O_2(C_G(Z)/Q)|\le2^5$. So $O^{3'}(\calf)$ is the fusion system of 
$U_6(2)$, and $C_G(Z)/Q$ contains a normal subgroup isomorphic to 
$(Q_8\times Q_8)\rtimes C_3$. Hence $C_G(Z)$ is ``similar to a 
$3$-centralizer in a group of type $\PSU_6(2)$ or $F_4(2)$'' in the sense 
of Parker and Stroth \cite[Definition 1.1]{PS-U62}, and $F^*(G)\cong 
U_6(2)$ or $F_4(2)$ by \cite[Theorem 1.3]{PS-U62}. The group $F_4(2)$ does 
contain subgroups isomorphic to $E_{81}$ (a maximal torus and the Thompson 
subgroup of a Sylow 3-subgroup), but all such subgroups have automiser the 
Weyl group of $F_4$, and so we conclude that $G\cong U_6(2)$ or 
$U_6(2)\rtimes C_2$. 
\end{proof}


\appendix

\section{Some special \texorpdfstring{$p$-groups}{p-groups}}

In this section, we give a few elementary results on special or 
extraspecial $p$-groups and their automorphism groups. Most of them 
involve $p$-groups of the form $p^{2+4}$ or $p^{1+4}_+$, but we start 
with the following, slightly more general lemma.

\begin{Lem} \label{l:spec}
Fix a prime $p$, and let $Q$ be a finite nonabelian $p$-group such that 
$Z(Q)=[Q,Q]$ and is elementary abelian. Set $Z=Z(Q)$ and $\4Q=Q/Z$ for 
short. Then the following hold.
\begin{enuma} 

\item The quotient group $\4Q$ is elementary abelian, and hence $Q$ is 
a special $p$-group.

\item If $p$ is odd, then there is a homomorphism $\rho\:\4Q\too Z$ such that 
$g^p=\rho(gZ)$ for each $g\in Q$. 

\item Assume $\4Q\cong E_{p^3}$ and $Z\cong E_{p^2}$. Then there is a 
unique abelian subgroup $A\le Q$ of order $p^4$ and index $p$.

\item Assume $|\4Q|=p^4$, and $|Z|\le p^2$. Then for 
each $g\in Q\sminus Z$, there is an abelian subgroup $A\le Q$ of index 
$p^2$ such that $g\in A$, and $A$ is unique if $[g,Q]=Z\cong E_{p^2}$. If 
$|Z|=p^2$ and $[g,Q]=Z$ for each $g\in Q\sminus Z$, then there are exactly 
$p^2+1$ abelian subgroups of index $p^2$ in $Q$, any two of which intersect 
in $Z$. 

\end{enuma}
\end{Lem}

\begin{proof} Set $\4P=PZ/Z$ and $\4g=gZ\in Q/Z$ for each $H\le Q$ and 
$g\in Q$. Since $[Q,Q]\le Z(Q)$, the commutator map $\4Q\times\4Q\too 
Z$ is bilinear.

\smallskip

\noindent\textbf{(a) } For each $g,h\in Q$, we have $[g,h]\in Z$ and 
$[g,h]^p=1$ by assumption. Hence $[g^p,h]=1$ for all $h\in Q$, so 
$g^p\in Z(Q)=Z$, and $\4Q=Q/Z$ is elementary abelian. 

\smallskip

\noindent\textbf{(b) } For each $g,h\in Q$, since $[h,g]\in Z(Q)$, we 
have $(gh)^n=g^nh^n[h,g]^{n(n-1)/2}$ for each $n\ge1$. (Recall that 
$[h,g]=hgh^{-1}g^{-1}$ here.) So if $p$ is 
odd, then $(gh)^p=g^ph^p$ for each $g,h\in Q$. 

\smallskip

\noindent\textbf{(c) } Assume $|Q|=p^5$ and $|Z|=p^2$. Since $|[Q,Q]|>p$, 
there is at most one abelian subgroup of index $p$ in $Q$ (see 
\cite[Lemma 1.9]{indp1}). 

Fix $a,b,c\in Q$ such that $\{\4a,\4b,\4c\}$ is a basis for $\4Q\cong 
E_{p^3}$, and consider the three commutators $[a,b]$, $[a,c]$, and $[b,c]$. 
Since $\rk(Z)=2$, one of these is in the subgroup generated by the other 
two, and without loss of generality, we can assume there are $i,j\in\Z$ 
such that $[a,b]=[a,c]^i[b,c]^j=[a,c^i][b,c^j]$ (recall $[Q,Q]\le Z(Q)$). 
Then $[ac^j,bc^{-i}]=1$, and hence $Z\gen{ac^j,bc^{-i}}$ is abelian of 
index $p$ in $Q$. 

\smallskip

\noindent\textbf{(d) } Assume $\4Q\cong E_{p^4}$ and $|Z|\le p^2$, and 
fix $g\in Q\sminus Z$. Then commutator with $g$ defines a homomorphism 
$\chi\:Q/Z\gen{g}\too Z$, and this is not injective since 
$\rk(Q/Z\gen{g})>\rk(Z)$. So there is $h\in Q\sminus Z\gen{g}$ such 
that $[g,h]=1$ and $Z\gen{g,h}$ is abelian. If $[g,Q]=Z\cong E_{p^2}$, 
then $\chi$ is surjective, $\Ker(\chi)$ is generated by the class of 
$h$, and hence $Z\gen{g,h}$ is the only abelian subgroup of index $p^2$ 
in $Q$ containing $g$. 

Now assume $[g,Q]=Z\cong E_{p^2}$ for each $g\in Q\sminus Z$, and let 
$\cala$ be the set of abelian subgroups of index $p^2$ in $Q$. Then each 
$\4P\le \4Q$ of order $p$ is contained in $\4A$ for some unique 
$A\in\cala$, and each such $\4A$ has $p^2-1$ subgroups of order $p$. So 
$|\cala|=(p^4-1)/(p^2-1)=p^2+1$.
\end{proof}

In the rest of the section, we prove some more specialized results on 
certain special $p$-groups. Recall that for each prime power $q$ and 
each $n\ge2$, we let $\UT_n(q)$ denote the group of upper triangular 
$(n\times n)$ matrices with $1$'s on the diagonal. The groups $UT_3(q)$ are 
a special case of what Beisiegel calls ``semi-extraspecial 
$p$-groups'' in \cite{Beisiegel}.

\begin{Lem} \label{l:UT3(q)}
Let $p$ be an odd prime, and set $q=p^m$ for some $m\ge1$. 
Set $Q=\UT_3(q)$ and $Z=Z(Q)$, and let 
	\[ \Psi\: \Aut(Q) \Right4{} \Aut(Q/Z) \]
be the natural homomorphism. We regard $Q/Z$ as a $2$-dimensional 
$\F_q$-vector space in the canonical way.
\begin{enuma} 

\item The image $\Psi(\Aut(Q))$ is the group of all $\F_q$-semilinear 
automorphisms of $Q/Z$, hence isomorphic to $\GGL_2(q)$. For 
$\alpha\in\Aut(Q)$, we have $\alpha|_Z=\Id$ if and only if $\Psi(\alpha)$ 
is linear of determinant $1$. 

\item We have $\Ker(\Psi)=O_p(\Aut(Q))\cong\Hom(Q/Z,Z)\cong 
E_{p^n}$ where $n=2m^2$. 


\item Let $\gamma\in\Aut(Q)$ be any automorphism such that 
$\Psi(\gamma)=-\Id_{Q/Z}$. Then 
	\[ C_{\Aut(Q)}(\gamma) \cong C_{\Out(Q)}(\gamma) \cong 
	\Psi(\Aut(Q)). \]
More precisely, each $\4\alpha\in\Psi(\Aut(Q))$ is the image under $\Psi$ of a 
unique element in $C_{\Aut(Q)}(\gamma)$ and of a unique class in 
$C_{\Out(Q)}(\gamma)$, and hence 
	\[ \Aut(Q) = O_p(\Aut(Q)) \rtimes C_{\Aut(Q)}(\gamma)
	\quad\textup{and}\quad
	\Out(Q) = O_p(\Out(Q)) \rtimes C_{\Out(Q)}(\gamma). \]

\end{enuma}
\end{Lem}

\begin{proof} 
\noindent\textbf{(a,b) } See \cite[Proposition 5.3]{PR-BNpairs}.


\smallskip

\noindent\textbf{(c) } Set $U=\Ker(\Psi)=O_p(\Aut(Q))$ for short. Fix 
$\gamma\in\Aut(Q)$ such that $\Psi(\gamma)=-\Id$. Then 
$\gamma|_Z=\Id_Z$ since $Z=[Q,Q]$. Each $\beta\in U$ has the 
form $\beta(g)=g\chi(g)$ for some $\chi\in\Hom(Q,Z)$ with 
$Z\le\Ker(\chi)$, and 
	\[ (\9\gamma\beta)(g) =\gamma\beta(\gamma^{-1}(g)) 
	= \gamma(\gamma^{-1}(g)\chi(g^{-1}))=g\chi(g)^{-1}=\beta^{-1}(g). 
	\]
Thus $c_\gamma$ sends each element of $U$ to its inverse, and since 
$\gamma\in\alpha_{-I}U$ (where $\alpha_{-I}\in\Aut(Q)$ is defined as in the proof 
of (a)), we have $\gamma^2=(\alpha_{-I})^2=\Id$. Note also that 
$C_{\Aut(Q)}(\gamma)\cap U=1$. 

Fix $\alpha\in\Aut(Q)$. Then $[\alpha,\gamma]\in U$ since 
$\Psi(\gamma)\in Z(\Aut(Q/Z))$, so $c_\gamma$ sends the coset $\alpha U$ to 
itself. Since $\gamma^2=1$ and $|\alpha U|=|U|$ is odd (a power of $p$), 
there is some $\alpha'\in\alpha U\cap C_{\Aut(Q)}(\gamma)$. Since 
$C_{\Aut(Q)}(\gamma)\cap U=1$, there is at most one such element 
$\alpha'\in\alpha U$ centralized by $\gamma$.

A similar argument shows that each $[\alpha]\in\Out(Q)$ is congruent modulo 
$U/\Inn(Q)$ to a unique class of automorphisms that centralizes the class 
of $\gamma$ in $\Out(Q)$. 
\end{proof}

When working with automorphisms of extraspecial groups $3^{1+4}_+$, we 
will need to know the conjugacy classes of elements of order $3$ in 
$\Sp_4(3)$.


\begin{Lem} \label{l:3ABCD}
Let $V$ be a $4$-dimensional $\F_3$-vector space with nondegenerate 
symplectic form $\bb$. Thus $\Aut(V,\bb)\cong\Sp_4(3)$. There are four 
conjugacy classes of elements of order $3$ in $\Aut(V,\bb)$. 
\begin{enuma} 

\item The elements $g\in\Aut(V,\bb)$ in class $\3a$ or $\3b$ are those that 
act on $V$ with one Jordan block of length $2$ and two of length 
$1$. Also, $g\in\3a$ implies $g^{-1}\in\3b$. 

\item The elements $g\in\Aut(V,\bb)$ in class $\3c$ or $\3d$ are those that 
act on $V$ with two Jordan blocks of length $2$. If $\calb=\{v_1,v_2,v_3,v_4\}$ 
is a basis for $V$ with respect to which the form $\bb$ has matrix 
$\pm\mxtwo0I{-I}0$, and if $g$ has matrix $\mxtwo{I}X0I$ with respect 
to $\calb$, then $g\in\3c$ if $\det(X)=1$ and $g\in\3d$ if 
$\det(X)=-1$.

\end{enuma}
\end{Lem}

\begin{proof} The conjugacy classes of elements of order $3$ in 
$\Sp_4(3)$ were first determined by Dickson, in \cite[p. 
138]{Dickson-Sp4(q)}.

Fix $g\in\Aut(V,\bb)$ of order $3$. Its Jordan blocks 
have length at most $3$, so there must be at least two of them. Thus 
$\dim(C_V(g))\ge2$ and $C_V(g)\cap[g,V]\ne0$, so there are $v,w\in V$ 
such that $\{gv-v,w\}$ are linearly independent and lie in $C_V(g)$. 
Also, $(gv-v)\perp w$ since $g$ preserves $\bb$, and so 
$W=\gen{gv-v,w}\le C_V(g)$ is totally isotropic. 

Fix a basis $\calb=\{v_1,v_2,v_3,v_4\}$ such that $W=\gen{v_1,v_2}$, 
and with respect to which $\bb$ has matrix $\pm\mxtwo0I{-I}0$. Then $g$ 
has matrix $\mxtwo{I}X0B$ with respect to $\calb$, and $B=I$ and 
$X=X^t$ since $g$ preserves $\bb$. Such a matrix $\mxtwo{I}X0I$ has 
Jordan blocks of length $2+2$ if $\det(X)\ne0$, or of length $2+1+1$ if 
$\det(X)=0$, showing that such elements lie in at least two different 
conjugacy classes of subgroups. 

If $g$ and $h$ have matrices $\mxtwo{I}X0I$ and $\mxtwo{I}Y0I$, 
respectively, where $X$ and $Y$ are invertible, then $W=C_V(g)=C_V(h)$. 
So if they are conjugate in $\Aut(V,\bb)$, they are conjugate by a 
matrix of the form $\mxtwo{A}00{(A^t)^{-1}}$, and hence $Y=AXA^t$ and 
$\det(Y)=\det(X)\det(A)^2=\det(X)$. Thus there are at least three 
conjugacy classes of subgroups of order $3$, and since there are 
exactly three by \cite{Dickson-Sp4(q)}, they are distinguished by 
$\det(X)$ when there is a generator of the form $\mxtwo{I}X0I$.

There are 40 maximal isotropic subspaces, each of which is fixed by 
three subgroups of the form $\Gen{\mxtwo{I}X0I}$ for $\det(X)=1$, and 
six of that form with $\det(X)=-1$. Also, there are 40 
$3$-dimensional subspaces, each of which is fixed by exactly one subgroup of 
the form $\Gen{\mxtwo{I}X0I}$ with $\det(X)=0$. Hence there are 120, 
240, and 40 subgroups conjugate to $\gen{\mxtwo{I}X0I}$ for 
$\det(X)=1$, $-1$, and $0$, respectively. Since they are named in order 
of occurrence in the group, they correspond to the classes $\3c$, 
$\3d$, and $\3{ab}$, respectively. 
\end{proof}

Finally, we consider certain subgroups of extraspecial groups of order $3^5$.

\begin{Lem} \label{l:Out(3^1+4)}
Assume $Q$ is extraspecial of order $3^5$ and exponent $3$. Let $1\ne 
P\le\Out(Q)$ be such that $O_3(P)=1$, $O^{3'}(P)=P$, and each element of 
order $3$ in $P$ is of type $\3c$ or $\3d$. Then either
\begin{enuma} 

\item $P$ is isomorphic to $2A_4$, $2A_5$, or $(Q_8\times Q_8)\rtimes C_3$, 
in each of which cases there is one $\Sp_4(3)$-conjugacy class containing 
elements of type $\3c$ and one containing elements of type $\3d$; or 

\item $P\cong2^{1+4}_-.A_5$ or $2A_6$, in each of which cases there is just 
one conjugacy class. 

\end{enuma}
\end{Lem}

\begin{proof} Set $Z=Z(Q)$ and $V=Q/Z$, and let $\bb$ be the symplectic 
form on $V$ defined by taking commutators in $Q$. Thus $V$ is a 
$4$-dimensional vector space over $\F_3$, and 
$O^{3'}(\Out(Q))\cong\Aut(V,\bb)\cong\Sp_4(3)$. Let $R\le 
O^{3'}(\Out(Q))$ be a maximal subgroup that contains $P$. By a theorem 
of Dickson \cite[\S\,71]{Dickson-Sp4(3)} (see also \cite[Theorem 
10]{Mitchell-Sp4}), $R$ must lie in one of five conjugacy classes. 
\begin{itemize} 

\item If $R$ is in one of the two classes of maximal parabolic 
subgroups, then $O^{3'}(R)/O_3(R)\cong\SL_2(3)\cong2A_4$. Since $O_3(P)=1$, 
it follows that $P\cong2A_4$. 

\item If $R\cong\Sp_2(3)\wr C_2\cong2A_4\wr C_2$, then $P\le 
O^{3'}(R)\cong2A_4\times2A_4$, and $V$ splits as a direct sum of 
2-dimensional $\F_3P$-submodules. Each $g\in P$ of order $3$ is in class 
$\3c$ or $\3d$ and hence acts on $V$ with two Jordan blocks of length $2$, 
and thus acts nontrivially on each of the two direct summands. In other 
words, each such $g$ acts diagonally on $O_2(R)\cong Q_8\times Q_8$, and so 
$P\le(Q_8\times Q_8)\rtimes C_3$. Hence either $P=(Q_8\times 
Q_8)\rtimes C_3$, or $P\cong2A_4$ diagonally embedded in $2A_4\times2A_4$. 

\item If $O^2(R)\cong\Sp_2(9)\cong2A_6$, then from a list of subgroups 
of $2A_6\cong\SL_2(9)$ (see \cite[Theorem 6.5.1]{GLS3}), we see that 
$P\cong2A_4$, $2A_5$, or $2A_6$. 

\item Assume $R\cong2^{1+4}_-.A_5$, and let $\5P$ be the image of $P$ 
in $R/O_2(R)\cong A_5$. Then $\5P\cong C_3$, $A_4$, or $A_5$: these are 
up to conjugacy the only nontrivial subgroups of $A_5$ generated by 
elements of order $3$. Also, $P$ acts faithfully on $O_2(R)/Z(R)\cong 
E_{16}$. Since $O_3(R)=1$, $P$ must be isomorphic to one of the 
following groups: 
	\begin{align*} 
	\5P\cong C_3 \quad&\implies\quad P\cong Q_8\rtimes C_3; \\
	\5P\cong A_4 \quad&\implies\quad P\cong A_4,~ 2A_4,~ 
	\textup{or}~
	2^{1+3}.A_4\cong(Q_8\times Q_8)\rtimes C_3; \\
	\5P\cong A_5 \quad&\implies\quad P\cong A_5,~ 2A_5,~ \textup{or}~ 
	2^{1+4}_-.A_5. 
	\end{align*}
The groups $A_4$ and $A_5$ cannot occur as subgroups of 
$\Sp_4(3)$, since an element of order $3$ would have to permute three 
distinct eigenspaces for the action of $O_2(A_4)\cong E_4$, hence have 
a Jordan block of length $3$, which contradicts Lemma \ref{l:3ABCD}.

\end{itemize}

Thus $P$ is isomorphic to $2A_4$, $2A_5$, $(Q_8\times Q_8)\rtimes C_3$, 
$2^{1+4}_-.A_5$, or $2A_6$. By \S11 and \S46 in \cite{Dickson-Sp4(3)}, 
there are two conjugacy classes of subgroups isomorphic to $2A_4$ and 
two of subgroups isomorphic to $2A_5$. Since 
$2A_6\cong\SL_2(9)<\Sp_4(3)$ has elements of both types $\3c$ and $\3d$ 
(the elements $\mxtwo1101$ and $\mxtwo1\zeta01$ are in different 
classes by the criterion in Lemma \ref{l:3ABCD}), the two classes in 
each case are distinguished by having elements of type $\3c$ or $\3d$. 
Likewise, by \cite[\S\,49]{Dickson-Sp4(3)}, there are two classes of 
subgroups of the form $(Q_8\times Q_8)\rtimes C_3$ (and not isomorphic 
to $2A_4\times Q_8$), and they are also distinguished by having 
elements of type $\3c$ or $\3d$. Finally, by 
\cite[\S\,61 \& \S\,68]{Dickson-Sp4(3)}, there is just one conjugacy class of 
subgroups isomorphic to $2A_6$ and one of subgroups isomorphic to 
$2^{1+4}_-.A_5$. 
\end{proof}

We finish the section with the following well known and elementary lemma. 

\begin{Lem} \label{l:[x,[x,a]]}
Fix a prime $p$. Let $G$ be a finite $p$-group, let $A\nsg G$ be a normal 
elementary abelian $p$-subgroup, and assume $x\in G\sminus A$ is such 
that $x^p\in A$. Let $\Phi_x\in\End(A)$ be the homomorphism 
$\Phi_x(a)=[x,a]=\9xa\cdot a^{-1}$. Then for each $a\in A$, $(ax)^p=x^p$ 
if and only if $(\Phi_x)^{p-1}(a)=1$. 
\end{Lem}


\begin{proof} Set $U=A\gen{x}/A\cong C_p$ and $u=xA\in U$, and regard $A$ as 
an $\F_pU$-module. Then 
	\[ (ax)^p = a\cdot\9xa\cdots\9{x^{p-1}}a\cdot x^p 
	= \bigl((1+u+\dots+u^{p-1})a\bigr) \cdot x^p 
	= (u-1)^{p-1}a \cdot x^p = \Phi_x^{p-1}(a)\cdot 
	x^p \]
(in additive notation). So $(ax)^p=x^p$ if and only if $\Phi_x^{p-1}(a)=0$. 
\end{proof}



\section{Strongly \texorpdfstring{$p$-}{p-}embedded subgroups}
\label{s:str.emb.}

We collect here some of the basic properties, especially for odd primes 
$p$, of finite groups with strongly $p$-embedded subgroups. All of the 
results here are proven independently of the classification of finite 
simple groups (but see remarks in the proof of Proposition 
\ref{p:str.emb.4}).

\begin{Lem} \label{l:str.emb2}
Let $G$ be a finite group, and let $G_0\nsg G$ be normal of index prime to 
$p$. Then $G_0$ has a strongly $p$-embedded subgroup if and only if $G$ 
does.
\end{Lem}

\begin{proof} Recall (see \cite[Theorem X.4.11(b)]{HB3}) that $G$ has a 
strongly $p$-embedded subgroup if and only if there is a partition 
$\sylp{G}=X_1\amalg X_2$, with $X_1,X_2\ne\emptyset$, such that for 
each $S_1\in X_1$ and $S_2\in X_2$, we have $S_1\cap S_2=1$ ($G$ is 
``$p$-isolated'' in the terminology of \cite{HB3}). Since 
$\sylp{G_0}=\sylp{G}$, the lemma follows immediately. 
\end{proof}

\begin{Lem} \label{l:str.emb.}
Let $G$ be a finite group with a strongly $p$-embedded subgroup $H<G$. 
\begin{enuma} 

\item Each proper subgroup $\5H<G$ that contains $H$ is also strongly 
$p$-embedded in $G$.

\item For each normal subgroup $K\nsg G$, either $HK/K$ is 
strongly $p$-embedded in $G/K$, or $HK=G$, or $p\nmid|G/K|$.

\end{enuma}
\end{Lem}

\begin{proof} \textbf{(a) } Assume $H\le\5H<G$. If $g\in 
G\sminus\5H$ is such that $p\mid|\5H\cap\9[1]g\5H|$, then there is 
$x\in\5H\cap\9[1]g\5H$ of order $p$. Since $H$ contains a Sylow 
$p$-subgroup of $\5H$, there are $a,b\in\5H$ such that $x\in\9[1]aH$ and 
$x\in\9[1]{gb}H$. Thus 
$p\mid|\9[1]aH\cap\9[1]{gb}H|=|H\cap\9{a^{-1}gb}H|$, so $a^{-1}gb\in H$ 
since $H$ is strongly $p$-embedded. Hence $g\in\5H$ since $a,b\in\5H$. So 
$\5H$ is also strongly $p$-embedded in $G$.

\smallskip

\noindent\textbf{(b) } If $K\nsg G$ and $HK<G$, then $HK$ is strongly 
$p$-embedded in $G$ by (a). 
Hence $HK/K$ is strongly $p$-embedded in $G/K$ if $p\mid|HK/K|$; 
equivalently, if $p\mid|G/K|$.
\end{proof}

The next few lemmas provides different ways of showing that certain groups 
do not have strongly $p$-embedded subgroups.

\begin{Lem} \label{l:G<prod}
Fix a finite group $G$ containing a strongly $p$-embedded subgroup. Let 
$\{K_i\}_{i\in I}$ be a finite set of normal subgroups, set 
$K_{I_0}=\bigcap_{i\in I_0}K_i$ for each $I_0\subseteq I$, and assume 
$K_I=1$. Let $J\subseteq I$ be the set of those $i\in I$ such that 
$p\nmid|K_i|$. Then the following hold.
\begin{enuma} 

\item In all cases, $J\ne\emptyset$ and $G/K_J$ has a strongly $p$-embedded 
subgroup. 

\item If $p^2\nmid|G|$, or (more generally) if there is a $p$-subgroup $T\le 
G$ such that $N_G(T)$ is strongly $p$-embedded in $G$, then there is 
$j\in J$ such that $G/K_j$ has a strongly $p$-embedded subgroup.

\end{enuma}
\end{Lem}

\begin{proof} Fix $S\in\sylp{G}$, and let $H<G$ be the minimal strongly 
$p$-embedded subgroup that contains $S$. 

\smallskip

\noindent\textbf{(a) } We show this by induction on $|I\sminus J|$. If 
$I=J$, there is nothing to prove, so assume $I\supsetneqq J$, fix $i_0\in 
I\sminus J$, and set $I_0=I\sminus\{i_0\}$. Then $p\mid|K_{i_0}|$ and 
$K_{i_0}\cap K_{I_0}=1$, so $I_0\ne\emptyset$ and $[K_{i_0},K_{I_0}]=1$. 
For each $g\in K_{I_0}$, we have $H\cap K_{i_0}\le C_H(g)\le H\cap\9gH$, 
and $p\mid|H\cap K_{i_0}|$ since $S$ contains some Sylow $p$-subgroup of 
$K_{i_0}$. Thus $g\in H$, and so $K_{I_0}\le H$. So $p\nmid|K_{I_0}|$, and 
$H/K_{I_0}$ is strongly $p$-embedded in $G/K_{I_0}$ by Lemma 
\ref{l:str.emb.}(b). Since $|I_0\sminus 
J|<|I\sminus J|$, we now conclude by the induction hypothesis (applied to 
the group $G/K_{I_0}$ and the subgroups $\{K_i/K_{I_0}\}_{i\in I_0}$) that 
$J\ne\emptyset$, and that $G/K_J$ has a strongly $p$-embedded subgroup.

\smallskip

\noindent\textbf{(b) } Assume $T\le S$ is such that $H=N_G(T)$ is 
strongly $p$-embedded in $G$. In particular, if $|S|=p$, this holds for 
$T=S$. We must show that $G/K_j$ has a strongly $p$-embedded subgroup 
for some $j\in J$, and it suffices to do this when $I=J$ and $|J|=2$; 
e.g., when $I=J=\{1,2\}$. Thus $K_1\cap K_2=1$, and $p\nmid|K_i|$ for 
$i=1,2$. Set $K=K_1K_2$. 

Assume neither $G/K_1$ nor $G/K_2$ contains a strongly $p$-embedded 
subgroup. Then $G=HK_1=HK_2$ by Lemma \ref{l:str.emb.}(b). Also, 
	\[ [H\cap K,T] = [N_K(T),T] \le T\cap K=1, \] 
and $H\cap K=N_K(T)=C_K(T)$. So for $i=1,2$, we have 
$K=(H\cap K)K_i=C_K(T)K_i$ since $G=HK_i$, and hence $[K,T]=[K_i,T]\le K_i$. 

Thus $[K,T]\le K_1\cap K_2=1$. But then $K$ and $H$ both normalize $T$, 
so $G=HK$ normalizes $T$, contradicting the assumption that 
$H=N_G(T)<G$. 
\end{proof}

The next lemma is an easy consequence of the well known list of 
subgroups of $\PSL_3(p)$. 

\begin{Lem} \label{l:G<GLn(p)}
Fix a prime $p$ and $n\ge2$. Let $G\le\GL_n(p)$ be a subgroup such that 
$G\ngeq\SL_n(p)$, $p^2\mid|G|$, and $G$ acts irreducibly on $\F_p^n$. Then 
$n\ge4$. 
\end{Lem}

\begin{proof} Since $p^2\nmid|\GL_2(p)|$, we have $n\ge3$. From the list of 
maximal subgroups of $\PSL_3(p)$ (see \cite[Theorem 6.5.3]{GLS3}), we see 
that there is no proper subgroup $G<\SL_3(p)$ (hence none in $\GL_3(p)$) 
such that $G$ is irreducible on $\F_p^3$ and $p^2\bmid|H|$. So $n\ge4$. 
\end{proof}

In the next few lemmas, $\Phi(P)$ denotes the Frattini subgroup of a 
finite $p$-group $P$. 

\begin{Lem} \label{l:P0<..<P}
Let $P$ be a finite $p$-group, and let $P_0\le P_1\le\cdots\le P_m=P$ 
be a sequence of subgroups, all normal in $P$, and such that 
$P_0\le\Phi(P)$. Let $\alpha\in\Aut(P)$ be such that $[\alpha,P_i]\le 
P_{i-1}$ for all $1\le i\le m$. Then $\alpha$ has $p$-power order. 
\end{Lem}

\begin{proof} For each such $\alpha$, $\alpha/P_0\in\Aut(P/P_0)$ 
has $p$-power order by \cite[Theorem 5.3.2]{Gorenstein}, and hence 
$\alpha$ has $p$-power order by \cite[Theorem 5.1.4]{Gorenstein}. 
\end{proof}

\begin{Lem} \label{l:filtered}
Let $\calf$ be a saturated fusion system over a finite $p$-group $S$, and 
assume $P\in\EE\calf$. Let $P_0\le P_1\le \dots\le P_m=P$ be a sequence of 
subgroups such that $P_0\le\Phi(P)$, and such that $P_i$ is normalized by 
$\autf(P)$ for each $0\le i\le m$. Assume also that $[P,P_i]\le P_{i-1}$ 
for each $1\le i\le m$. 
\begin{enuma} 

\item If $|N_S(P)/P|=p$, then there is at least one index $i=1,\dots,m$ 
such that $\rk(P_i/P_{i-1})\ge2$, and such that the image of $\autf(P)$ 
in $\Aut(P_i/P_{i-1})$ has a strongly $p$-embedded subgroup.

\item If $|N_S(P)/P|\ge p^2$, then there is at least one index 
$i=1,\dots,m$ such that $\rk(P_i/P_{i-1})\ge4$. If there is a unique 
such index $i$, then the image of $\autf(P)$ in $\Aut(P_i/P_{i-1})$ has 
a strongly $p$-embedded subgroup.

\end{enuma}
\end{Lem}

\begin{proof} Fix $i=1,\dots,m$. Since $[P,P_i]\le P_{i-1}$, 
the homomorphism $\autf(P)\too\Aut(P_i/P_{i-1})$ induced by restriction to 
$P_i$ contains $\Inn(P)$ in its kernel, and hence factors through a 
homomorphism $\varphi_i\:\outf(P)\too\Aut(P_i/P_{i-1})$. Set 
$K_i=\Ker(\varphi_i)\nsg\outf(P)$.

Assume that $\alpha\in\autf(P)$ is such that its class 
$[\alpha]\in\outf(P)$ lies in $\bigcap_{i=1}^mK_i$. Thus 
$[\alpha,P_i]\le P_{i-1}$ for each $i$, so $\alpha$ has $p$-power order 
by Lemma \ref{l:P0<..<P} and since $P_0\le\Phi(P)$. So 
$\bigcap_{i=1}^mK_i$ is a normal $p$-subgroup of $\outf(P)$. Since 
$\outf(P)$ has a strongly $p$-embedded subgroup (recall 
$P\in\EE\calf$), we have $O_p(\outf(P))=1$ (recall $O_p(-)$ is 
contained in all Sylow $p$-subgroups), and hence 
$\bigcap_{i=1}^mK_i=1$. We are thus in the situation of Lemma 
\ref{l:G<prod}. 

Recall that $N_S(P)/P\cong\Out_S(P)\in\sylp{\outf(P)}$. As in Lemma 
\ref{l:G<prod}, let $J$ be the set of all $i=1,\dots,m$ such that $|K_i|$ 
is prime to $p$, and set $K_J=\bigcap_{j\in J}K_j$. By Lemma 
\ref{l:G<prod}(a), $J\ne\emptyset$ and $\outf(P)/K_J$ contains a strongly 
$p$-embedded subgroup. 

Without loss of generality, in both points (a) and (b), we can assume 
that the filtration by the $P_i$ is maximal. Thus each quotient 
$P_i/P_{i-1}$ is elementary abelian, and the action of $\outf(P)$ on it 
is irreducible.

\smallskip

\noindent\textbf{(a) } If $|\Out_S(P)|=p$, then by Lemma 
\ref{l:G<prod}(b), there is $j\in J$ such that 
$\Im(\varphi_j)\cong\outf(P)/K_j$ contains a strongly $p$-embedded 
subgroup. 

\smallskip

\noindent\textbf{(b) } Now assume $|\Out_S(P)|\ge p^2$. Recall that the 
action of $\outf(P)$ on $P_j/P_{j-1}$ is irreducible for each $j\in J$. So 
$\rk(P_j/P_{j-1})\ge4$ for each $j\in J$ by Lemma \ref{l:G<GLn(p)}. In 
particular, if there is a unique $i$ such that $\rk(P_i/P_{i-1})\ge4$, then 
$|J|=1$, and $\outf(P)/K_j$ has a strongly $p$-embedded subgroup for $j\in 
J$. 
\end{proof}

The next lemma provides another way to show that certain subgroups of a 
$p$-group $S$ cannot be essential in any fusion system over $S$.

\begin{Lem} \label{l:filtered-c}
Let $\calf$ be a saturated fusion system over a finite $p$-group $S$. 
Assume $P<S$ and $T\le\Aut_S(P)$ are such that $|T/(T\cap\Inn(P))|\ge 
p^2$ and $[P:C_P(T)]=p$. Then $P\notin\EE\calf$. 
\end{Lem}

\begin{proof} Assume otherwise: assume $P$ is $\calf$-essential. Set 
$G=\outf(P)$, and set $\4T=T{\cdot}\Inn(P)/\Inn(P)\le\Out_S(P)$. Thus 
$|\4T|\ge p^2$ by assumption. Let $H<G$ be a strongly $p$-embedded 
subgroup that contains $\Out_S(P)\in\sylp{G}$. Fix $g\in G\sminus H$, 
and set $K=\gen{T,\9gT}\le\autf(P)$. Since $H$ is strongly $p$-embedded 
and $g\notin H$, no $p$-subgroup of $G$ can intersect nontrivially with 
both $\4T$ and $\9g\4T$, and in particular, 
	\beqq \textup{either\quad $O_p(K)\cap T\le\Inn(P)$\quad or\quad 
	$O_p(K)\cap\9gT\le\Inn(P)$.} 
	\label{e:OpK} \eeqq

By assumption, $C_P(T)$ has index $p$ in $P$, and so does $C_P(\9gT)$. 
If $C_P(T)=C_P(\9gT)$, then $K$ is an abelian $p$-group, 
contradicting \eqref{e:OpK}. So $C_P(K)=C_P(T)\cap C_P(\9gT)$ has index 
$p^2$ in $P$, and $P/C_P(K)\cong E_{p^2}$. The group of elements of $K$ 
that induce the identity on $P/C_P(K)$ is a $p$-group by 
Lemma \ref{l:P0<..<P}, and hence contained in $O_p(K)$. 
Since $p^2\nmid|\GL_2(p)|$, we have $[T:O_p(K)\cap T]\le p$, and since 
$|\4T|\ge p^2$, this implies $O_p(K)\cap T\nleq\Inn(P)$. But 
$O_p(K)\cap\9gT\nleq\Inn(P)$ by a similar argument, this again
contradicts \eqref{e:OpK}, and so $P$ cannot be $\calf$-essential. 
\end{proof}

The next lemma gives yet another simple criterion for a subgroup not to be 
essential. Again, $\Phi(-)$ denotes the Frattini subgroup. 

\begin{Lem} \label{QcharP}
Let $\calf$ be a saturated fusion system over a finite $p$-group $S$, 
and fix $P\le S$. Assume there are subgroups $P_0\nsg P_1\nsg\cdots\nsg 
P_k= P$, all normalized by $\autf(P)$, such that $P_0\le\Phi(P)$. 
Assume also there is $x\in N_S(P)\sminus P$ such that 
$[x,P_i]\le P_{i-1}$ for each $1\le i\le k$. Then $P\notin\EE\calf$. 
\end{Lem}

\begin{proof} By Lemma \ref{l:P0<..<P}
and since $P_0\le\Phi(P)$, the group $\Gamma$ of all $\alpha\in\Aut(P)$ 
such that $[\alpha,P_i]\le P_{i-1}$ for $1\le i\le k$ is a $p$-subgroup 
of $\Aut(P)$, and $\Gamma\cap\autf(P)$ is normal in $\autf(P)$ 
since the $P_i$ are normalized by $\autf(P)$. So $c_x\in O_p(\autf(P))$, 
and either $c_x\in\Inn(P)$, in which case $x\in PC_S(P)\sminus P$ and 
hence $P$ is not $\calf$-centric, or $O_p(\outf(P))\ne1$, in which case 
$\outf(P)$ has no strongly $p$-embedded subgroup (since $O_p(-)$ is 
contained in all Sylow $p$-subgroups). In either case, 
$P\notin\EE\calf$. 
\end{proof}

We finish by listing the subgroups of $\SL_4(p)$ that have strongly 
$p$-embedded subgroups and order a multiple of $p^2$. We indicate how to 
arrange the proof so as to be independent of the classification of finite 
simple groups.

\begin{Prop} \label{p:str.emb.4}
Fix an odd prime $p$, let $V$ be a $4$-dimensional vector space over 
$\F_p$, and let $H<G\le\Aut(V)$ be such that $p^2\mid|G|$ and $H$ is 
strongly $p$-embedded in $G$. Set $G_0=O^{p'}(G)$. Then either 
$G_0\cong\SL_2(p^2)$ and $V$ is its natural module, in which case each 
element of order $p$ in $G_0$ acts on $V$ with two Jordan blocks of length 
$2$; or $G_0\cong\PSL_2(p^2)$ and $V$ is the natural 
$\Omega_4^-(p)$-module, in which case each element of order $p$ in $G_0$ 
acts on $V$ with Jordan blocks of lengths $1$ and $3$. 
\end{Prop}

\begin{proof} By Aschbacher's theorem \cite{A-max}, applied to the finite 
simple classical group $\PSL_4(p)$, either $G$ is contained in a member of 
one of the ``geometric'' classes $\scrc_i$ ($1\le i\le 8$) defined in 
\cite{A-max}, or the image of $G$ in $\Aut(V)/Z(\Aut(V))\cong\PGL_4(p)$ is 
almost simple. 

By Lemma \ref{l:str.emb2}, $G_0=O^{p'}(G)$ also has a strongly 
$p$-embedded subgroup. 

\smallskip

\noindent\textbf{Case 1: } Assume $G$ is contained in a member of 
Aschbacher's class $\scrc_k$, for some $1\le k\le8$. Since $\F_p$ has no 
proper subfields, the class $\scrc_5$ is empty. 

If $k=1$ or $k=2$, then $G_0$ acts reducibly on $V$, contradicting 
Lemma \ref{l:filtered}(b). 

If $k=3$, then $G_0$ is contained in $\SL_2(p^2)$ (where $V$ is the 
natural module). Since $\SL_2(p^2)$ is generated by any two of its Sylow 
$p$-subgroups (and since they have order $p^2$), $G_0$ cannot be a 
proper subgroup of $\SL_2(p^2)$. 

If $k=4$ or $k=7$, then the restriction of $V$ to $G_0$ splits as a tensor 
product of $2$-dimensional representations, and $G_0$ is isomorphic to a 
subgroup of $\SL_2(p)\circ\SL_2(p)$. By Lemma \ref{l:str.emb.}(b), the 
image of $G_0$ in $\PSL_2(p)\times\PSL_2(p)$ has a strongly $p$-embedded 
subgroup. But this contradicts Lemma \ref{l:G<prod}(a), applied with $K_i$ 
the kernels of the two projections to $\PSL_2(p)$. 

The class $\scrc_6$ consists of the normalizers of $K\cong2^{1+4}_{\pm}$ 
(if $p\equiv3$ (mod $4$)), or that of $K\cong C_4\circ2^{1+4}$ (if 
$p\equiv1$ (mod $4$)). Thus $\Out(K)\cong\Sigma_3\wr C_2$, $\Sigma_5$, or 
$\Sigma_6$, respectively. If $k=6$, then since $p^2\mid|G|$, we have $p=3$ and 
$K\cong2^{1+4}_+$, so $G_0$ is a subgroup of $\SL_2(3)\circ\SL_2(3)$, 
and $G$ is contained in a member of $\scrc_7$. 

Assume $k=8$. The class $\scrc_8$ consists of the normalizers of 
$\Sp_4(p)$, $\Omega_4^+(p)\cong\SL_2(p)\circ\SL_2(p)$, and 
$\Omega_4^-(p)\cong\PSL_2(p^2)$. The symplectic group $\Sp_4(p)$ is 
generated by the two parabolic subgroups that contain $S$, each of which 
would be contained in a strongly $p$-embedded subgroup if there were one. 
So $G\ncong\Sp_4(p)$, and the proper subgroups of this group are 
eliminated by again applying Aschbacher's theorem using similar 
arguments. The subgroup $\SO_4^+(p)$ is in class $\scrc_7$. This leaves 
the case $G_0\le\Omega_4^-(p)\cong\PSL_2(p^2)$ (see \cite[Th\'eor\`eme 
5.21]{Artin} or \cite[Corollary 12.43]{Taylor}), with equality since 
$\PSL_2(p^2)$ is generated by any two of its Sylow $p$-subgroups.

\smallskip

\noindent\textbf{Case 2: } It remains to check the cases where the image in 
$\PGL_4(p)$ of $G$ is almost simple, and show that none of them (aside from 
those already listed) have strongly $p$-embedded subgroups. By Tables 8.9 
and 8.13 in \cite{BHR}, the only almost simple groups that could appear in 
this way as \emph{maximal} subgroups of $\SL_4(p)$ are normalizers of 
$L_2(7)$ or $A_7$ (if $p\equiv1,2,4$ (mod $7$)), or $U_4(2)$ (if $p\equiv1$ 
(mod $6$)) in $L_4(p)$, or $A_6$, $A_7$ (if $p=7$), $L_2(p)$ (if $p>7$) in 
$\Sp_4(p)$. None of these subgroups can occur when $p=3$, which is the only 
odd prime whose square can divide the order of the subgroup, so they and 
their subgroups do not come under consideration. 

The tables in \cite{BHR} were made using the classification of finite 
simple groups. But lists of maximal subgroups of $\PSL_4(q)$ and 
$\PSp_4(q)$ for odd $q$, compiled independently of the classification, 
had already appeared in \cite{Mitchell-Sp4} for the symplectic case, 
and in \cite[Chapter VII]{Blichfeldt} and the main theorems in 
\cite{SZ,Suprunenko} for the linear case. 

\smallskip

\noindent\textbf{Elements of order $p$: } The description of the Jordan 
blocks for the natural action of $\SL_2(p^2)$ is clear. So assume $V$ is 
the natural module for $G_0=\Omega_4^-(p)\cong\PSL_2(p^2)$. The 
isomorphism extends to an isomorphism $\GO_4^-(p)\cong\PGGL_2(p^2)$ between 
automorphism groups, so all elements of order $p$ in $G_0$ have similar 
actions on $V$. Hence it suffices to describe the action of one element $t$ of 
order $p$ in $\Omega_3(p)\le\Omega_4^-(p)$. The action of $\Omega_3(p)$ on 
$\F_p^3$ is induced by the conjugation action of $\PSL_2(p)$ on 
the additive group $M_2^0(\F_p)$ of $(2\times2)$-matrices of trace $0$ 
(see, e.g., \cite[Proposition A.5]{LO}), and using this one easily checks 
that $t$ acts with one Jordan block of length $3$.
\end{proof}


\end{document}